\documentclass{amsart}

\usepackage{amssymb}
\usepackage{mathtools}
\usepackage{stmaryrd}
\usepackage{mathrsfs}
\usepackage{color}

\usepackage[all]{xy}

\numberwithin{equation}{section}

\newtheorem{theo}{Theorem} 

\newtheorem{lemma}{Lemma}[section]
\newtheorem{prop}[lemma]{Proposition}

\newtheorem{theoint}[lemma]{Theorem}
\newtheorem{claim}[lemma]{Claim}

\theoremstyle{remark}
\newtheorem{remark}[lemma]{Remark}

\theoremstyle{definition}

\newcommand{\ds}{\displaystyle}

\newcommand{\hdot}{\dot{H}^1}

\newcommand{\NN}{\mathbb{N}}
\newcommand{\RR}{\mathbb{R}}
\newcommand{\ZZ}{\mathbb{Z}}

\newcommand{\eps}{\varepsilon}

\newcommand{\tu}{\tilde{u}}
\newcommand{\tv}{\tilde{v}}
\newcommand{\tT}{\widetilde{T}}

\newcommand{\teps}{\tilde{\eps}}

\newcommand{\DDD}{\mathcal{D}}

\newcommand{\HHH}{\mathcal{H}}
\newcommand{\III}{\mathcal{I}}

\newcommand{\JJJ}{\mathcal{J}}

\newcommand{\Ebf}{\mathbf{E}}
\newcommand{\Hbf}{\mathbf{H}}

\newcommand{\tW}{\widetilde{W}}

\newcommand{\indic}{1\!\!1}

\DeclareMathOperator{\loc}{loc}
\DeclareMathOperator{\vect}{span}

\DeclareMathOperator{\supp}{supp}
\DeclareMathOperator{\rad}{rad}
\DeclareMathOperator{\inn}{in}
\DeclareMathOperator{\out}{out}
\DeclareMathOperator{\sinc}{sinc}

\title[Soliton resolution for critical wave equations]{Soliton resolution for critical co-rotational wave maps and radial cubic wave equation}

\author[T.~Duyckaerts]{Thomas Duyckaerts$^1$}
\author[C.~Kenig]{Carlos Kenig$^2$}
\author[Y.~Martel]{Yvan Martel$^3$}
\author[F.~Merle]{Frank Merle$^4$}
\thanks{$^1$LAGA (UMR 7539), Universit\'e Sorbonne Paris Nord, and Institut Universitaire de France}
\thanks{$^2$University of Chicago. Partially supported by NSF Grants DMS-14363746 and DMS-1800082}
\thanks{$^3$CMLS  (UMR 7640), Ecole Polytechnique}
\thanks{$^4$AGM (UMR 8088), CY Cergy Paris Universit\'e and IHES}

\keywords{Wave maps, focusing wave equation, dynamics, soliton resolution, global solutions, blow-up}

\begin{document}
\begin{abstract}
In this paper we prove the soliton resolution conjecture for all times, for all solutions in the energy space, of the co-rotational wave map equation. To our knowledge this is the first such result for all initial data in the energy space for a wave-type equation. We also prove the corresponding results for radial solutions, which remain bounded in the energy norm, of the cubic (energy-critical) nonlinear wave equation in space dimension 4.
\end{abstract}

\maketitle
\tableofcontents
\section{Introduction}
In this paper we first consider co-rotational wave maps, from Minkowski space into the two-sphere, which in spherical coordinates for the two-sphere, corresponds to solutions of the following equation:
\begin{equation}
 \label{WM}
 \partial_t^2\psi-\partial_r^2\psi-\frac{1}{r}\partial_r\psi+\frac{\sin(2\psi)}{2r^2}=0,
\end{equation} 
where $r>0$ and $t\in \RR$, with initial data
\begin{equation}
 \label{IDWM}
 \vec{\psi}_{\restriction t=0}=(\psi_0,\psi_1)\in \Ebf,
\end{equation} 
where $\vec{\psi}=(\psi,\partial_t\psi)$ (see for example \cite{ShTZ94} and the introduction of \cite{CoKeLaSc15a}).  Here $\Ebf$ is the space of initial data such that the conserved energy:
$$E_{M}(\psi_0,\psi_1)=\frac{1}{2}\int_0^{\infty} (\psi_1)^2rdr+\frac{1}{2}\int_0^{\infty} (\partial_r\psi_0)^2 rdr+\frac{1}{2}\int_0^{\infty}\frac{\sin^2\psi_0}{r^2}rdr$$
is finite. 

If $(\psi_0,\psi_1)$ in $\Ebf$, there exists $(\ell,m)\in \ZZ^2$ such that
\begin{equation}
 \label{lim_lm}
\lim_{r\to 0} \psi_0(r)=\ell\pi,\quad \lim_{r\to\infty}\psi_0(r)=m\pi.
 \end{equation} 
We denote by $\Hbf_{\ell,m}$ the set of $(\psi_0,\psi_1)$ in $\Ebf$ such that \eqref{lim_lm} holds. This is an affine space, parallel to the Hilbert space $\Hbf=\Hbf_{0,0}$ of functions $(\psi_0,\psi_1)\in \left(L^2_{\loc}\big((0,\infty)\big)\right)^2$ such that 
\begin{equation}
\label{normH}
\|(\psi_0,\psi_1)\|^2_{\Hbf}=\int_0^{\infty} \left((\partial_r\psi_0)^2+\frac{1}{r^2}(\psi_0(r))^2\right)rdr+\int_0^{\infty} \psi_1^2rdr <\infty.
\end{equation}
It is well known (\cite{ShTZ94}) that \eqref{WM} is locally well-posed in $\Ebf$ and that the spaces $\Hbf_{\ell,m}$ are stable by the flow (see for instance \cite{CoKeLaSc15a}). The equation \eqref{WM} has the following scaling invariance: if $\psi$ is a solution of \eqref{WM}, then $\psi(\lambda\,\cdot,\lambda\,\cdot)$ is also a solution of \eqref{WM}, with same energy. 

The bubble
\begin{equation}
 \label{defQ}
 Q(r)=2\arctan r
\end{equation} 
is a stationary solution of \eqref{WM} such that $(Q,0)\in \Hbf_{0,1}$. Other stationary solutions of \eqref{WM} are given by $\ell \pi\pm Q(\lambda\,\cdot)$, $\lambda>0$. These are the only finite energy, stationary solutions of \eqref{WM}. 

The stationary solution $Q$ has several important properties. It
is stable up to the symmetries (including scaling) as a solution of \eqref{WM}, but also as a general wave map from the Minkowski space into the two-sphere (this is, for example, a simple consequence of Theorem 6.1 in \cite{DuJiKeMe17b} and the conservation laws for general wave maps). In addition, $Q$ is up to symmetries the static solution of least energy among all the (general) static wave maps from Minkowski space into the two-sphere. Thus, $Q$ is the ``ground state''.

The equation \eqref{WM} is the case $k=1$ of the equation obtained by considering $k$-equivariant wave maps from Minkowski space into the two-sphere:
\begin{equation}
 \label{WMk}
 \partial_t^2\psi-\partial_r^2\psi-\frac{1}{r}\partial_r\psi+k^2\frac{\sin(2\psi)}{2r^2}=0,
 \end{equation}
 The co-rotational case $k=1$ is distinguished among the $k$-equivariant ones, because it is the one satisfied by the ground state $Q$ (of the whole wave map equation).
  
 The linearized equation of \eqref{WMk} at $\psi=0$ is 
$$\partial_t^2\psi_L-\partial_r^2\psi_L-\frac{1}{r}\partial_r\psi_L+\frac{k^2}{r^2}\psi_L=0,$$
which can be reduced to the radial wave equation in dimension $2k+2$ by the change of unknown function $r^{k}u_L=\psi_L$. The linearized equation for \eqref{WM} is thus essentially the $4D$ wave equation.

In this paper we will also consider a twin problem, the wave equation on $\RR^4$, with the energy-critical focusing nonlinearity:
\begin{equation}
 \label{NLW}
 \partial_t^2u-\Delta u=u^3,
\end{equation} 
where $t\in \RR$ and $x\in \RR^4$, 
with initial data
\begin{equation}
 \label{ID}
 \vec{u}_{\restriction t=0}=(u_0,u_1)\in \HHH,
\end{equation} 
where $\vec{u}=(u,\partial_t u)$, and $\HHH=\hdot(\RR^4)\times L^2(\RR^4)$. We will only consider radial initial data, i.e. data depending only on $r=|x|=\sqrt{x_1^2+x_2^2+x_3^2+x_4^2}$.

We denote by
$$W(x)=\frac{1}{1+\frac{|x|^2}{8}}$$
the ground state of \eqref{NLW}. 

The equation \eqref{NLW} is a special case of the energy-critical wave equation 
\begin{equation}
\label{NLWN}
\partial_t^2u-\Delta u=|u|^{\frac{4}{N-2}}u 
\end{equation} 
in general space dimension $N\geq 3$. In this general case, the ground state stationary solution is given by $W(x)=\left(1+\frac{|x|^2}{N(N-2)}\right)^{1-\frac{N}{2}}$.

\subsection{Background on the soliton resolution conjecture}

\emph{The main results of this paper are the proofs of soliton resolution, without size constraints, and for all times, for solutions of \eqref{WM} and for radial solutions of \eqref{NLW}. }In particular, we will obtain the complete classification of the solutions of \eqref{WM}, which is to our knowledge the  first result of soliton resolution \emph{for all initial data in the energy space} for a wave-type equation. In this work we only consider co-rotational wave maps and radial solutions of \eqref{NLW}. The general problems, without any symmetry assumption, seem
out of reach by current methods.

To put these results in perspective, we start with a general discussion of the soliton resolution conjecture for nonlinear dispersive equations. This conjecture predicts that any global in time solution of this type of equation evolves asymptotically as a sum of decoupled solitons (traveling wave solutions, which are well-localized and traveling at a fixed speed), a radiative term (typically a solution to a linear equation) and a term going to zero in the energy space. For finite time blow-up solutions, a similar decomposition should hold, depending on the nature of the blow-up.

This conjecture arose in the 1970's from numerical simulations and the theory of integrable equations, in order to explain a ``puzzling paradox'' 
stemming from the birth of scientific computation, in a numerical simulation carried out at Los Alamos in the early 1950's (see \cite{FermiPastaUlam55}) by Fermi, Pasta and Ulam. In 1965 Kruskal  found that as the spacial mesh in the discretization of the Fermi-Pasta-Ulam model tends to $0$, the solutions of the Fermi-Pasta-Ulam problem converge to solutions of the KdV equation (see \cite{Kruskal78} for an account). The fact that the KdV equation has soliton solutions provided an explanation for the ``puzzling paradox''. Following this discovery, Zabusky and Kruskal \cite{ZabuskyKruskal65} conducted another influential numerical simulation, which showed numerically the emergence of solitons and multisolitons in the KdV equation. This experiment led to the soliton resolution conjecture, and to the theory of complete integrability, to explain the inelastic collision of solitons that was observed.

The first theoretical results in the direction of soliton resolution were obtained for the completely integrable KdV, mKdV and $1$-dimensional cubic NLS, using the method of inverse scattering (\cite{Lax68}, \cite{EcSc83},\cite{Eckhaus86},  \cite{Schuur86BO},  \cite{SegurAblowitz76}, \cite{Novoksenov80}, \cite{BoJeMcL18}). Very few complete results exist for equations that are not completely integrable. For non-integrable nonlinear dispersive equations, such as nonlinear Schr\"odinger and Klein-Gordon equations, complete results seem out of reach. Known results include scattering below a threshold given by the ground state of the equation (see e.g. \cite{KeMe06}, \cite{DuHoRo08}, \cite{IbMaNa11}, \cite{Dodson15}), local study close to the ground state soliton \cite{NaSc11Bo}), and in some special cases the existence of a global compact attractor (e.g. \cite{Tao07DPDE}). We refer to the introduction of \cite{DuJiKeMe17a} for a more complete discussion and more references on the subject.

\emph{The soliton resolution conjecture is believed to hold unconditionally for energy-critical wave maps into the two-sphere, and for energy-bounded solutions of the energy-critical nonlinear wave equation. }For the case of wave maps, a first result in this direction was that any solution that blows up in finite time converges locally in space, up to the symmetries of the equation, for a sequence of times, to a traveling wave solution (see \cite{ChTZ93}, \cite{Struwe03b} for the equivariant case, 
\cite{StTa10a}, \cite{StTa10b} for the general case). Similar results hold for solutions that exist for infinite time but don't scatter (see  \cite{StTa10b}). In the equivariant case, using techniques developed by the first, second and fourth authors, one can prove stronger statements, namely that the soliton resolution holds in the equivariant case with a condition on the energy ruling out a multisoliton configuration \cite{CoKeLaSc15a}, \cite{CoKeLaSc15b}) and that it holds for a sequence of times without this condition in the co-rotational case \cite{Cote15} and for the $k=2$ equivariant case \cite{JiaKenig17} (and for
all equivariant cases, modulo a  technical condition regarding the local well-posedness theory, also in \cite{JiaKenig17}). The limiting case of a pure two-soliton was treated in \cite{JendrejLawrie18}), where it is shown, in particular, that the collision between the two solitons is inelastic. Subsequent work on the case of two solitons is in \cite{JendrejLawrie20Pa}, \cite{JendrejLawrie20Pb} and \cite{Rodriguez21}, while in \cite{JendrejLawrie20Pc}, the full resolution is proved in the equivariant case, in the two-soliton setting. For wave maps into the two-sphere without any symmetry assumption, the resolution is known only close to the ground state, in work of the first, second and fourth authors, with Hao Jia \cite{DuJiKeMe17b} (see also \cite{Grinis17} for a weaker version of the decomposition, without size constraint, for a sequence of times).

For the parabolic analog of wave maps (the harmonic map heat flow) these questions have been studied earlier (see e.g. \cite{Struwe85},
\cite{Qing95}, \cite{QingTian97}, \cite{Topping97}). In full generality, in the parabolic case, the soliton resolution is only known for a sequence of times. An example of Topping \cite{Topping97} shows that for a general target manifold, two sequences of times may lead to different decompositions. In the case when this target is the two-sphere, it is conjectured that the analog of the soliton resolution holds, but this is still open. For a recent work in the parabolic analog of \eqref{NLWN}, in dimension greater than or equal to $7$, without radial symmetry, near the ground state, see \cite{CollotMerleRaphael17}. 

All the results in non-integrable cases and in the absence of size constraints, mentioned in the last two paragraphs, are proven using monotonicity laws after time averaging, and hence hold only for well-chosen sequences of times. The example of Topping just mentioned shows the necessity of this. Thus, in order to obtain results that hold for all times, new methods that go beyond monotonic time averages have to be created, which address the notoriously difficult issue of time oscillations. This, for the case of \eqref{WM} and the radial case of \eqref{NLW}, is a key contribution of this work.

For the case of the energy-critical nonlinear wave equation \eqref{NLWN}, classification results for solutions ``below the ground state'' were obtained in \cite{KeMe08}, \cite{DuMe08}. (Corresponding classification results for co-rotational wave maps are in \cite{CoteKenigMerle08} and in \cite{CoKeLaSc15a}, \cite{CoKeLaSc15b}). These results are inspired by earlier results for the nonlinear Schr\"odinger equation going back to \cite{Merle93}. (See also the results for gKdV in \cite{MartelMerleRaphael15}).

Decomposition results for \eqref{NLWN} in the $3D$ radial case, near the ground state, are in \cite{DuKeMe11a}, and corresponding results in the non-radial case,  near the ground state are in \cite{DuKeMe12}. Constructions near the ground state of center stable manifold, in the radial case and non-radial cases, are in \cite{KrNaSc13b}, \cite{KrNaSc15}. The soliton resolution for sequences of time in the radial case, for solutions which are bounded in the energy norm, was proved by \cite{DuKeMe12b} in $3$ dimensions, \cite{Rodriguez16} in all other odd dimensions, in \cite{CoKeLaSc18} in $4$ dimensions and in \cite{JiaKenig17} in $6$ dimensions. In \cite{DuJiKeMe17a}, the first, second and fourth authors, with Hao Jia, proved the decomposition for sequences of time, in the nonradial case, for solutions which are bounded in the energy norm, in dimensions $3$, $4$ and $5$. In the radial case, $W$, defined above, is the only soliton up to sign change and scaling. The full resolution (for all times), for solutions bounded in the energy norm, for $N=3$, in the radial case, was proved in \cite{DuKeMe13}, by the first, second and fourth authors. The key fact in the proof is a rigidity theorem giving a dynamical characterization of the static solutions, in terms of outer energy lower bounds. The proof of this fact used techniques that are very specific to radial solutions in $3$ space dimensions. Recently, combining the papers \cite{DuKeMe19Pc}, \cite{DuKeMe19Pa} and \cite{DuKeMe19Pb}, the first, second and fourth authors were able to prove the soliton resolution for solutions of \eqref{NLWN} that remain bounded in energy norm, in the radial case, for all odd dimensions. An important tool in this proof is the odd dimensional linear outer energy inequalities proved in \cite{KeLaLiSc15}, that hold up to certain finite dimensional subspaces with increasing dimension as a function of $N$. When $N=3$, the dimension is $1$ and one can use scaling to deal with this exceptional subspace. For higher odd dimensions, the dimension is larger than $1$ and this was a stumbling block for many years. The proofs mentioned above involve modulation analysis to analyse the collisions of solitons, showing that they are inelastic and hence pure multisolitons in both time directions are ruled out.

\subsection{Background on rigidity theorems for dispersive equations}

As seen in many recent works, \emph{the proof of rigidity (also called Liouville) theorems, classifying 
solutions that are non-dispersive (in a sense to be specified) is crucial in the understanding of the asymptotic dynamics of semilinear dispersive equations as \eqref{WMk} and \eqref{NLWN}.} A typical statement is that the only non-dispersive solutions are the stationary solutions (or more generally the solitons) of the equation.
 
A first notion of non-dispersive solutions is given by \emph{solutions with the compactness property}, that are solutions whose trajectory is precompact up the modulations of the equation. These solutions are also sometimes called quasi-periodic or almost periodic. This concept goes back to (at least) the work of the third and fourth authors \cite{MaMe00}, in the context of the KdV equation. See also   \cite{KeMe06} for the energy-critical NLS, \cite{Keraani06}, \cite{TaViZh08} for mass-critical NLS. We refer to \cite{Tao07DPDE} and \cite{DuKeMe15a} for works highlighting the importance of solutions with the compactness property in the asymptotic study of bounded solutions of nonlinear dispersive equations.

For equation \eqref{NLWN}, solutions with the compactness property were first considered in \cite{KeMe08}, where a rigidity theorem with a size constraint is proved. A general rigidity theorem, without a size constraint (but with an additional nondegeneracy assumption for solutions without symmetry) is proved in \cite{DuKeMe16a} (see also \cite[Theorem 2]{DuKeMe11a} for the radial, $3D$ case).

Nevertheless, this notion is not useful to understand the collision of solitons. Indeed a pure multisoliton in both time directions, that is a solution that is, asymptotically as $t\to+\infty$ and as $t\to-\infty$, a sum of decoupled solitons, is not a solution with the compactness property, but it should be definitively considered as a non-dispersive solution.

It is believed that for completely integrable equations, the collision of solitons is always elastic, but that, for non-integrable equations such as \eqref{WMk}  and \eqref{NLWN}, it is not elastic and should always generate some radiation (see e.g. \cite{MartelMerle11} in the context of generalized Korteweg-de Vries equations).

To deal with this problem, the first, second and fourth authors have introduced the concept of \emph{non-radiative} solutions of \eqref{NLWN}. By definition, these are solutions of \eqref{NLWN}, defined for $|x|>|t|$, and such that 
\begin{equation}
\label{non-radiative}
\sum_{\pm} \lim_{t\to\pm\infty}  \int_{|x|>|t|} \left(|\nabla u(t,x)|^2+(\partial_tu(t,x))^2\right)dx=0.
\end{equation}
(see \cite{DuKeMe13} where this concept is used without formal definition, and \cite{DuKeMe19Pa} where the term ``non-radiative'' is introduced). This definition can be easily adapted to \eqref{WMk}.

Note that solitary waves (which always travel at speed $\ell<1$ \cite{DuKeMe15a}) are non-radiative. As shown in \cite{DuKeMe13} using the profile decomposition of \cite{BaGe99}, a rigidity theorem stating that all non-radiative solutions are solitons (or in a radial context, stationary solutions) essentially implies the soliton resolution. 

The usefulness of this concept is that, using finite speed of propagation, it can be applied by first studying solutions in the exterior of a wave cone $\{|x|>R+|t|\}$, for large $R$, thus restricting to small solutions, that are close to solutions of the linear wave equation. In connection with this, the study of lower bounds of the form
\begin{equation}
 \label{lower_bound}
 C\sum_{\pm}\lim_{t\to\pm\infty} \int_{|x|>R+|t|} |\nabla_{t,x} u_L(t,x)|^2dx\geq \int_{|x|>R} (u_1(x))^2+|\nabla u_0(x)|^2dx
\end{equation} 
for radial solutions of the linear wave equation
\begin{equation}
 \label{LWintro}
 \partial_t^2u_L-\Delta u_L=0, \quad (t,x)\in \RR\times \RR^N,\quad |x|>R+|t|
\end{equation} 
with initial data $\vec{u}_{\restriction t=0}=(u_0,u_1)$ is crucial to prove the rigidity theorem for the equations \eqref{NLWN} and \eqref{WMk}. The linear estimate  \eqref{lower_bound} depends strongly on the dimension $N$. In the case $R=0$ it was proved that for $N$ odd, \eqref{lower_bound} holds for any $(u_0,u_1)\in \HHH$ (see \cite{DuKeMe12}), whereas when $N$ is even, it is not valid in full generality, but it holds (at least in the radial case) for initial data of the form $(u_0,0)$ or $(0,u_1)$ depending on the congruence of $N$ modulo $4$ (see \cite{CoKeSc14}). In the case $R>0$,  explicit counterexamples to \eqref{lower_bound} exist and one can hope to prove \eqref{LWintro} only on a strict subspace of the energy space $\HHH$.
For $N=3$, \eqref{lower_bound} is valid for all radial initial data $(u_0,u_1) \in \HHH(\{r>R\})$, that are orthogonal to $\left( \frac{1}{r},0 \right)$. This single nondegenerate direction can be handled with the scaling of the equation, and corresponds to the stationary solution $W=\left( 1+\frac{|x|^2}{3} \right)^{-1/2}$. This leads to the proof of a strong rigidity theorem, as was mentioned earlier, and the soliton resolution for all radial solutions of \eqref{NLWN} with $N=3$ that are bounded in the energy space $\HHH$ \cite{DuKeMe13}.

For $N$ odd, $N\geq 5$, \eqref{lower_bound} holds in the radial case, for all radial data in an $\frac{N-1}{2}$ co-dimensional subspace of $\HHH(\{r>R\})$, which does not seem sufficient to deduce a strong rigidity result  for \eqref{NLWN} as in space dimension $3$,   using the scaling invariance of the equation. However, as mentioned earlier, combining asymptotic estimates on non-radiative solutions  of \eqref{NLW} deduced from \eqref{lower_bound} with a careful study of the modulation equations close to a multisoliton, the soliton resolution for all radial solutions of \eqref{NLW} that are bounded in the energy space, for all times, was proved by three of the authors in \cite{DuKeMe19Pa,DuKeMe19Pb,DuKeMe19Pc}.

In even space dimension, up to now, no lower bound of the form \eqref{lower_bound} has been known. Note that this is relevant for equivariant wave maps (verifying \eqref{WMk}), for which the underlying space dimension, after linearization, is $2k+2$.

\subsection{Main results and ideas of proofs}
In this article, we consider co-rotational wave maps (equation \eqref{WM}), for which the underlying space dimension is $4$, and \eqref{NLW}. Besides the difficulty arising from linear estimates in even space dimensions, these are also special issues at the nonlinear level, since the modulation equations close to a multisoliton degenerate, and this makes the method used in \cite{DuKeMe19Pc} for large odd dimensions ineffective.

Our first step is to prove that \eqref{lower_bound} is valid for $R>0$, when $N=4$, for radial solutions of \eqref{LWintro}, with initial data of the form $(u_0,0)$, where $u_0\in \dot{H}^1(\{r>R\})$ is radial and orthogonal to $\frac{1}{r^2}$. The proof of this is based on the result of \cite{CoKeSc14} and involved (but elementary) explicit computations, using a different method of proof than in the corresponding results in \cite{KeLaLiSc15}.

When $N=4$, the inequality \eqref{lower_bound} is not valid for radial solutions of \eqref{LWintro} with data of the form $(0,u_1)$, $u_1\in L^2(\RR^4)$, even if $u_1$ is taken in a finite co-dimensional subspace of $L^2$. This is essentially due to the resonant solution $\frac{t}{r^2}$, whose initial data barely fails to be in the energy space. In Subsection \ref{SS:approx} of this paper, we will construct an approximate solution for each of the equations \eqref{WM} and \eqref{NLW}, which is odd in time and whose time derivative is non-radiative. The initial data for this solution again barely fail to be in the energy space, but the solution itself satisfies the same global $L^pL^q$ estimates in the exterior of wave cones, as the global, finite energy solutions of \eqref{WM} or \eqref{NLW}. \emph{Although this solution has similar properties as the linear solution $\frac{t}{r^2}$, it is not a perturbation of this solution, but a truly nonlinear object, which comes from a cancellation between the linear and nonlinear parts and depends on the specific form of the nonlinearity in equations \eqref{WM} and \eqref{NLW}.}

Using the new estimate \eqref{lower_bound} for the linear problem, and the approximate non-radiative solution, we 
obtain in this paper strong rigidity theorems for the equations \eqref{WM} and (radial) \eqref{NLW}, which are identical to the one proved in \cite{DuKeMe13} for radial solutions of \eqref{NLWN} in dimension $N=3$. The proofs, which are much deeper than the one in \cite{DuKeMe13}, are first based on the decoupling of a non-radiative solution of \eqref{WM} or (radial) \eqref{NLW} into its odd and even parts (in the time variable), which satisfy similar approximate equations, with a cubic interaction term. Using this decoupling, we consider two cases. 
\begin{itemize}
 \item 
If the non-radiative solution has constant sign, one easily sees that the even part is dominant. Using that \eqref{lower_bound} holds for initial data of the form $(u_0,0)$, where $u_0$ is taken on a codimension $1$ subspace of $\dot{H}^1(\{r>R\})$, we prove that the solution is equal, up to the invariances of the equation, to one of the nonzero explicit stationary solutions $\pm Q(\lambda \,\cdot)+p\pi$ for \eqref{WM} or $\pm \lambda W(\lambda\, \cdot)$ for \eqref{NLW} for large $r$. We then extend the result for all $r>0$. This strategy can be seen as a dispersive analog of Alexandrov's method of moving plane for elliptic equations.

\item If, on the other hand, $u$ vanishes at some point $(t_0,r_0)$, we first prove using \eqref{lower_bound}  and some kind of maximum principle for non-radiative solutions, that $u(t_0,r)=0$ for $r> r_0$.  
Thus $u$ is after time translation odd in time in the exterior of the wave cone $\{r>|t|+r_0\}$.  In this case we prove by  contradiction that $\partial_tu(t_0,r)=0$ for $r> r_0$.
If not, using \eqref{lower_bound} on the time derivative of the solution, we prove that $u$ is asymptotically close to the non-radiative nonlinear solution mentioned above, which does not have finite energy, giving a contradiction. Extending the result to $r>0$ as in the previous step, we conclude that $u(t,r)$ is the zero solution.
\end{itemize} 

Precisely, we obtain the following two rigidity theorems for equations \eqref{WM} and (radial) \eqref{NLW}.

\begin{theo}[Rigidity for co-rotational wave maps]
 \label{T:rigidityWM}
Let $(\psi_0,\psi_1)\in \Ebf$ (and thus to $\Hbf_{\ell,m}$ for some $(\ell,m)\in \ZZ^2$). Assume that $(\psi_0,\psi_1)$ is not the initial data of a stationary solution of \eqref{WM}, i.e. if $\ell=m$, $(\psi_0,\psi_1)\neq (m\pi,0)$ and otherwise, for all $\lambda>0$, $(\psi_0-m\pi,\psi_1)$ is not equal to $\pm(\pi- Q(\lambda \,\cdot),0)$. Then there exists $\eta>0$ such that the following holds for all $t>0$ or for all $t<0$:
\begin{equation}
 \label{lower_boundWM}
 \int_{|t|}^{\infty} \left((\partial_t\psi(t,r))^2+(\partial_r\psi(t,r))^2+\frac{1}{r^2}\sin^2(\psi(t,r))\right)rdr\geq \eta.
\end{equation} 
\end{theo}
We refer to \S \ref{SS:outside} for the definition of solutions outside a wave cone. As proved there, for any initial data $(\psi_0,\psi_1)\in \Ebf$, the solution $\psi$ with initial data $(\psi_0,\psi_1)$ is always well-defined for $r>|t|$.

\begin{theo}[Rigidity for $4D$ radial critical wave]
 \label{T:rigidity}
 Let $(u_0,u_1)\in \HHH$, radial. Assume that $(u_0,u_1)$ is not the initial data of a radial stationary solution of \eqref{NLW}, i.e. that $\forall \lambda \in \RR$, $(u_0,u_1)$ is not equal to $(\lambda W(\lambda \,\cdot),0)$. Then there exists $R>0$ such that the solution $u$ with initial data $(u_0,u_1)$ is defined on $\{|x|>R+|t|\}$ and there exists $\eta>0$ such that the following holds for all $t>0$ or for all $t<0$:
 \begin{equation}
  \label{channel}
  \int_{|x|>R+|t|}(\partial_tu(t,x))^2+|\nabla_xu(t,x)|^2 dx\geq \eta.
 \end{equation} 

 \end{theo}
Note that Theorems \ref{T:rigidityWM} and \ref{T:rigidity} imply the fact that the collision of two or more solitons yields dispersion, and thus that there is no pure multisoliton solution of equation \eqref{WM} and of equation \eqref{NLW} in the radial case. This was previously known only in the case of equation \eqref{WM}, for two-solitons (see \cite{Rodriguez21}, and also \cite{JendrejLawrie18} for two-soliton  solutions of equation \eqref{WMk} with $k\geq 2$).

As a consequence of the two rigidity theorems, we obtain the soliton resolution for both equations. We first state the result for co-rotational wave maps.

If $f(t)$ and $g(t)$ are positive functions defined in a neighborhood of $a\in \RR\cup\{\pm\infty\}$, we will write $f(t)\ll g(t)$ as $t\to a$ when $\lim_{t\to a} f(t)/g(t)=0$.

\begin{theo}[Soliton resolution for co-rotational wave maps]
\label{T:mainWM}
Consider any solution $\psi$ in the energy space of \eqref{WM} with maximal time of existence $T_+$. Then one of the following holds:

 \textbf{Blow-up.} $T_+<\infty$, and there exist $(\phi_0,\phi_1)\in \Ebf$, an integer $J\in \NN\setminus\{0\}$, and for each $j\in \{1,\ldots,J\}$, a sign $\iota_j\in \{\pm 1\}$, and a positive function $\lambda_j(t)$ defined for $t$ close to $T_+$ such that
\begin{gather}  
\label{hyp_lambda_bup} 
 \lambda_1(t)\ll \lambda_2(t)\ll \ldots \ll\lambda_{J}(t)\ll T_+-t\text{ as }t\to T_+\\
\label{expansion_psi_bup}
\left\|(\psi(t),\partial_t\psi(t))-\left(\phi_{0}+\sum_{j=1}^J\iota_jQ\left(\frac{\cdot }{\lambda_{j}(t)}\right),\phi_{1}\right)\right\|_{\Hbf}\underset{t\to T_+}{\longrightarrow} 0.
\end{gather}
\textbf{Global solution.} $T_+=+\infty$ and there exist $\ell\in \ZZ$, a solution $\psi_{L}$ of 
$$\partial_t^2\psi_L-\partial_r^2\psi_L-\frac{1}{r}\partial_r\psi_L-\frac{1}{r^2}\psi_L=0,$$
with initial data in $\Hbf$,
an integer $J\in \NN$, and for each $j\in \{1,\ldots,J\}$, a sign $\iota_j\in \{\pm 1\}$, and a positive function $\lambda_j(t)$ defined for large $t$ such that 
\begin{gather}
\label{hyp_lambda}
 \lambda_1(t)\ll \lambda_2(t)\ll \ldots \ll\lambda_{J}(t)\ll t\text{ as }t\to +\infty\\
 \label{expansion_psi}
\Bigg\|(\psi(t),\partial_t\psi(t))-\left(\psi_{L}(t)+\ell\pi+\sum_{j=1}^J\iota_j Q\left(\frac{\cdot }{\lambda_{j}(t)}\right),\partial_t\psi_{L}(t)\right)\Bigg\|_{\Hbf}%
\underset{t\to+\infty}{\longrightarrow} 0.
\end{gather}
 \end{theo}

 \begin{remark}
In the finite time-blow-up case, if $(\psi_0,\psi_1)\in \Hbf_{\ell,m}$, $(\phi_0,\phi_1)\in \Hbf_{\ell',m'}$, the expansion \eqref{expansion_psi_bup} yields the following constraints:
$$ \ell=\ell',\quad m=m'+\sum_{j=1}^J \iota_j.$$
 Similarly, in the global case, $(\psi_0,\psi_1)\in \Hbf_{\ell,m}$, where $\ell$ is as in \eqref{expansion_psi} and $m=\ell+\sum_{j=1}^J\iota_j$.
\end{remark}

\begin{remark}
We highlight here the fact that Theorem \ref{T:mainWM} is a soliton resolution for \emph{any finite energy initial data} of equation \eqref{WM}, which should be compared with the results for radial solutions of \eqref{NLW}, or \eqref{NLWN}, where we must assume that the solution is bounded in energy norm or global in time, or sometimes both.
\end{remark}
The strategy of the proof of the soliton resolution from the rigidity theorem is the same as in \cite{DuKeMe13}. However, since the exterior energy bound is not valid for data of the form $(0,u_1)$, we need the additional fact, proved in \cite{Cote15}, that there exists a sequence $\{t_n\}_n\to T_+$ such that $\partial_t\psi(t_n)-\phi_1$ in the finite time blow-up case  (respectively $\partial_t\psi(t_n)-\partial_t \psi_L(t_n)$ in the global case) goes to $0$ in $L^2$. Moreover, the proof uses a ``nonlinear profile decomposition'' outside wave cones which was not known earlier for equivariant wave maps (see Proposition \ref{P:approxWM}). We prove this nonlinear profile decomposition using a new space-time bound outside wave cones, involving the nonlinearity (see Lemma \ref{L:L3L6sinus}). These results should be of independent interest for further study of equivariant wave maps.
\begin{remark}
Solutions of \eqref{WM} satisfying \eqref{expansion_psi_bup} with $J=1$ are known: see  \cite{KrScTa08},
\cite{RaphaelRodnianski12} and \cite{GaoKrieger15}. See also \cite{RodnianskiSterbenz10} for solutions of \eqref{WMk} with $J=1$, $k\geq 4$. We do not know any construction of solutions of \eqref{WM} satisfying \eqref{expansion_u_bup} or \eqref{expansion}  with $J\geq 2$, known examples of solutions satisfying \eqref{expansion} with $J=1$ are the stationary solutions and the recently constructed solutions in \cite{Pillai2019P}, \cite{Pillai2020P}. For construction of solutions of \eqref{WMk} with $k\geq 3$, satisfying \eqref{expansion} with $J=2$, see \cite{Jendrej19}.
\end{remark}

\begin{theo}[Soliton resolution for radial $4D$ critical wave]
\label{T:main}
 Let $u$ be a radial solution of \eqref{NLW} and $T_+$ its maximal time of existence. Then one of the following holds:

\noindent\textbf{Type I blow-up.} $T_+<\infty$ and 
\begin{equation}
\label{BUp}
\lim_{t\to T_+} \|(u(t),\partial_tu(t)\|_{\HHH}=+\infty. 
\end{equation} 
\textbf{Type II blow-up.} $T_+<\infty$, and there exist $(v_0,v_1)\in \HHH$, an integer $J\in \NN\setminus\{0\}$, and for each $j\in \{1,\ldots,J\}$, a sign $\iota_j\in \{\pm 1\}$, and a positive function $\lambda_j(t)$ defined for $t$ close to $T_+$ 
such that \eqref{hyp_lambda_bup} holds and
\begin{gather}
\label{expansion_u_bup}
\left\|(u(t),\partial_tu(t))-\left(v_{0}+\sum_{j=1}^J\frac{\iota_j}{\lambda_j(t)}W\left(\frac{x}{\lambda_{j}(t)}\right),v_{1}\right)\right\|_{\HHH}\underset{t\to T_+}{\longrightarrow} 0.
\end{gather}
\textbf{Global solution.} $T_+=+\infty$ and there exist a solution $v_{L}$ of the linear wave equation, an integer $J\in \NN$, and for each $j\in \{1,\ldots,J\}$, a sign $\iota_j\in \{\pm 1\}$, and a positive function $\lambda_j(t)$ defined for large $t$ such that \eqref{hyp_lambda} holds and
\begin{equation}
\label{expansion}
\Bigg\|(u(t),\partial_tu(t))-\left(v_{L}(t)+\sum_{j=1}^J\frac{\iota_j}{\lambda_j(t)}W\left(\frac{x}{\lambda_{j}(t)}\right),\partial_tv_{L}(t)\right)\Bigg\|_{\HHH}%
\underset{t\to+\infty}{\longrightarrow} 0.
\end{equation}
\end{theo}
As in the wave maps case, Theorem \ref{T:main} follows from the rigidity theorem, Theorem \ref{T:rigidity}, using the strategy of \cite{DuKeMe13} and the extra-fact, proved in \cite{CoKeLaSc18}, that there exists a sequence $\{t_n\}_n\to T_+$ such that $\partial_t u(t_n)-v_1$ (in the finite time blow-up case) or $\partial_t u(t_n)-\partial_tv_L((t_n)$ (in the global case) goes to $0$ with $n$ going to infinity.
\begin{remark}
Except for the scattering solutions (\eqref{expansion} with $J=0$) and the stationary solutions $\pm \lambda W(\lambda\,\cdot)$, the only constructed radial solutions of \eqref{NLW}, which are bounded in $\HHH$, are the blow-up solutions constructed in \cite{HiRa12}, satisfying \eqref{expansion_u_bup} with $J=1$. Note however that the construction in \cite{KrNaSc15} of center-stable manifolds for \eqref{NLWN}, with $N=3$ or $N=5$, should be adaptable to the case $N=4$  in the radial case, using Theorem \ref{T:main}. This scarcity of contructions seems to be the result of the very strong interaction between solitons, due to their slow decay, and the fact that the modulation equations for \eqref{NLWN} degenerate when $N=4$, as was mentioned earlier.
\end{remark}
\begin{remark}
 The results in Theorems \ref{T:mainWM} and \ref{T:main} naturally suggest the possibility of proving analogous results for solutions of \eqref{WMk} with $k\geq 2$, for the radial Yang-Mills equation, and for radial solutions of \eqref{NLWN} in the even dimensional case, $N\geq 6$.  We expect that the ideas in this paper, together with the work in \cite{DuKeMe19Pc}, \cite{DuKeMe19Pa} and \cite{DuKeMe19Pb} will set the path to solve these problems. 
\end{remark}

%
\subsection{Outline of the paper} We conclude the introduction with a very brief outline of the paper. Section \ref{S:preliminaries} introduces notations, reviews the well-posedness theory and relevant space-times estimates (Strichartz estimates) and reviews solutions outside a wave cone. Section \ref{S:linear} is devoted to the proof of our new linear estimate \eqref{lower_bound}. Sections \ref{S:rigidityI} and \ref{S:odd} prove the rigidity theorems, with Section \ref{S:rigidityI} treating constant sign solutions and Section \ref{S:odd} treating odd solutions in time. In Section \ref{S:resolution} we prove the soliton resolution for radial solutions of the non-linear wave equation \eqref{NLW}, while in Section \ref{S:resolutionWM} we prove the necessary results on co-rotational wave maps, and then prove the soliton resolution for them. The Appendices \ref{A:radiation}, \ref{A:psdo_ortho}, \ref{A:anal_harm} collect technical results that are needed in the proofs.

\section{Preliminaries}
\label{S:preliminaries}
\subsection{Notations}
\label{SS:notations}
We will reduce co-rotational wave maps to a radial wave equation in space dimension 
$4$, and work in most of the article with functions that are defined on $\RR^4$ or $\RR_t\times \RR^4_x$ and that are radial in the space variable, i.e. depending only on $r=|x|$, $x\in \RR^4$. Unless explicitly mentioned, the notation $L^p$ will always denote the space $L^p(\RR^4)$. We normalize the $L^p$ norm as follows:
$$ \|f\|_{L^p}^p=\int_{0}^{\infty} |f(r)|^pr^3dr,\quad f\in L^p,\text{ radial.}$$
We will denote by $L^pL^q$ the space $L^p(\RR_t,L^q_x(\RR^4))$.
If $A$ is a space of functions on $\RR^4$, we will denote by $A(R)$ the space of restrictions of radial elements of $A$ to $[R,\infty)$. If $B$ is a space of functions on $\RR_t\times \RR_x^4$, we denote by $B(R)$ the space of restriction of radial elements of $B$ in the space variable $x$ to $\{(t,x)\; |x|>R+|t|\}$. We will endow this space with the usual norms, e.g.
\begin{gather*}
\|\varphi\|_{L^p(R)}^p=\int_R^{+\infty} |\varphi(r)|^pr^3dr,\quad \|\varphi\|_{\dot{H}^1(R)}^2=\int_{R}^{\infty} \left|\frac{d \varphi}{dr}(r)\right|^2r^{3}dr\\
\|f\|_{(L^pL^q)(R)}^p=\int_{\RR}\bigg( \int_{R+|t|}^{+\infty} |f(r)|^qr^{3}dr \bigg)^{\frac pq}dt. 
\end{gather*}
If $f$ is a function of space and time, we write $\nabla_{t,x}f=(\partial_tf,\nabla_xf)$, and (if $f$ is radial), $\partial_{t,r}f=(\partial_tf,\partial_rf)$.
The notation $\HHH$ denotes the energy space for the wave equation $\dot{H}^1(\RR^4)\times L^2(\RR^4)$. The notation $\Ebf$ stands for the space of finite energy data for the co-rotational wave maps equation \eqref{WM}. We recall that $\Ebf$ is the disjoint union 
$$ \Ebf=\bigcup_{\ell,m} \Hbf_{\ell,m}$$
where the affine spaces $\Hbf_{\ell,m}$ parallel to the vector space $\Hbf:=\Hbf_{0,0}$, are defined in the introduction. Note that $\Hbf=H\times (L^2((0,\infty),rdr))$, where $H$ is the space of radial $L^1_{\loc}$ functions $\psi_0$ such that:
\begin{equation}
\label{defH}
\|\psi_0\|_H^2:=\int_{0}^{\infty}  (\partial_r\psi_0)^2rdr+\int_0^{\infty}\frac{1}{r^2}\psi_0^2rdr <\infty.
\end{equation} 
As above, we will use the notations $\HHH(R)$ to denote the space of restriction  to $r>R$ of pair of radial functions in $\HHH$ and $\Hbf(R)$ for the space of restrictions to $r>R$ of pair of functions in $\Hbf$, endowed with the following norm
$$ \|(\psi_0,\psi_1)\|_{\Hbf(R)}^2=\int_{R}^{\infty} (\psi_1)^2rdr+\int_{R}^{\infty} (\partial_r\psi_0)^2rdr+\int_R^{\infty}\frac{1}{r^2}\psi_0^2rdr.$$

If $u$ is a function of space and time, $\vec{u}$ denotes the pair $(u,\partial_tu)$.

In all the article (except in Appendix \ref{A:radiation}), $\Delta=\frac{\partial^2}{\partial r^2}+\frac{3}{r^2} \frac{\partial}{\partial r}$ will denote the radial Laplace operator in dimension $4$. 
\subsection{Strichartz estimates and well-posedness}
\label{SS:Strichartz}
We first recall some results on Strichartz estimates and local well-posedness. We refer to \cite{GiSoVe92}, \cite{GiVe95}, \cite{LiSo95} and \cite{KeMe08} for the details.

Let $I$ be an interval, $t_0\in I$, $(u_0,u_1)\in \HHH$, $f\in L^1(I,L^2(\RR^4)))$. By definition, the solution of the wave equation
\begin{equation}
 \label{inhom_wave}
 \partial_t^2u-\Delta u=f
\end{equation}
with initial data
\begin{equation}
\label{IDt0}
\vec{u}_{\restriction t=t_0}=(u_0,u_1)
\end{equation}
is given by the Duhamel formulation:
$$ u(t,x)=\cos((t-t_0)\sqrt{-\Delta}) u_0+\frac{\sin((t-t_0)\sqrt{-\Delta})}{\sqrt{-\Delta}} u_1+\int_{t_0}^t \frac{\sin\left((t-s)\sqrt{-\Delta}\right)}{\sqrt{-\Delta}} f(s)ds.$$
Note that $\vec{u}\in C^0(I,\mathcal{H})$. Furthermore, $u$ satisfies the following Strichartz estimates (see \cite{St77a}, \cite{GiVe95}):
\begin{equation} 
 \label{Strichartz}
\|u\|_{L^2(I,L^8)}+\|u\|_{L^3(I,L^6)}+\sup_{t\in I} \|\vec{u}(t)\|_{\HHH}\lesssim \|(u_0,u_1)\|_{\HHH}+\|f\|_{L^1(I,L^2)}.
\end{equation} 
Radial solutions of \eqref{inhom_wave} in the case where $I=\RR$ and $f=0$ also satisfy the following dispersive property:
\begin{equation}
\label{Linfty0}
\lim_{t\to\pm \infty} \|r u(t)\|_{L^{\infty}}=0. 
\end{equation} 
This can be proved for smooth compactly supported initial data using the explicit formula for the solution. The general case can be deduced by a density argument and the radial Sobolev inequality:
$$ \|r U\|_{L^{\infty}}\lesssim \|U\|_{\dot{H}^1},\quad U\in \dot{H}^1_{\rad}(\RR^4).$$
If $(u_0,u_1)\in \HHH$, a solution $u$ of the nonlinear wave equation \eqref{NLW} on $I$ with initial data \eqref{IDt0} is a function $u\in L^3_{\loc}(I,L^6)$ which is a solution of \eqref{NLW} in the preceding Duhamel sense, on every compact subinterval of $I$. Using Strichartz estimates and a standard fixed point argument, one can prove that for any initial data $(u_0,u_1)\in  \HHH$, there exists a unique \emph{maximal} solution $u$ such that $\vec{u}\in  C^0\left((T_-(u),T_+(u)),\HHH\right)$ satisfying the blow-up criterion
\begin{multline*}
T_+<\infty\; (\text{ respectively }|T_-|<\infty)\\
 \Longrightarrow \|u\|_{L^3\left((t_0,T_+),L^6\right)}=+\infty \;(\text{respectively }\|u\|_{L^3\left((T_-,t_0),L^6\right)}=+\infty),
 \end{multline*}
where we have denoted $T_{\pm}=T_{\pm}(u)$. 

We say that a solution $u$ of \eqref{NLW} \emph{scatters} forward in time if and only if $T_+(u)=\infty$ and there exists a solution $u_L$ of the free wave equation $\partial_t^2u_L=\Delta u_L$ such that 
$$\lim_{t\to\infty} \|\vec{u}(t)-\vec{u}_L(t)\|_{\HHH}=0,$$
or equivalently, $u\in L^3((t_0,T_+),L^6)$.

\subsection{Solutions outside a wave cone}
\label{SS:outside}
We will now restrict to radial functions and define solutions of wave equations outside wave cones. Let us mention that it is possible to construct a local well-posedness theory in this context (see Subsection 2.3 of \cite{DuKeMe19Pc}) and also the space-time maximal domain of influence of a given solution (in the spirit of \cite{Alinhac95Bo}). We will not need these notions here.

To simplify notations, we will restrict to initial data at $t_0=0$. Recall from \S \ref{SS:notations} above the notations $\HHH(R)=\hdot(R)\times L^2(R)$ and $(L^pL^q)(R)$. 
If $f\in (L^1L^2)(R)$ and $(u_0,u_1)\in \HHH(R)$, we define the solution of \eqref{inhom_wave}, \eqref{ID} on $\{r>R+|t|\}$ as the restriction to $\{r>R+|t|\}$ of a solution $\tilde{u}$ of the equation 
\begin{equation*}
 \partial_t^2\tilde{u}-\Delta \tilde{u}=\tilde{f},\quad 
\vec{\tilde{u}}_{\restriction t=0}=(\tilde{u}_0,\tilde{u}_1),
\end{equation*}
where $(\tilde{u}_0,\tilde{u}_1)\in \HHH$, $\tilde{f}\in L^1L^2$, $(u_0,u_1)(r)=(\tilde{u}_0,\tilde{u}_1)(r)$ for $r>R$, and $\tilde{f}(t,r)=f(t,r)$ for $r>|t|+R$. By finite speed of propagation, $u$ does not depend on the choice of the extensions $\tilde{u}_0,\tilde{u}_1$ and $\tilde{f}$. Furthermore, \eqref{Strichartz} and finite speed of propagation implies the Strichartz estimates outside the wave cone:
\begin{multline}
 \label{Strichartz_cone}
\|u\|_{(L^2L^8)(R)}+\|u\|_{(L^3L^6)(R)}+\sup_{t\in \RR} \|\vec{u}(t)\|_{\HHH(R+|t|)}\\
\lesssim \|(u_0,u_1)\|_{\HHH(R)}+\|f\|_{(L^1L^2)(R)}.
\end{multline} 
Let $(u_0,u_1)\in \HHH(R)$. We say that $u$ is a solution of \eqref{NLW} in $\{r>R+|t|\}$ if $u\in (L^3L^6)(R)$ and $u$ is a solution of \eqref{inhom_wave}, \eqref{ID} with $f=u^3$ on $\{r>R+|t|\}$. 
Using the usual small data theory and finite speed of propagation, one sees that there exists $\delta_0$ (independent of $R$) such that if $\|(u_0,u_1)\|_{\HHH(R)}<\delta_0$, then there exists a (unique) solution $u$ on $\{|x|>R+|t|\}$ with initial data $(u_0,u_1)$. As a consequence, for any $(u_0,u_1)\in \HHH$, there exists $R\geq 0$ and a solution of \eqref{NLW} in $\{|x|>R+|t|\}$ with initial data $(u_0,u_1)$. One also has the following uniqueness properties:
\begin{itemize}
 \item If $(u_0,u_1) \in \HHH$, $R\geq 0$, and there exists a solution of \eqref{NLW} on $\{|x|>R+|t|\}$ with initial data $(u_0,u_1)$, then this solution coincides with the solution of \eqref{NLW} defined in \S \ref{SS:Strichartz} in $\{t\in (T_-(u_0,u_1),T_+(u_0,u_1)\}\cap \{|x|>R+|t|\}$. We will consider $u$ as a solution of \eqref{NLW} with initial data on the domain $\{t\in (T_-(u_0,u_1),T_+(u_0,u_1)\}\cup \{|x|>R+|t|\}$, with obvious modifications of the definitions above.
\item If $(u_0,u_1)\in \HHH$ and $0\leq R\leq R'$, and there exists a solution of \eqref{NLW} with initial data $(u_0,u_1)$ in $\{|x|>R+|t|\}$, then the restriction of this solution to $\{|x|>R'+|t|\}$ is the solution of \eqref{NLW} with initial data $(u_0,u_1)_{\restriction \{r>R'\}}$ in $\{|x|>R'+|t|\}$.
\end{itemize}
We can define the same way solutions of the co-rotational wave maps equation \eqref{WM} outside wave cones. However in this case the solution is always global in time:
\begin{claim}
 \label{Cl:outsideWM}
For any initial data $(\psi_0,\psi_1)\in \Ebf$, the solution $\psi$ of \eqref{WM} is well-defined on $\{r>|t|\}$, and satisfies $\vec{\psi}(t)\in \Hbf(|t|)+(m\pi,0)$ for all $t$, where $m\pi=\lim_{r\to\infty} \psi_0(r)$.
 \end{claim}
\begin{proof}
We assume $(\psi_0,\psi_1)\in \Hbf_{\ell,m}$.
 By the standard well-posedness theory for \eqref{WM} (see \cite{ShTZ94}), there exist $t_0>0$ and a solution of \eqref{WM} defined on $[-t_0,t_0]\times (0,\infty)$ which satisfies $\vec{\psi}(t)\in \HHH_{\ell,m}$ for all $t\in [-t_0,t_0]$. We are thus reduced to prove the existence of a solution of \eqref{WM} with initial data $\vec{\psi}(t_0)$ at $t=t_0$, defined for $r>t> t_0$, and a solution of \eqref{WM} with initial data $\vec{\psi}(-t_0)$ at $t=-t_0$, defined for $r>-t>t_0$. Translating in time, we see that the conclusion of the claim will follow from the fact that for all $R>0$ and $m\in \ZZ$, for all initial data $(\psi_0,\psi_1)\in \Hbf(R)+(m\pi,0)$, there exists a solution $\psi$ of \eqref{WM} defined for $r>R+|t|$, with initial data $(\psi_0,\psi_1)$ at $t=0$. 
 
 Let $T$ be the positive maximal time of existence of $\psi$ as a solution of \eqref{WM} in $\{r>R+|t|\}$, and assume that $T<\infty$. Let $R_1>T+R$. We have $\indic_{|T|+R<r<R_1} \frac{\sin(2\psi)}{2r^2}\in L^1((0,T),L^2)$. Thus by standard energy estimates, $\vec{\psi}(t)$ has a limit,  as $t\to T$, in $(H^1\times L^2)(\{R+T<r<R_1\})$. Taking $R_1$ large, and combining with small data theory and finite speed of propagation to obtain convergence for $r>R_1$, we obtain that $\vec{\psi}(t)$ converges in $\Hbf(R+T)+(m\pi,0)$ as $t\to T$, a contradiction. 
\end{proof}

\section{Exterior energy estimates}
\label{S:linear}
This section concerns lower bounds on the exterior energy for radial solutions of the free wave equation
\begin{equation}
\label{FW}
\left\{
\begin{aligned}
\partial_t^2u-\Delta u&=0,\quad (t,x)\in \RR\times \RR^4\\
\vec{u}_{\restriction t=0}&=(u_0,u_1)\in \HHH
\end{aligned}\right.
\end{equation}
and of its inhomogeneous counter part
\begin{equation}
\label{LWinhom}
\left\{
\begin{aligned}
\partial_t^2u-\Delta u &=f,\quad (t,x)\in \RR\times \RR^4\\
\vec{u}_{\restriction t=0} &=(u_0,u_1)\in \HHH. 
\end{aligned}
\right.
\end{equation} 
Exterior energy estimates were first etablished for radial solutions of the 3D linear wave equation in \cite{DuKeMe11a} where the following typical result was obtained (see \cite[Lemma~4.2 and proof of Corollary~4.3]{DuKeMe11a}):
\emph{for any $(u_0,v_0)\in \hdot(\RR^3)\times L^2(\RR^3)$ with radial symmetry, the solution of the 
$3$\textnormal{D} linear wave equation $u_L$ satisfies, either for all $t\geq 0$ or for all $t\leq 0$,}
\begin{equation}\label{eq:even}
\|\nabla u_L(t)\|_{L^2(|x|> t)}^2+\|\partial_t u_L(t)\|_{L^2(|x|> t)}^2 \geq \frac 12 \left( \|\nabla u_0\|_{L^2}^2 + \|u_1\|_{L^2}^2\right).
\end{equation}
This result was later extended to any odd dimension $d\geq 3$, see \cite{DuKeMe12} and references therein.
Moreover, \cite{DuKeMe11a,KeLaLiSc15} proved suitable variants of such estimates over \emph{channels},
\emph{i.e.} in space time regions of the form $|x|>t+R$, for any $R>0$:
see \cite[Corollary~1]{KeLaLiSc15} and Proposition~\ref{P:FW_R} below.

The possiblity of extending estimates of the form~\eqref{eq:even} to even dimensions $d\geq 2$ is thoroughly discussed in \cite{CoKeSc14}.
While estimate \eqref{eq:even} is proved to \emph{fail} in any even dimension for general data $(u_0,u_1)$
(even after replacing $\frac 12$ by any constant $C>0$),
it is proved to hold for any data of the form $(u_0,0)$, if $d=0$ mod $4$ and for any data of the form $(0,u_1)$, if $d=2$ mod $4$. 
More specifically, in dimension $4$,
\cite[Corollary 2]{CoKeSc14} first proved the existence of a constant $C>0$ such that any radial solution of the free wave equation \eqref{FW} satisfies either for all $t\geq 0$ or for all $t\leq 0$,
\begin{equation}
\label{CKS}
\|\nabla u_L(t)\|_{L^2(|x|>t)}^2 + \|\partial_t u_L(t)\|_{L^2(|x|>t)}^2 
\geq \frac{1}{C} \|\nabla u_0\|^2_{L^2}.
\end{equation}
Note that by decomposing the solution $u_L$ in even and odd parts (in the time variable), it is equivalent to state
an estimate of the type \eqref{eq:even} for initial data of the form $(u_0,0)$ for both $t\geq 0$ and $t\leq 0$, or to state
\eqref{CKS} for any general initial data $(u_0,u_1)$, either for all $t\geq 0$ or for all $t\leq 0$.
We also recall that the exterior energy being a nonincreasing function of $|t|$, it is also equivalent to formulate an asymptotic estimate, \emph{i.e.} in terms of the limits $\lim_{\pm \infty} \big\{\|\nabla u_L(t)\|_{L^2(|x|>t)}^2 + \|\partial_t u_L(t)\|_{L^2(|x|>t)}^2\big\}$.
In the sequel, we will restrict ourselves to initial data of the form $(u_0,0)$ and we will focus on 
limits of the exterior energy as $t\to \infty$.

The main goal of this section is to obtain lower bounds on the exterior energy over channels 
$\{|x|>t+R\}$ in space dimension $4$, in the same spirit as in \cite{DuKeMe11a,KeLaLiSc15} for odd dimensions. We will also deduce, by a simple trick, an analogous  estimate in space dimension $6$ (see Subsection \ref{SS:N6}). This will leave open the same question for dimensions $d=4n$, $n\geq 2$ integer, and the analogue estimate for initial data of the form $(0,u_1)$ for dimension $d=4n+2$, $n\geq 2$ integer.

\subsection{Exterior energy estimates for the linear equation}
Here, we discuss a slightly more precise version of \eqref{CKS} which follows from combining~\cite{CoKeSc14} and~\cite{DuKeMe19}.

\begin{prop}[\cite{CoKeSc14},\cite{DuKeMe19}]\label{pr:channel4D}
Let $u_0\in \hdot(\RR^4)$ have radial symmetry.
Let $u_L$ be the solution of \eqref{FW} with initial data $(u_0,0)$. Then
\begin{equation}\label{P:exterior}
\lim_{t\to\infty} \|\nabla u_L(t)\|_{L^2(|x|>t)}^2
=\lim_{t\to\infty} \|\partial_t u_L(t)\|_{L^2(|x|>t)}^2
\geq \frac 14 \|\nabla u_0\|_{L^2}^2.
\end{equation}
\end{prop}
For the reader's convenience, we provide a simplified proof of \eqref{P:exterior} compared to \cite{CoKeSc14}, using computations from \cite{CoKeSc14,KeLaLiSc15} and the notion of \emph{radiation profile} (see \cite[Theorem 2.1]{DuKeMe19} and references therein).

\begin{proof}
Let $u_0\in \hdot(\RR^4)$ with radial symmetry and let $u_L$ be the solution of \eqref{FW} with initial data $(u_0,0)$.
From Proposition \ref{P:radiation} in the case $f=0$, there exists
a function $G\in L^2(\RR)$, called the radiation profile, such that
\begin{equation}\label{rad_prof}\begin{aligned}
&\lim_{t\to\infty} \int_0^\infty \left|r^\frac 32 \partial_r u_L(t,r)-G(r-t)\right|^2 dr=0,\\
&\lim_{t\to\infty} \int_0^\infty \left|r^\frac 32 \partial_t u_L(t,r)+G(r-t)\right|^2 dr=0,
\end{aligned}
\end{equation}
with $c_3 \|G\|_{L^2}^2 = \frac 12 \int |\nabla u_0|^2$, where $c_3$ is the measure of the unit sphere of $\RR^4$.
We observe that the equality of the two limits in~\eqref{P:exterior} is a direct consequence of~\eqref{rad_prof}.
This being known, the result of Proposition~\ref{pr:channel4D} follows from \cite[Corollary~2]{CoKeSc14}, 
since looking carefully at the proof of this result, one sees that the constant $c(d)$ appearing there is exactly $c(d)=\frac12$.
As in \cite{CoKeSc14}, the proof extends to any space dimension $d=4n$, where $n\geq 1$ is integer.

\emph{Notation.}
We use the following notation for the Fourier transform of a Schwartz function~$f$
\[
\hat f (\xi) = \int_{\RR^4} e^{-i x \xi} f(x) dx,\quad f(x) = (2\pi)^{-4} \int_{\RR^4} e^{i x\xi} \hat f(\xi) d\xi.
\]
In particular,
\begin{align*}
\|\hat f\|_{L^2(\RR^4)}^2 &= (2\pi)^4 \|f\|_{L^2(\RR^4)}^2,\tag{Plancherel}\\ 
\int_{\RR^4} \hat f(\xi) \hat g(\xi) d\xi & = (2\pi)^4 \int_{\RR^4} f(x) \bar g(x) dx. \tag{Parseval}
\end{align*}
Recall that if $f$ has the radial symmetry, then $\hat f$ also has the radial symmetry and the Plancherel identity yields
\[
\int_0^\infty |\hat f(\rho)|^2 \rho^3 d\rho = (2\pi)^4 \int_0^{\infty} |f(r)|^2 r^3 dr.
\]
Moreover,
\begin{equation}\label{eq:Fradial}
f(r) = (2\pi)^{-2} \int_0^\infty \hat f(\rho) J_1(r\rho) (r\rho)^{-1} \rho^3 d\rho
\end{equation}
where $J_1$ is the Bessel function of first type of order one.

\emph{Reduction of the proof.}
We assume without loss of generality that $u_0:\RR^4\to \RR$ is a Schwartz function with radial symmetry such that for some 
$0<\rho_*<\rho^*$, it holds
\begin{equation*}
\supp \hat u_0 \subset B_{\rho^*}(0)\setminus B_{\rho_*}(0).
\end{equation*}
Using the formulas (12), (14) of \cite{CoKeSc14} with $f=u_0$ and $g\equiv 0$,
we have
\begin{equation}\label{eq:Aeps}
(2\pi)^4 \lim_{t\to \infty}\int_t^\infty |\partial_t u_L(t,r)|^2 r^3 dr
= \frac 2\pi \lim_{t\to \infty} \lim_{\varepsilon\downarrow 0} A_\varepsilon(t) 
\end{equation}
where
\begin{equation*}
A_\varepsilon(t)
=\int_t^\infty \int_0^\infty\int_0^\infty 
\hat u_0(\rho_1) \overline{\hat u}_0(\rho_2)
k(t,\rho_1,\rho_2) \rho_1^{\frac 52}\rho_2^{\frac 52} e^{-\varepsilon r} d\rho_1\,d\rho_2\,dr
\end{equation*}
the function $k$ being defined by
\[
k(t,\rho_1,\rho_2) = 
\sin(t\rho_1)\sin(t\rho_2)\cos\left(r\rho_1-\frac {3\pi}4\right) \cos\left(r\rho_2-\frac {3\pi}4\right).
\]
Recall that \eqref{eq:Aeps} follows from the Fourier inversion formula for radial functions~\eqref{eq:Fradial},
and the asymptotic behavior of the Bessel functions, see \cite[(10)]{CoKeSc14}.
Now, we use standard trigonometry formulas to obtain
\begin{align*}
2 \cos\left(r\rho_1-\frac {3\pi}4\right) \cos\left(r\rho_2-\frac {3\pi}4\right)
&= \cos r(\rho_1-\rho_2) + \cos\left(r (\rho_1 + \rho_2) -\frac{3\pi}2\right)\\
& = \cos r(\rho_1-\rho_2) - \sin r (\rho_1 + \rho_2).
\end{align*}
Thus,
\[
k(t,\rho_1,\rho_2)
=\frac 12 \left[ \cos r(\rho_1-\rho_2) - \sin r (\rho_1 + \rho_2)\right] 
\sin (t\rho_1)\sin(t\rho_2)\]

Now, applying formulas (15) and (16) from \cite{CoKeSc14} with 
\[
\phi(\rho_1)= \hat u_0(\rho_1) \sin(t\rho_1)\rho_1^\frac 52,
\quad 
\psi(\rho_2)=\overline{\hat u_0(\rho_2)} \sin(t\rho_2)\rho_2^\frac 52,
\]
we obtain
\[
\lim_{\varepsilon\downarrow 0} A_\varepsilon(t) = A_1(t)+A_2(t)+A_3(t)
\]
where
\begin{align*}
A_1(t) &=\frac \pi 2 \int_0^\infty |\hat u_0(\rho)|^2 \sin^2(t\rho) \rho^5 d\rho, \\
A_2(t) &=-\frac 12 \int_0^\infty\int_0^\infty \hat u_0(\rho_1)\overline{\hat u}_0(\rho_2) k_2(t,\rho_1,\rho_2)\rho_1^\frac 52\rho_2^\frac 52d\rho_1\,d\rho_2,
\\
A_3(t) &=-\frac 12 \int_0^\infty\int_0^\infty \hat u_0(\rho_1)\overline{\hat u}_0(\rho_2)k_3(t,\rho_1,\rho_2) \rho_1^\frac 52\rho_2^\frac 52d\rho_1\,d\rho_2,
\end{align*}
the functions $k_2$ and $k_3$ being defined by
\begin{align*}
k_2(t,\rho_1,\rho_2) & =\frac{\sin t(\rho_1-\rho_2)}{\rho_1-\rho_2} \sin(t\rho_1)\sin(t\rho_2),\\
k_3(t,\rho_1,\rho_2) & =\frac{\cos t(\rho_1+\rho_2)}{\rho_1+\rho_2} \sin(t\rho_1)\sin(t\rho_2).
\end{align*}

\emph{Calculation of $\lim_\infty A_1$.}
We use $\sin^2(t\rho)= \frac 12 (1-\cos (2t\rho))$ and we remark that
\[
\lim_{t\to \infty} \int_0^\infty |\hat u_0(\rho)|^2 \cos(2t\rho) \rho^5 d\rho =0,
\]
by integration by parts. Thus, we obtain
\begin{equation}\label{eq:A1}
\lim_{t\to \infty} A_1(t) = \frac \pi 4 \int_0^\infty |\hat u_0(\rho)|^2 \rho^5 d\rho.
\end{equation}

\emph{Calculation of $\lim_\infty A_2$.}
Applying the formula
\[
2 \sin (t\rho_1)\sin(t\rho_2) = \cos t(\rho_1-\rho_2) - \cos t(\rho_1+\rho_2),
\]
and 
\begin{equation*}
\sin t(\rho_1-\rho_2)\cos t(\rho_1-\rho_2) =\frac 12 \sin 2t(\rho_1-\rho_2), 
\end{equation*}
we find
\[
k_2(t,\rho_1,\rho_2)=\frac1{4(\rho_1-\rho_2)} \Big( \sin 2t(\rho_1-\rho_2)- 2 \sin t(\rho_1-\rho_2)\cos t(\rho_1+\rho_2)\Big).
\]
For the first term in the above expression of $k_2$, we use the computation after (19) in \cite{KeLaLiSc15}, which yields
\begin{equation*}
\lim_{t\to \infty} \int_0^\infty\int_0^\infty \hat u_0(\rho_1)\overline{\hat u}_0(\rho_2) \frac{\sin 2t(\rho_1-\rho_2)}{\rho_1-\rho_2}\rho_1^\frac 52\rho_2^\frac 52d\rho_1\,d\rho_2
= \pi \int_0^\infty |\hat u_0(\rho)|^2 \rho^5 d\rho.
\end{equation*}
For the second term in the above expression of $k_2$, we rewrite
\begin{multline*}
\int_0^\infty\int_0^\infty \hat u_0(\rho_1)\overline{\hat u}_0(\rho_2) \frac{\sin t(\rho_1-\rho_2)}{\rho_1-\rho_2}\cos t(\rho_1+\rho_2)\, \rho_1^\frac 52\rho_2^\frac 52d\rho_1\,d\rho_2\\
=\int_0^\infty \int_0^\infty t \sinc t(\rho_1-\rho_2)\cos t(\rho_1+\rho_2)\phi(\rho_1+\rho_2)d\rho_1d\rho_2,
\end{multline*}
where $\sinc \theta=\frac{\sin \theta}{\theta}$, $\phi(\rho_1,\rho_2)=\hat u_0(\rho_1)\overline{\hat u}_0(\rho_2) \rho_1^\frac 52\rho_2^\frac 52$.
By the assumptions on $u_0$ and $\hat u_0$, the function
$\phi$ is smooth and compactly supported on $(0,\infty)^2$. By integration by parts, the above expression is equal to 
\begin{multline*}
 -\frac{1}{4t} \int_0^\infty \int_0^\infty \sinc t(\rho_1-\rho_2)\left(\frac{d}{d\rho_1}+\frac{d}{d\rho_2}\right)^2\Big(\cos t(\rho_1+\rho_2)\Big)\phi(\rho_1,\rho_2)d\rho_1d\rho_2
\\= -\frac{1}{4t} \int_0^\infty \int_0^\infty \sinc t(\rho_1-\rho_2)\cos t(\rho_1+\rho_2)\left(\frac{d}{d\rho_1}+\frac{d}{d\rho_2}\right)^2\Big(\phi(\rho_1,\rho_2)\Big)d\rho_1d\rho_2,
\end{multline*}
which goes to $0$ as $t\to\infty$.
Therefore, we obtain
\begin{equation}\label{eq:A2}
\lim_{t\to \infty} A_2(t) = - \frac \pi 8 \int_0^\infty |\hat u_0(\rho)|^2 \rho^5 d\rho.
\end{equation}

\emph{Estimate for $\lim_\infty A_3$.}
Using trigonometry, we compute
\begin{align*}
- k_3(t,\rho_1,\rho_2) & = 
\frac 1{2(\rho_1+\rho_2)}\cos t(\rho_1+\rho_2) \left(\cos t (\rho_1+\rho_2)-\cos t (\rho_1-\rho_2) \right)\\
&=\frac 1{4(\rho_1+\rho_2)} \left(1+\cos 2t(\rho_1+\rho_2) -\cos 2t\rho_1 -\cos 2t \rho_2\right).
\end{align*}
By integration by parts, the contribution of the last three terms to the limit of $A_3$ is zero.
Therefore, we obtain
\begin{align*}
\lim_{t\to \infty} A_3(t)
&= \frac 18 \int_0^\infty\int_0^\infty \frac 1{\rho_1+\rho_2} \hat u_0(\rho_1)\overline{\hat u}_0(\rho_2)\rho_1^\frac 52\rho_2^\frac 52d\rho_1\,d\rho_2\\
&=\frac 18 \int_0^\infty \hat u_0(\rho_1)\rho_1^\frac 52 H\big(\rho^\frac 52 \overline{\hat u}_0 \big)(\rho_1) d\rho_1,
\end{align*}
where $H$ denotes the Hankel operator, defined by
\[
(H u) (\rho_1)=\int_0^\infty \frac {u(\rho)} {\rho+\rho_1} d\rho.
\]
Since $H$ is a positive operator (see references in the Introduction of \cite{CoKeSc14}), 
and since $\hat u_0$ is real-valued ($u_0$ has radial symmetry), we obtain
\begin{equation}\label{eq:A3}
\lim_{t\to \infty} A_3(t) \geq 0.
\end{equation}

\emph{Conclusion.} Gathering \eqref{eq:Aeps}, \eqref{eq:A1}, \eqref{eq:A2}, \eqref{eq:A3}
and using
\[
\int_0^\infty |\hat u_0(\rho)|^2 \rho^5 d\rho
=(2\pi)^4 \int_0^\infty |\partial_r u_0(r)|^2 r^3 dr,
\]
we have proved
\begin{equation*}
\lim_{t\to \infty} \int_t^\infty |\partial_t u_L(t,r)|^2 r^3 dr 
\geq \frac 14 \int_0^\infty |\partial_r u_0(r)|^2 r^3 dr.
\end{equation*}
Since
\[
\lim_{t\to \infty} \int_t^\infty |\partial_r u_L(t,r)|^2 r^3 dr 
=\lim_{t\to \infty} \int_t^\infty |\partial_t u_L(t,r)|^2 r^3 dr 
\]
as a direct consequence of \eqref{rad_prof}, the proof of \eqref{P:exterior} is complete.
\end{proof}

\subsection{Exterior energy estimates over channels}
We prove a lower bound of the asymptotic energy over channels of the form $\{|x|>|t|+R\}$ for any $R>0$.
 We recall that $1/r^2$ is a solution of \eqref{FW} on the wave cone $\{|x|>R+|t|\}$, with initial data $(1/r^2,0)\in \hdot(R)\times L^2(R)$.
We denote by $\vect_R(1/r^2)$ the one dimensional subspace of $\hdot(R)$ spanned by the function $1/r^2$ and by $\pi_R$ the orthogonal projection on $\vect_R(1/r^2)$ for the $\hdot(R)$ scalar product, explicitly given for any function $f\in \hdot(R)$ by 
\[
\pi_R(f) (r)= \frac {R^2f(R)}{r^2},\quad r>R.
\]
The projection onto the orthogonal complement of $\vect_R(1/r^2)$ is given by
\begin{equation*}
\pi_R^\perp(f)(r)=f(r)-\frac{R^2f(R)}{r^2},\quad r>R.
\end{equation*}
We have the following result.
\begin{prop}
\label{P:FW_R}
Let $R>0$ and let $u_0\in \hdot(R)$ have radial symmetry.
Let $u_L$ be the solution of \eqref{FW} with initial data 
$(u_0,0)$ in the region $\{(t,x)\in [0,\infty)\times \RR^4: |x|>R+|t|\}$. Then
\[
\lim_{t\to\infty} \|\nabla u_L(t)\|_{L^2(|x|>t+R)}^2 =\lim_{t\to\infty} \|\partial_t u_L(t)\|_{L^2(|x|>t+R)}^2
\geq \frac 3{20} \|\nabla \pi_R^\perp u_0\|^2_{L^2(|x|>R)}.
\]
 \end{prop}

To prove Proposition \ref{P:FW_R}, we will need the following estimate.
\begin{lemma}\label{pr:linearv}
Let $v_0:\RR^4\to\RR$ be a smooth function with radial symmetry such that for some~$A>1$,
\begin{equation}\label{eq:supp}
\supp(v_0)\subset B_A(0)\setminus B_1(0).
\end{equation}
Let $v_L$ be the solution of \eqref{FW} with initial data $(v_0,0)$.
Then,\begin{equation}\label{on:v}
\limsup_{t\to \infty} \|\partial_r v_L(t)\|_{L^2(t< |x|< t+1)}^2
\leq \frac 1{10}\|\partial_r v_0\|_{L^2}^2.
\end{equation}
\end{lemma}
\begin{proof}[Proof of Proposition~\ref{P:FW_R} assuming Lemma~\ref{pr:linearv}]
We start with the case $R=1$.
We define the space-time region $\Sigma_1=\{(t,x)\in [0,\infty)\times \RR^4 : |x|>t+1\}$.
First, consider a smooth initial data $(w_0,0)$ with radial symmetry such that 
\begin{equation}\label{eq:suppu}
\supp(w_0)\subset B_A(0)\setminus B_1(0),
\end{equation}
for some $A>0$,
and denote by $w_L$ the corresponding solution of \eqref{FW} on $\Sigma_1$.
Applying Lemma~\ref{pr:linearv} to $w_L$, we have
\begin{equation*}
\lim_{t\to \infty} \|\nabla w_L(t)\|_{L^2(t< |x|< t+1)}^2
\leq \frac 1{10}\|\nabla w_0\|_{L^2}^2.
\end{equation*}
Thus, by Proposition~\ref{pr:channel4D} and $\|\nabla w_0\|_{L^2}=\|\nabla w_0\|_{L^2(|x|>1)}$, we obtain
\begin{equation}\label{eq:still}
\lim_{t\to \infty} \|\nabla w_L(t)\|_{L^2(|x|> t+1)}^2
\geq \frac 3{20}\|\nabla w_0\|_{L^2(|x|>1)}^2.
\end{equation}

Second, consider a radial function $w_0\in \hdot(\RR^4)$ that vanishes at $r=1$.
By finite speed of propagation, we can assume that $w_0(r)$ is zero on $[0,1]$ without changing the solution $w_L$ on $\Sigma_1$.
Moreover, such a function $w_0$ can be approximated in $\hdot(\RR^4)$ by a sequence of smooth functions $(w_{0,n})_n$ satisfying
\eqref{eq:suppu}, with some $A_n\to\infty$. By this density argument, \eqref{eq:still} also holds for any such $w_0$.

Third, consider a general initial data $u_0 \in \hdot(\RR^4)$ with radial symmetry.
Let
\[
w_0(x) =\pi_1^\perp(u_0)(x) = u_0(x) - \frac{u_0(1)}{|x|^2} .
\]
Since $w_0$ vanishes at $r=1$, if we extend it by $0$ for $r\in [0,1]$, the property~\eqref{eq:still} holds for $w_L$, and 
$w_L= u_L - \frac{u_0(1)}{|x|^2}$ on $\Sigma_1$ since 
 $\frac 1{|x|^2}$ is a solution of \eqref{FW} on $\Sigma_1$, 
 and uniqueness holds on $\Sigma_1$.
Moreover, 
\[\lim_{t\to \infty} \|\nabla u_L(t)\|_{L^2(|x|>t+1)}^2
=\lim_{t\to \infty} \|\nabla w_L(t)\|_{L^2(|x|>t+1)}^2,\]
which implies that
\[
\lim_{t\to\infty} \|\nabla u_L(t)\|_{L^2(|x|>t+1)}^2 \geq  \frac 3{20}\|\nabla \pi_1^\perp u_0\|^2_{L^2(|x|>1)}.
\]
By \eqref{rad_prof}, we have
\[
\lim_{t\to\infty} \|\partial_t u_L(t)\|_{L^2(|x|>t+1)}^2 = \lim_{t\to\infty} \|\nabla u_L(t)\|_{L^2(|x|>t+1)}^2,
\]
which proves the result for $R=1$.

We complete the proof by a scaling argument. Let $R_0>0$ and let $u_L$ be a solution of \eqref{FW}.
Applying the result for $R=1$ to the solution $v_L(t,x)=u_L(R_0 t, R_0 x)$
implies the result for the solution $u_L$ for $R=R_0$.
\end{proof}

\begin{proof}[Proof of Lemma~\ref{pr:linearv}]
Recall that $v_L$ is given explicitly by the representation formula
\begin{equation*}
v_L(t,x') 
=\frac 1{4\pi^2}\left(\frac \partial{\partial t}\right)\left( \frac 1t\frac \partial {\partial t}\right)
\left( \int_{|x-x'|<t} \frac{v_{0}(x)}{\left( t^2-|x-x'|^2 \right)^{\frac 12}} dx\right)
\end{equation*}
(see \cite[(38) of \S2.4]{Evans10Bo} for $n=4$, and then use the fact that the volume of the four dimensional unit ball is $\frac{\pi^2}2$).

Denote by ${\mathbf{e}_1}$ the first vector of the canonical basis of $\RR^4$. 
For any $t,s> 0$, set
\[
V(t,s)=\frac {1}{4\pi^2t^3} \int_{|x-s{\mathbf{e}_1}|<t} \frac{v_0(x)}{\left( t^2-|x-s{\mathbf{e}_1}|^2 \right)^{\frac 12}} dx,
\]
so that
\begin{align*}
v_L(t,s{\mathbf{e}_1}) &=\left(\frac {\partial}{\partial t}\right) \left( \frac 1t \frac {\partial}{\partial t}\right)\left( t^3 V\right)(t,s)\\
&=3 V (t,s)+5 t \partial_t V (t,s)+ t^2 \partial_t^2 V (t,s)\\
&=(V_1+V_2+V_3)(t,s).
\end{align*}
We focus on $V_3=t^2\partial_t^2 V$.
By the change of variable
\begin{equation}
y=\frac{x-s{\mathbf{e}_1}}{t},\quad x=s{\mathbf{e}_1} + t y
\end{equation}
we rewrite $V$ as
\begin{equation}\label{eq:formV}
V(t,s)=\frac 1{4\pi^2} \int_{|y|<1} \frac{v_0(s{\mathbf{e}_1} + t y)}{\left(1-|y|^2 \right)^{\frac 12}} dy.
\end{equation}
Thus,
\begin{align*}
V_3(t,s)& =\frac {t^2} {4\pi^2} \partial_t^2 \left(\int_{|y|<1} \frac{v_0(s{\mathbf{e}_1} + t y)}{\left(1-|y|^2 \right)^{\frac 12}} dy\right)\\
&=\frac {t^2}{4\pi^2} \int_{|y|<1} \frac{\sum_{j,k=1}^4 y_j y_k \partial_{x_j}\partial_{x_k}v_0(s{\mathbf{e}_1} + t y)}{\left(1-|y|^2 \right)^{\frac 12}} dy.
\end{align*}
Turning back to the variable $x$ using
\[
y = -{\mathbf{e}_1} + \frac 1t (x-(s-t){\mathbf{e}_1}),
\]
we obtain
\begin{align*}
V_3=V_4+V_5,
\end{align*}
where
\begin{equation}\label{eq:formV4}
\begin{aligned}
V_4(t,s) & =
-\frac1{2\pi^2t^2} 
\int_{|x-s{\mathbf{e}_1}|<t} \frac {(x_1-(s-t)) \partial_{x_1}^2v_0(x)}{\left( t^2-|x-s{\mathbf{e}_1}|^2 \right)^{\frac 12}} dx\\
&\quad+\frac1{4\pi^2t^3}
\int_{|x-s{\mathbf{e}_1}|<t} \frac{(x_1-(s-t))^2 \partial_{x_1}^2v_0(x)}{\left( t^2-|x-s{\mathbf{e}_1}|^2 \right)^{\frac 12}} dx\\
&\quad -\frac1{2\pi^2t^2}
\int_{|x-s{\mathbf{e}_1}|<t} \frac{\sum_{k=2}^4 x_k \partial_{x_1}\partial_{x_k}v_0(x)}{\left( t^2-|x-s{\mathbf{e}_1}|^2 \right)^{\frac 12}} dx\\
&\quad+\frac1{2\pi^2t^3}
\int_{|x-s{\mathbf{e}_1}|<t} \frac{\sum_{k=2}^4 (x_1-(s-t)) x_k \partial_{x_1}\partial_{x_k}v_0(x)}{\left( t^2-|x-s{\mathbf{e}_1}|^2 \right)^{\frac 12}} dx\\
&\quad +\frac1{4\pi^2t^3}
\int_{|x-s{\mathbf{e}_1}|<t} \frac{\sum_{j,k=2}^4 x_j x_k \partial_{x_j}\partial_{x_k}v_0(x)}{\left( t^2-|x-s{\mathbf{e}_1}|^2 \right)^{\frac 12}} dx
\end{aligned}
\end{equation}
and 
\[
V_5(t,s)=\frac {1}{4\pi^2t}\int_{|x-s{\mathbf{e}_1}|<t} \frac{\partial_{x_1}^2 v_0(x)}{\left( t^2-|x-s{\mathbf{e}_1}|^2 \right)^{\frac 12}} dx.
\]
We focus on $V_5$.
We are interested in the regime $t\gg 1$, for $s\in [t,t+1]$. Thus, we set
\[
\eta = s-t\in [0,1],\quad s=t+\eta.
\]
In particular,
\[
V_5(t,s)
=\frac {1}{4\pi^2t}\int_{|x-(t+\eta){\mathbf{e}_1}|<t} \frac{\partial_{x_1}^2v_0(x)}{\left( t^2-|x-(t+\eta){\mathbf{e}_1}|^2 \right)^{\frac 12}} dx.
\]
We compute
\[
t^2-|x-(t+\eta){\mathbf{e}_1}|^2
=2 t (x_1 - \eta) - |x-\eta {\mathbf{e}_1}|^2.
\]
Since $\eta \in [0,1]$, $|x|\leq A$ and $t\gg 1$, we have $t^2-|x-(t+\eta){\mathbf{e}_1}|^2= 2 t(x_1-\eta)+\mathcal{O}(1)$ and the domain of integration defined by $t^2-|x-(t+\eta){\mathbf{e}_1}|^2>0$ converges asymptotically as $t\to\infty$ to the domain $\{x\in \RR^4: x_1>\eta\}$. This motivates the following decomposition
\[
V_5=V_6 + V_7
\]
with
\[
V_6(t,s)=V_5(t,s)- t^{-\frac 32} G(s-t)
\]
and
\[
V_7(t,s)=t^{-\frac 32} G(s-t)
\]
where the function $G:[0,1]\to \RR$ is defined by 
\[
G(\eta)=\frac {1}{4\sqrt{2}\pi^2}\int_{x_1>\eta} \frac{\partial_{x_1}^2v_0(x)}{\left(x_1-\eta\right)^{\frac 12}} dx
\]
(the notation $\int_{x_1>\eta}$ means integration over the half-space $\{x\in \RR^4: x_1>\eta\}$).
Summarising, we have decomposed $v_L(t,s{\mathbf{e}_1})$ as 
\[
v_L(t,s{\mathbf{e}_1})=\left( V_1+V_2+V_4+V_6+V_7\right)(t,s).
\]

To estimate the main term $V_7$, we start with the following lemma.

\begin{lemma}\label{le:hyper}
Let $f:\RR^4\to \RR$ be a continuous, compactly supported function with radial symmetry.
For any $\eta>0$,
\[
\int_{x_1>\eta} \frac{f(x)}{\left(x_1-\eta\right)^{\frac 12}} dx
= \eta^{-\frac 12} \int_\eta^\infty f(\rho) h\left(\frac \rho\eta\right) \rho^3 d\rho
\]
where
\begin{equation}\label{eq:h}
h(r) = \frac{4\pi}{r^2} \int_1^r \sqrt{\frac{r^2-z^2}{z-1}} dz.
\end{equation}
\end{lemma}
\begin{proof}
We use the hyperspherical coordinates in dimension four:
\begin{align*}
x_1 & = \rho \cos \psi\\
x_2 & = \rho \sin \psi \cos \theta\\
x_3 & = \rho \sin \psi \sin \theta \cos \varphi\\
x_4 & = \rho \sin \psi \sin \theta \sin \varphi
\end{align*}
with $\rho>0$, $\psi\in (0,\pi)$, $\theta\in (0,\pi)$ and $\varphi\in (0,2\pi)$.
Using
\[
dx = \rho^3 \sin^2 \psi \sin\theta \,d\rho \, d\psi\, d\theta \,d\varphi,
\]
the change of variable yields
\begin{equation*}
\int_{x_1>\eta} \frac{f(x)}{\left(x_1-\eta\right)^{\frac 12}} dx
=4\pi \int_\eta^\infty f(\rho) \left(\int_0^{\arccos \eta/\rho} \frac{\sin^2 \psi}{(\rho \cos \psi-\eta)^{\frac 12}} d\psi\right)
\rho^3 d\rho
\end{equation*}
(the constant $4\pi$ comes from the integral of $(\theta,\varphi)\mapsto \sin \theta$ over $(0,\pi)\times (0,2\pi)$).
Next, changing variable $z=(\rho/\eta) \cos \psi$,
\begin{equation*}
4\pi \int_0^{\arccos \eta/\rho} \frac{\sin^2 \psi}{(\rho\cos \psi-\eta)^{\frac 12}} d\psi
= 4\pi \frac {\eta^{\frac 32}} {\rho^2} \int_1^{\rho/\eta} \sqrt{\frac{(\rho/\eta)^2-z^2}{z-1}} dz
=\eta^{-\frac 12} h\left(\frac \rho\eta\right),
\end{equation*}
where the function $h$ is defined in \eqref{eq:h}.
\end{proof}
By integration by parts, it holds for $j=2,3,4$,
\[
\int_{x_1>\eta} \frac{\partial_{x_j}^2v_0(x)}{\left(x_1-\eta\right)^{\frac 12}} dx=0.
\]
This implies that the function $G$ rewrites as
\[
G(\eta) = \frac 1{4\sqrt 2 \pi^2}\int_{x_1>\eta} \frac{\Delta v_0(x)}{\left(x_1-\eta\right)^{\frac 12}} dx.
\]
The function $\Delta v_0$ has radial symmetry and is equal to $r^{-3} \left( r^3 v_0'\right)'$.
Thus, using Lemma~\ref{le:hyper},
\[
G(\eta) 
 = \frac 1{4\sqrt 2 \pi^2}
 \eta^{-\frac 12} \int_\eta^\infty \left( \rho^3 v_0'\right)'(\rho) h\left(\frac \rho\eta\right)d\rho.
\]
Thus, integrating by parts twice and using \eqref{eq:supp},
\begin{equation*}
G(\eta)
=\frac 1{4\sqrt 2 \pi^2} \int_1^\infty \rho^{\frac 12} v_0(\rho) H_1\left(\frac \rho\eta\right) d\rho
\end{equation*}
where we have defined
\[
H_1(r)= r^{-\frac 12} \left( r^3 h'\right)'.
\]

Differentiating in $\eta$, we find
\[
G'(\eta) = - \frac 1{4\sqrt 2 \pi^2} \frac 1\eta \int_1^\infty \rho^{\frac 12} v_0(\rho) H\left(\frac \rho\eta\right) d\rho
\]
where
\begin{equation}\label{def:H}
H(r) =  r H_1'(r) = r \left( r^{-\frac 12} \left( r^3 h'\right)' \right)'.
\end{equation}
The Cauchy-Schwarz inequality and then a change a variable $\rho=r \eta$ yield
\begin{align*}
(G')^2(\eta)
&\leq \frac 1{32\pi^4}
\left( \frac 1{\eta^2} \int_1^\infty H^2\left(\frac {\rho}\eta\right) d\rho\right)
\left(\int_1^\infty v_0^2(r) r dr\right)\\
&\leq \frac 1{32\pi^4}\left( \frac 1{\eta}\int_{1/\eta}^\infty H^2(r) dr\right)
\left(\int_1^\infty v_0^2(r) r dr\right)
.
\end{align*}
Integrating on $[0,1]$ in the variable $\eta$ and then using Tonelli's theorem,
\begin{align*}
\int_{0}^1 (G')^2(\eta) d\eta
&\leq \frac 1{32\pi^4}
\left( \int_{0}^1 \frac 1{\eta}\int_{1/\eta}^\infty H^2(r) dr\, d\eta\right)
\left(\int_1^\infty v_0^2(r) r dr\right)\\
&\leq \frac 1{32\pi^4}
\left( \int_1^\infty H^2(r) \log r dr \right)
\left(\int_1^\infty v_0^2(r) r dr\right).
\end{align*}

\begin{lemma}[Estimate on $I$]\label{le:I}
Let
\[
I=\frac 1{32\pi^4} \int_1^\infty H^2(r) \log r \,dr.
\]
It holds
\[
I\leq \frac 1{10} .
\]
\end{lemma}
\begin{proof}
We change variable in the expression of $h$ in \eqref{eq:h}
\begin{align*}
& z=1+\zeta(r-1),\quad z-1=\zeta(r-1),\\
& r^2-z^2=(r-z)(r+z)=(1-\zeta)(r-1)(r+1+\zeta(r-1)),
\end{align*}
and we obtain
\begin{align*}
h(r) & = 4\pi\left(\frac 1r-\frac 1{r^2}\right)\int_0^1 (1-\zeta)^{\frac 12} \zeta^{-\frac 12} K(\zeta,r) d\zeta ,\\
K (\zeta,r) &=(r(1+\zeta)+1-\zeta)^{\frac 12}.
\end{align*}
Denote
\begin{align}
K' & =\frac{\partial K}{\partial r} =\frac 12 (1+\zeta) K^{-1},\label{eq:dK}\\
K'' & =\frac{\partial^2 K}{\partial r^2}=- \frac 14 (1+\zeta)^2 K^{-3},\label{eq:ddK}\\
K''' & =\frac{\partial^3 K}{\partial r^3}= \frac 38 (1+\zeta)^3 K^{-5}.\label{eq:dddK}
\end{align}
We compute $H(r)$. First,
\begin{align*}
 h'=4\pi\left(-\frac 1{r^2}+\frac 2{r^3}\right)\int_0^1 (1-\zeta)^{\frac 12} \zeta^{-\frac 12} K d\zeta
+4\pi\left(\frac 1{r}-\frac 1{r^2}\right)\int_0^1 (1-\zeta)^{\frac 12} \zeta^{-\frac 12}K' d\zeta
\end{align*}
and thus
\begin{align*}
r^3 h'=4\pi\left(-r+2\right)\int_0^1 (1-\zeta)^{\frac 12} \zeta^{-\frac 12} K d\zeta
+4\pi\left(r^2-r\right)\int_0^1 (1-\zeta)^{\frac 12} \zeta^{-\frac 12}K' d\zeta.
\end{align*}
Second,
\begin{align*}
(r^3 h')'&= -4\pi\int_0^1 (1-\zeta)^{\frac 12} \zeta^{-\frac 12} K d\zeta
+4\pi\left(r+1\right)\int_0^1 (1-\zeta)^{\frac 12} \zeta^{-\frac 12}K' d\zeta
\\&\quad +4\pi\left(r^2-r\right)\int_0^1 (1-\zeta)^{\frac 12} \zeta^{-\frac 12}K'' d\zeta
\end{align*}
and so
\begin{align*}
r^{-\frac 12} \left( r^3 h'\right)'
&= -4\pi r^{-\frac 12} \int_0^1 (1-\zeta)^{\frac 12} \zeta^{-\frac 12} K d\zeta
+4\pi\left(r^{\frac 12}+r^{-\frac 12}\right)\int_0^1 (1-\zeta)^{\frac 12} \zeta^{-\frac 12}K' d\zeta
\\&\quad +4\pi\left(r^{\frac 32}-r^{\frac 12}\right)\int_0^1 (1-\zeta)^{\frac 12} \zeta^{-\frac 12}K'' d\zeta.
\end{align*}
Third,
\begin{align*}
\left(r^{-\frac 12} \left( r^3 h'\right)'\right)'
&= 2\pi r^{-\frac 32} \int_0^1 (1-\zeta)^{\frac 12} \zeta^{-\frac 12} K d\zeta\\
&\quad -2\pi\left(r^{-\frac 12}+r^{-\frac 32}\right)\int_0^1 (1-\zeta)^{\frac 12} \zeta^{-\frac 12}K' d\zeta
\\&\quad +2\pi\left(5r^{\frac 12}+r^{-\frac 12}\right)\int_0^1 (1-\zeta)^{\frac 12} \zeta^{-\frac 12}K'' d\zeta\\
&\quad +4\pi\left(r^{\frac 32}-r^{\frac 12}\right)\int_0^1 (1-\zeta)^{\frac 12} \zeta^{-\frac 12}K''' d\zeta.
\end{align*}
Thus, we have obtained
\[
 H(r)=2\pi r^{-\frac 12} \int_{0}^1(1-\zeta)^{\frac 12} \zeta^{-\frac 12} Z(\zeta,r)d\zeta
\]
where
\begin{equation*}
Z(\zeta,r)= K-(r+1)K'+(5r^2+r) K''+2(r^3-r^2)K'''.
\end{equation*}
We replace $K'$, $K''$ and $K'''$ by their expressions in \eqref{eq:dK}, \eqref{eq:ddK} and \eqref{eq:dddK},
\begin{multline*}
Z(\zeta,r) = K^{-5} \bigg[K^6-\frac 12 (r+1)(1+\zeta)K^4\\-\frac 14(5r^2+r)(1+\zeta)^2 K^2+\frac 34 (1+\zeta)^3(r^3-r^2)\bigg].
\end{multline*}
Next, we insert
\[
K^2(\zeta,r) =(r(1+\zeta)+1-\zeta)
\]
and we expand in powers of $r$
\begin{align*}
Z(\zeta,r) & = K^{-5} \bigg[
r^3\left\{ (1+\zeta)^3 -\frac 12 (1+\zeta)^3 -\frac 54 (1+\zeta)^3+\frac 34 (1+\zeta)^3\right\}\\
&\qquad +r^2 \biggl\{3(1+\zeta)^2(1-\zeta)-(1+\zeta)^2(1-\zeta) - \frac 12 (1+\zeta)^3\\
&\qquad \qquad-\frac 54 (1+\zeta)^2(1-\zeta) - \frac 14 (1+\zeta)^3 - \frac 34 (1+\zeta)^3
\biggr\}\\
& \qquad + r \biggl\{3(1+\zeta)(1-\zeta)^2-\frac 12 (1+\zeta)(1-\zeta)^2\\
&\qquad\qquad-(1+\zeta)^2(1-\zeta)-\frac 14 (1+\zeta)^2(1-\zeta)\biggr\}\\
& \qquad + \left\{(1-\zeta)^3 - \frac 12 (1+\zeta)(1-\zeta)^2\right\}\bigg].
\end{align*}
Note that  in the first line, the terms in $r^3$ vanish. 
This expression simplifies into
\begin{equation*}
Z(\zeta,r) = \frac 14 K^{-5} \Bigr[ - 3 r^2 (1+\zeta)^2(1+3\zeta)
+ 5 r(1+\zeta)(1-\zeta)(1-3\zeta)
+ 2(1-\zeta)^2(1-3\zeta)\Bigl].
\end{equation*}
We observe that for all $\zeta\in [0,1]$, $|1-3\zeta|\leq 1+\zeta$.
Indeed, for $\zeta\in [0,\frac 13]$, $|1-3\zeta|=1-3\zeta\leq 1$, while for $\zeta\in [\frac 13,1]$,
$|1-3\zeta|=3\zeta-1=1+\zeta-2(1-\zeta)\leq 1+\zeta$.
It follows that 
\begin{equation*}
|Z(\zeta,r)|
\leq \frac 14 K^{-5} \left[3 r^2 (1+\zeta)^2 (1+3\zeta)+ 5r(1+\zeta)^2(1-\zeta)+2(1-\zeta)^2(1+\zeta)\right].
\end{equation*}
Therefore,
\begin{align*}
|H(r)|&\leq\frac{3\pi}2 r^{\frac 32}  \int_0^1(1-\zeta)^{\frac 12} \zeta^{-\frac 12} (1+\zeta)^2 (1+3\zeta)K^{-5}(\zeta,r) d\zeta\\
&\quad + \frac{5\pi}2 r^{\frac 12}  \int_0^1(1-\zeta)^{\frac 32} \zeta^{-\frac 12}  (1+\zeta)^2K^{-5}(\zeta,r) d\zeta\\
&\quad + \pi r^{-\frac 12} \int_0^1(1-\zeta)^{\frac 52} \zeta^{-\frac 12}  (1+\zeta) K^{-5}(\zeta,r) d\zeta
\end{align*}
Now, we estimate $K^{-5}$ to reduce to simple integrals.
Since
\[
K^{-5}(\zeta,r)=r^{-\frac 52} (1+\zeta)^{-\frac 52} \left(1+\frac {1-\zeta}{r(1+\zeta)}\right)^{-\frac 52}
\leq r^{-\frac 52} (1+\zeta)^{-\frac 52},
\]
we have
\[
|H(r)|\leq a r^{-1} + b r^{-2} + c r^{-3}
\]
where
\begin{align*}
a& = \frac{3\pi}2  \int_0^1(1-\zeta)^{\frac 12}\zeta^{-\frac 12} (1+\zeta)^{-\frac 12}(1+3\zeta) d\zeta\\
b& = \frac{5\pi}2 \int_0^1(1-\zeta)^{\frac 32} \zeta^{-\frac 12}  (1+\zeta)^{-\frac 12}  d\zeta\\
c& = \pi \int_0^1(1-\zeta)^{\frac 52} \zeta^{-\frac 12}  (1+\zeta)^{-\frac 32}   d\zeta.
\end{align*}
We estimate $a$, $b$ and $c$.
For $a$, using the change of variable $\tilde \zeta=\zeta^{\frac 12}$ and then the Cauchy-Schwarz inequality,
\begin{multline*}
a = 3 \pi \int_0^1 (1-\tilde \zeta^2)^{\frac 12} (1+\tilde \zeta^2)^{-\frac 12} (1+3\tilde\zeta^2) d\tilde \zeta\\ 
\leq    3\pi \sqrt{\int_0^1 \frac{d\tilde\zeta}{1+\tilde\zeta^2}}\sqrt{\int_0^1 (1-\tilde\zeta^2) (1+3\tilde\zeta^2)^2 d\tilde\zeta}=3\pi\sqrt{\frac{\pi}{4}\left(1+\frac 53+\frac 35-\frac 97\right)} 
\leq 12.
\end{multline*}
We proceed similarly for $b$ and $c$, which yields
\begin{align*}
b &= 5 \pi \int_0^1 (1-\tilde \zeta^2)^{\frac 32} (1+\tilde \zeta^2)^{-\frac 12} d\tilde \zeta\\
&\leq 5\pi \sqrt{\int_0^1 \frac{d\tilde\zeta}{1+\tilde\zeta^2}} \sqrt{\int_0^1 (1-\tilde\zeta^2)^3 d\tilde\zeta}
=5\pi \sqrt{\frac{\pi}{7}}\leq 11,\\
c  &= 2 \pi \int_0^1 (1-\tilde \zeta^2)^{\frac 52} (1+\tilde \zeta^2)^{-\frac 32} d\tilde \zeta\leq 2\pi.
\end{align*}
For the sake of simplicity, we bound $b\leq 12$ and $c\leq 12$, so that we have proved
\[
|H(r)|\leq 12\left( r^{-1} + r^{-2} +  r^{-3} \right).
\]
Using the expression of $I$ and the estimate on $H$, it holds
\begin{align*}
\int_1^\infty H^2(r) \log r dr
&\leq 144 \int_1^\infty ( r^{-1} + r^{-2}+ r^{-3})^2 \log r dr\\
&\leq 144 \int_1^\infty \left( r^{-2} +2 r^{-3}
+ 3 r^{-4}+ 2 r^{-5} + r^{-6}\right) \log r dr.
\end{align*}
Using $\int_1^\infty r^{-k} \log r dr=  (k-1)^{-2}$, we obtain
\begin{equation*}
I \leq \frac {144}{32\pi^4}\left( 1\times 1 + 2 \times \frac 1{4}
+3 \times \frac 19 + 2 \times \frac 1{16} + 1 \times \frac 1{25} \right)
=\frac{3597}{400}\frac 1{\pi^4} \leq \frac 1{10} 
\end{equation*}
where the last estimate, checked numerically, is given for the sake of simplicity.
\end{proof}
 
We continue the proof of Lemma~\ref{pr:linearv}.
By Lemma \ref{le:I} and then the Hardy inequality
(see \emph{e.g.} \cite[Appendix A.4]{Stein70Bo})
\[
\int_0^\infty v_0^2 r dr \leq \int_0^\infty (v_0')^2 r^3 dr,
\] 
we obtain
\begin{equation*}
\int_{0}^1 (G')^2(\eta) d\eta
\leq \frac{1}{10}\int_1^\infty v_0^2(r) r dr
=\frac{1}{10}\int_0^\infty v_0^2(r) r dr
\leq \frac{1}{10} \int_0^\infty (v_0')^2(r) r^3 dr.
\end{equation*}
For the term $V_7$, we conclude
\begin{align*}
\int_t^{t+1} (\partial_s V_7)^2 (t,s) s^3 ds
&=t^{-3} \int_{t}^{t+1} (G')^2(s-t) s^3 ds\\
&=t^{-3} \int_0^1 (G')^2(\eta) (t+\eta)^3 d\eta
\leq \left(1+\frac 1t\right)^3 \int_0^1 (G')^2(\eta) d\eta,
\end{align*}
and so
\begin{equation}\label{on:V17}
\limsup_{t\to \infty} \int_{t\leq |x|\leq t+1} |\nabla V_7|^2 (t,x) dx
\leq \frac{1}{10}\int  |\nabla v_0(x)|^2 dx.
\end{equation}

Now, we claim that for $k=1,2,4,6$, for $t$ large and for any~$\eta \in [0,1]$, it holds
\begin{equation}\label{on:V1to6bis}
|\partial_s V_k(t,t+\eta)| \lesssim t^{-\frac 52}.
\end{equation}
Observe that setting $s=t+\eta$, estimate~\eqref{on:V1to6bis} implies that
\begin{multline*}
\int_t^{t+1} (\partial_s V_k)^2(t,s) s^3 ds 
=\int_0^1 (\partial_s V_k)^2(t,t+\eta) (t+\eta)^3 d\eta \\
\leq (t+1)^3 \int_0^1 (\partial_s V_k)^2(t,t+\eta)\, d\eta\lesssim t^{-2}.
\end{multline*}
Thus, for $k=1,2,4,6$,
\begin{equation}\label{on:V1to6}
\lim_{t\to \infty} \int_t^{t+1} (\partial_s V_k)^2(t,s) s^3 ds =0.
\end{equation}
Since estimates \eqref{on:V17} and \eqref{on:V1to6} imply Lemma~\ref{pr:linearv}, we only have to prove~\eqref{on:V1to6bis}.
We start with the following technical result.
\begin{lemma}\label{le:1}
Let $A>0$ and let $f:\RR^4\to \RR$ be a smooth function such that $\supp f \subset B_A(0)$.
Define
\[
\Omega[f](t) = \int_{|x-t{\mathbf{e}_1}|<t} \frac{f(x)}{\left( t^2-|x-t{\mathbf{e}_1}|^2 \right)^{\frac 12}} dx,
\]
For $t$ large enough,
\begin{equation}\label{eq:asympf}
\left|
(2t)^{\frac 12} \Omega[f](t)
-\int_{x_1>0} \frac{f(x)} {x_1^{\frac 12}} dx
\right|\leq  \frac Ct
\end{equation}
and
\begin{equation}\label{eq:boundf}
\left|\Omega[f](t)\right|\leq C t^{-\frac 12}
\end{equation}
where the constant $C$  depends on $A$, $\|f\|_{L^\infty}$ and $\|\nabla f\|_{L^\infty}$.
\end{lemma}
\begin{remark}
We do not assume that $f$ has radial symmetry.
\end{remark}
\begin{proof}
Observe that \eqref{eq:boundf} is a direct consequence of \eqref{eq:asympf}.
We prove \eqref{eq:asympf}.
Since 
$t^2-|x-t{\mathbf{e}_1}|^2 = 2tx_1 -|x|^2$,
we rewrite
\[
(2t)^{\frac 12}\Omega[f](t)
=\int_{x_1>\frac{|x|^2}{2t}} \frac{f(x)} {\left( x_1 - \frac{|x|^2}{2t} \right)^{\frac 12}} dx.
\]
We change variables, setting
\[
y=\Phi(x)=\left(x_1-\frac{|x|^2}{2t},x_2,x_3,x_4\right).
\]
For $t$ large enough (depending on $A$), the map $\Phi$ is a diffeomorphism from $B_{2A}(0)$ to its image $\Phi(B_{2A}(0))$,
and $B_{A}(0)\subset \Phi(B_{2A}(0))\subset B_{4A}(0)$.
We compute $\Phi^{-1}$.
Using the notation $\bar x=(x_2,x_3,x_4)$, $\bar y=(y_2,y_3,y_4)$, we have for $y=\Phi(x)$,
\[
|\bar y|^2 = |\bar x|^2,\quad x_1^2 - 2t x_1 + 2t y_1 + |\bar y|^2 =0,
\]
and thus
\[
\Phi^{-1}(y)=\left( t - \sqrt{t^2 -2ty_1-|\bar y|^2},y_2,y_3,y_4\right).
\]
In particular, on $\Phi(B_{2A}(0))$, it holds
\[
|\Phi^{-1}(y)-y|\lesssim \frac 1t.
\]
Let $J_\Phi$ denote the Jacobian matrix of $\Phi$.
Since $\det J_\Phi(x)= 1-\frac {x_1}t$, we find 
\[
\det J_\Phi(\Phi^{-1}(y)) = \sqrt{1 -2t^{-1} y_1-t^{-2} |\bar y|^2},
\]
and thus on $\Phi(B_{2A}(0))$, it holds
\[
|\det J_\Phi(\Phi^{-1}(y)) - 1 |\lesssim \frac 1t.
\]
Let
\begin{align*}
X_A & = \left\{x\in \RR^4: x\in B_{2A}(0) \mbox{ and } x_1-\frac{|x|^2}{2t}>0\right\},\\
Y_A & = \{y\in \RR^4: y \in B_{4A}(0) \mbox{ and } y_1>0\}.
\end{align*}
Since $\supp f \subset B_A(0)$ and $\Phi(X_A)\subset Y_A$,
\begin{align*}
\int_{x_1-\frac{|x|^2}{2t}>0} \frac{f(x)} {\left( x_1 - \frac{|x|^2}{2t} \right)^{\frac 12}} dx
&= \int_{X_A} \frac{f(x)} {\left( x_1 - \frac{|x|^2}{2t} \right)^{\frac 12}} dx\\
&= \int_{\Phi(X_A)} \frac{f(\Phi^{-1}(y))} { y_1^{\frac 12}}\frac {dy}{|\det J_\Phi(\Phi^{-1}(y))|}\\
&\leq \int_{Y_A} \frac{f(\Phi^{-1}(y))} { y_1^{\frac 12}}\frac {dy}{|\det J_\Phi(\Phi^{-1}(y))|}.
\end{align*}
On $Y_A$, one has
\[
|f(\Phi^{-1}(y))-f(y)|\lesssim \frac 1t,\quad
\left|\frac 1{\det J_\Phi(\Phi^{-1}(y))} -1 \right|\lesssim \frac 1t.
\]
Thus,
\[
\left|\int_{Y_A} \frac{f(\Phi^{-1}(y))} { y_1^{\frac 12}}\frac {dy}{|\det J_\Phi(\Phi^{-1}(y))|}
- \int_{y_1>0} \frac{f(y)}{y_1^{\frac 12}} dy \right| \lesssim \frac 1t,
\]
which completes the proof of the lemma.
\end{proof}

\emph{Estimate for $V_1$.}
We observe by the change of variable $x=x'-\eta{\mathbf{e}_1}$ that
\[
V_1(t,t+\eta) = 3 V (t,t+\eta)  = \frac 3{4\pi^2 t^3} \int_{|x-t{\mathbf{e}_1}|<t} \frac {v_0(x+\eta{\mathbf{e}_1})}{(t^2 -|x-t{\mathbf{e}_1}|^2)^{\frac 12}} dx
\]
and so
\[
\partial_s V_1(t,t+\eta)
=\frac 3{4\pi^2 t^3} \Omega[\partial_{x_1}v_0(\cdot+\eta \mathbf{e}_1)] (t)
\]
For $\eta\in [0,1]$, the support of the function $x\mapsto \partial_{x_1}v_0(x+\eta{\mathbf{e}_1})$ is included in $B_{A+1}(0)$ and thus
by~\eqref{eq:boundf}, we have $|\partial_s  V_1(t,t+\eta)|\lesssim t^{-\frac 72}$ for $t$ large.

\emph{Estimate for $V_2$.}
By direct computation using the expression of $V$ in \eqref{eq:formV}, we have
\[
V_2(t,s)=5t \partial_t V (t,s)
=\frac{5t}{4\pi^2} \int_{|y|<1} \frac {\sum_{j=1}^4 y_j \partial_{x_j} v_0(s{\mathbf{e}_1} + ty)}{(1-|y|^2)^{\frac 12}} dy.
\]
By the change of variable $y = \frac {x}t-{\mathbf{e}_1}$,
we rewrite
\begin{align*}
V_2(t,t+\eta) &= \frac{5}{4\pi^2 t^2} \int_{|x-t{\mathbf{e}_1} |<t} \frac {\left( \frac {x_1}t-1\right) \partial_{x_1} v_0(x+\eta{\mathbf{e}_1})}
{(t^2-|x-t{\mathbf{e}_1}|^2)^{\frac 12}} dx\\
&\quad + \frac{5}{4\pi^2 t^3} \int_{|x-t{\mathbf{e}_1} |<t} \frac {\sum_{j=2}^4 x_j \partial_{x_j} v_0(x+\eta{\mathbf{e}_1})}
{(t^2-|x-t{\mathbf{e}_1}|^2)^{\frac 12}} dx.
\end{align*}
Arguing as for $V_1$, using \eqref{eq:boundf}, for any $\eta\in [0,1]$, we find 
$|\partial_s  V_{2}(t,t+\eta)|\lesssim t^{-\frac 52}$.

\emph{Estimate for $V_4$.}
The first term in the expression of $V_4$ in \eqref{eq:formV4} is rewritten as
\[
V_{4,1}(t,t+\eta)=-\frac 1{2\pi^2t^2} \int_{|x-t {\mathbf{e}_1}|<t} \frac{ x_1 \partial_{x_1}^2 v_0(x+\eta{\mathbf{e}_1})}{(t^2 -|x-t{\mathbf{e}_1}|^2)^{\frac 12}} dx.
\]
Arguing as for $V_1$, using \eqref{eq:boundf}, for $\eta\in [0,1]$, we find $|\partial_s  V_{4,1}(t,t+\eta)|\lesssim t^{-\frac 52}$.
The other terms in the expression of $V_4$ are estimated similarly.

\emph{Estimate for $V_6$.}
Recall that
\begin{align*}
V_6(t,t+\eta)
& =  \frac {t^{-\frac 32}}{4\sqrt{2}\pi^2}
\left( (2t)^{\frac 12} \Omega[\partial_{x_1}^2v_0(\cdot+\eta\mathbf{e}_1)](t)
- \int_{x_1>0} \frac{\partial_{x_1}^2v_0(x+\eta\mathbf{e}_1)}{x_1^{\frac 12}} dx\right).
\end{align*}
Thus, by \eqref{eq:asympf} of Lemma \ref{le:1}, applied to the function $x\mapsto \partial_{x_1}^3v_0(x+\eta {\mathbf{e}_1})$, it follows
that for any $\eta\in (0,1)$,
$|\partial_s  V_6(t,t+\eta)|\lesssim t^{-\frac 52}$
and the proof of \eqref{on:V1to6bis} is complete.
\end{proof}
\subsection{Exterior energy in space dimension $6$}
\label{SS:N6}
In this subsection, we consider the linear wave equation in space dimension $6$:
\begin{equation}
\label{FW6}
\partial_t^2u_L-\Delta u_L=0,\quad (t,x)\in \RR\times \RR^6,\quad |x|>R+|t|.
\end{equation}
With initial data of the form:
\begin{equation}
\label{ID6}
\vec{u}_{\restriction t=0}=(0,u_1),\quad u_1\in L^2(\{x\in \RR^6,\; |x|>R\}).
\end{equation}
We note that $t/r^4$ is a radial solution of \eqref{FW6},\eqref{ID6}, with initial data $(0,1/r^4)$. We will denote by $\Pi_R$ the orthogonal projection, on $\mathrm{span}(1/r^4)$ in the space of radial functions in $L^2(\{x\in \RR^6,\; |x|>R\})$, and by $\Pi_R^{\bot}=\mathrm{Id}-\Pi_R$. We will prove:
\begin{prop}
 \label{P:channel6d}
 Let $u$ be a radial solution of \eqref{FW6} with initial data of the form \eqref{ID6}. Then 
 \begin{multline}
\label{ineg6}
 \lim_{t\to\infty}\int_{t+R}^{\infty}  |\nabla u_L(t,r)|^2r^5dr =\lim_{t\to\infty} \int_{t+R}^{\infty} |\partial_t u_L(t,r)|^2r^5dr
\\
\geq \frac 3{20}\int_{R}^{\infty} \left|\Pi_R^\perp u_1(r)\right|^2r^5dr.
\end{multline}
\end{prop}
\begin{proof}
This is a simple consequence of Proposition \ref{P:FW_R}.
 By using the radiation profile (see Proposition \eqref{ILW}), it suffices to prove the inequality in \eqref{ineg6}. By scaling, we can assume $R=1$. Consider
 $$ v(t,r)=\int_r^{\infty} \rho\partial_tu_L(t,\rho)d\rho.$$
A calculation shows that $v(t,r)$ solves the linear wave equation in $\RR\times \RR^4$ for $|x|>1+|t|$. Moreover, $\partial_rv(0,r)\in L^2(\{x\in \RR^4,\;|x|>1\})$ and $\partial_tv(0,r)=0$ since $u_L(0,r)=0$. By Proposition \ref{P:FW_R}, we have that  
\begin{equation*}
\lim_{t\to\infty}\int_{1+t}^{\infty}(\partial_tu_L(t,r))^2r^5dr=\lim_{t\to\infty}\int_{1+t}^{\infty}(\partial_rv(t,r))^2r^3dr\geq \frac{3}{20}\int_{1+t}^{\infty} \left|\nabla \pi_1^{\bot}v(0)\right|^2r^3dr.
\end{equation*}
Now $\partial_rv(0,r)=-ru_1(r)$, and
\begin{multline*}
\partial_r (\pi^{\bot}_1v(0,r))=\frac{\partial}{\partial r}\left( v(0,r)-\frac{v(0,1)}{r^2} \right)\\
=-ru_1(r)+\frac{2}{r^3}\int_{1}^{\infty} \rho u_1(\rho)d\rho=-r\left( u_1(r)-\frac{2}{r^4}\int_1^{\infty} \rho u_1(\rho)d\rho \right),
\end{multline*}
which by a simple calculation equals $-r\Pi^{\bot}_1(u_1)$, and the proposition follows.
 \end{proof}

\subsection{Exterior energy for the linear inhomogeneous wave equation}
We deduce from Propositions~\ref{pr:channel4D} and \ref{P:FW_R} similar lower bounds on the exterior energy for the inhomogeneous problem \eqref{LWinhom}.

\begin{lemma}
\label{L:inhomog}
Let $u_0\in \dot{H}^1(\RR^4)$ and $f\in L^1((0,\infty),L^2(\RR^4))$ have radial symmetry.
Let $u$ be the solution of \eqref{LWinhom}. Then
\begin{equation}
\label{bound_inhom}
\|u_0\|_{\hdot}\leq 2 \lim_{t\to +\infty} \|\nabla u(t)\|_{L^2(|x|>t)}+ 2 \left\|\indic_{\{|x|>|t|\}} f\right\|_{L^1((0,\infty),L^2)},
\end{equation} 
and for $R>0$,
\begin{equation}
\label{bound_inhomR}
\|\pi_R^{\bot} u_0\|_{\hdot(R)}\leq \sqrt{\frac {20}3} \left(\lim_{t\to +\infty} \|\nabla u(t)\|_{L^2(|x|>t+R)}
+ \left\|\indic_{\{|x|>|t|+R\}} f\right\|_{L^1((0,\infty),L^2)}\right). 
\end{equation} 
\end{lemma}
\begin{proof}
 Let $v$ be the solution of 
\begin{equation*}
\left\{
\begin{aligned}
\partial_t^2v-\Delta v & =\indic_{\{|x|>|t|\}}f\\
\vec{v}_{\restriction t=0} & =(0,0) .
\end{aligned}
\right.
\end{equation*} 
On the one hand, by standard energy computations and the Cauchy-Schwarz inequality, one sees that
\[
\left| \frac d{dt} \|\nabla_{t,x}v(t)\|_{L^2}\right| \leq \|\indic_{\{|x|>|t|\}}f(t)\|_{L^2}
\]
and so by integration on $(0,\infty)$, 
\[
\|\nabla_{t,x}v(t)\|_{L^2} \leq \|\indic_{\{|x|>|t|\}}f\|_{L^1L^2},
\]
where in this proof, we denote $\|\cdot\|_{L^1L^2}=\|\cdot\|_{L^1((0,\infty),L^2(\RR^4))}$.

On the other hand, we observe that the function $u-v$ satisfies
\begin{equation*}
\left\{
\begin{aligned}
\partial_t^2(u-v)-\Delta (u-v) &=\indic_{\{|x|\leq |t|\}}f\\
(\vec{u}-\vec{v})_{\restriction t=0} &=(u_0,0)
\end{aligned}
\right.
\end{equation*}
and so $u-v$ coincides in the region \{$|x|>|t|$\} with the solution $u_L$ of the free wave equation \eqref{FW} with initial data $(u_0,0)$. As a consequence, by Proposition \ref{pr:channel4D},
\begin{align*}
\|u_0\|_{\hdot} &\leq 2 \lim_{t\to+\infty} \|\nabla u_L\|_{L^2(|x|>t)}\\
&\leq 2 \lim_{t\to+\infty} \|\nabla u \|_{L^2(|x|>t)} + 2 \lim_{t\to+\infty} \|\nabla v \|_{L^2(|x|>t)}\\
&\leq 2 \lim_{t\to+\infty} \|\nabla u \|_{L^2(|x|>t)}+ 2\, \|\indic_{\{|x|>|t|\}}f\|_{L^1L^2},
\end{align*}
which completes the proof of \eqref{bound_inhom}.

The proof of \eqref{bound_inhomR} is similar, using the solution $v$ of 
\begin{equation*}
\left\{
\begin{aligned}
\partial_t^2v-\Delta v & =\indic_{\{|x|>|t|+R\}}f\\
\vec{v}_{\restriction t=0} & =(0,0) .
\end{aligned}
\right.
\end{equation*} 
and Proposition~\ref{P:FW_R} instead of Proposition~\ref{pr:channel4D}.
\end{proof}

\section{Rigidity theorem part I: compactly supported pertubations and constant sign solutions}
\label{S:rigidityI}

In this section, we will start the proof of Theorems \ref{T:rigidity} and \ref{T:rigidityWM}. We will prove the two theorems at the same time, reducing the co-rotational wave maps equation to a critical wave equation in four space dimension (see \S \ref{SS:reduction}). In \S \ref{SS:compact}, we consider solutions with initial data that are equal to a stationary solution for large $r$. In \S \ref{SS:sign} we will treat the case of solutions that have a constant sign outside a wave cone. Section \ref{S:odd} concerns solutions that are odd in time. 

\subsection{From co-rotational wave maps to four dimensional waves}
\label{SS:reduction}
Let $R> 0$, $(\psi_0,\psi_1)\in \Hbf(R)$. Let $\psi$ be the solution of \eqref{WM} on $\{|x|>R+|t|\}$. Let 
$$u(t,r)=\sqrt{\frac{2}{3}} \frac{\psi}{r},\;r>R+|t|\quad \text{and}  \quad (u_0,u_1)(r)=\vec{u}_{\restriction t=0}(r),\;r>R.$$
Then $(u_0,u_1)\in \HHH(R)$ and 
$$ \partial_t^2u-\partial_r^2u -\frac{3}{r}\partial_ru -\frac{1}{r^2}u+\frac{\sin(\sqrt{6}ru)}{\sqrt{6}ru}=0,\quad r>R+|t|.$$
We write this equation as
\begin{equation}
 \label{gNLW}
 \partial_t^2u-\partial_r^2u-\frac{3}{r}\partial_ru=\frac{1}{r^3}\Lambda(ru),
\end{equation} 
where 
\begin{equation}
 \label{defLambda}
 \Lambda(U)=U-\frac{\sin(\sqrt{6}U)}{\sqrt{6}}.
\end{equation} 
We note that $\Lambda$ is odd, smooth, and that it satisfies:
\begin{equation}
 \label{propLambda}
 \forall k\in \{0,1,2\}, \quad \left|\left(\frac{d}{dU}\right)^k \left(\Lambda(U)-U^3\right)\right|\lesssim \min \left( U^{5-k},U^{3-k} \right).
\end{equation} 
As a consequence, we have the following Lipschitz bound:
\begin{equation}
 \label{LipschitzLambda}
 \left|\Lambda(U)-\Lambda(V)\right|\lesssim |U-V|(U^2+V^2).
\end{equation} 
Note that \eqref{gNLW} with $\Lambda(U)=U^3$ is exactly the $4D$ radial wave equation (and that in this case, \eqref{propLambda} is trivially satisfied). 
We have the following small data statement:
\begin{prop} 
 \label{P:smalldata}
Let $\Lambda\in C^{\infty}(\RR)$, odd, such that \eqref{propLambda} holds. Then there exists $\eps>0$ with the following property. For all $(u_0,u_1)\in \HHH(R)$ such that $\|u_L\|_{(L^3L^6)(R)}\leq \eps$, the solution $u$ of \eqref{gNLW} with initial data $(u_0,u_1)$ is defined for $r>R+|t|$ and 
$$ \sup_{t\in \RR} \|\vec{u}(t)-\vec{u}_L(t)\|_{\HHH(R+|t|)} +\|u-u_L\|_{(L^3L^6)(R)}\lesssim \|u_L\|^3_{(L^3L^6)(R)},$$
where $u_L$ is the solution of the linear wave equation \eqref{FW} with initial data $(u_0,u_1)$.
\end{prop}
\begin{proof}[Sketch of proof]
Using equation \eqref{gNLW}, the Strichartz estimate outside the cone \eqref{Strichartz_cone} and a straightforward bootstrap argument, one first obtains:
$$ \|u\|_{(L^3L^6)(R)}\lesssim \|u_L\|_{(L^3L^6)(R)}.$$
The statement then follows from \eqref{gNLW} and \eqref{Strichartz_cone}.
\end{proof}
Recall that $Q(r)=2\arctan r$ is a stationary solution of \eqref{WM}. By the above reduction, we obtain a solution $\tW=\sqrt{\frac{2}{3}} \frac{1}{r}(\pi-Q)$ of \eqref{gNLW} (with $\Lambda$ is defined by \eqref{defLambda}), such that for large $r$,
\begin{equation*}
 \left|\tW(r)-\frac{2\sqrt{2}}{ \sqrt{3}r^2}\right|\lesssim \frac{1}{r^4}, \quad r\geq 1.
\end{equation*} 
Rescaling, we obtain a solution $W(r)=\lambda \tW(\lambda r)$ of \eqref{gNLW} such that 
\begin{equation}
 \label{gW} \left|W(r)-\frac{8}{r^2}\right|\leq \frac{C}{r^4},
\end{equation} 
that is, with exactly the same asymptotic behaviour as the stationary solution $W$ of \eqref{NLW} defined in the introduction. We denote by $\HHH(0^+)$ the set of functions $(u_0,u_1)$ defined on $(0,\infty)$ such that 
$$\forall R>0,\quad (u_0,u_1)_{\restriction (R,\infty)}\in \HHH(R),$$
and define similarly $\dot{H}^1(0^+)$.
Theorems \ref{T:rigidity} and \ref{T:rigidityWM} reduce to the following rigidity theorem:
\begin{theoint}
\label{T:g_rigidity} 
Let $\Lambda\in C^{\infty}(\RR)$ odd, such that \eqref{propLambda} holds. Assume that the wave equation \eqref{gNLW} admits a stationary solution $W \in \dot{H}^1(0^+)$  such that \eqref{gW} holds. Let $(u_0,u_1)\in \HHH(0^+)$, and assume that $\forall \lambda \in \RR$, $(u_0,u_1)$ is not equal to $(\lambda W(\lambda \cdot),0)$.
Then there exists $R>0$ such that the solution $u$ with initial data $(u_0,u_1)$ is defined on $\{|x|>R+|t|\}$ and
 \begin{equation}
  \label{g_channel}
  \sum_{\pm}\lim_{t\to\pm\infty} \int_{|x|>R+|t|}|\nabla_{t,x}u(t,x)|^2dx>0.
 \end{equation} 
\end{theoint}
\begin{claim}
 Theorem \ref{T:g_rigidity} implies Theorems \ref{T:rigidity} and \ref{T:rigidityWM}.
\end{claim}
\begin{proof}
The wave equation \eqref{NLW} is \eqref{gNLW} with $\Lambda(U)=U^3$. The ground state $W(x)=(1+|x|^2/8)^{-1}$ is a stationary solution of \eqref{NLW} that satisfies \eqref{gW}. Thus Theorem \ref{T:g_rigidity} implies that for any solution $u$ of \eqref{NLW} that is not a stationary solution, there exists $R>0$ such that $u$ is defined for $\{|x|>R+|t|\}$ and
\begin{equation}
\label{channelbis}
  \sum_{\pm}\lim_{t\to\pm\infty} \int_{|x|>R+|t|}(\partial_tu(t,x))^2+|\nabla_xu(t,x)|^2 dx>0.
 \end{equation} 
Next, observe that from finite speed of propagation and small data theory, if 
\begin{equation*}
  \sum_{\pm}\inf_{t\to\pm\infty} \int_{|x|>R+|t|}(\partial_tu(t,x))^2+|\nabla_xu(t,x)|^2 dx=0,
 \end{equation*} 
then \eqref{channelbis} cannot hold, concluding the proof of Theorem \ref{T:rigidity}.

Let $\psi$ as in Theorem \ref{T:g_rigidity}. Thus $(\psi_0,\psi_1)\in \Hbf_{\ell,m}$, which implies $(\psi_0-\pi m,\psi_1)\in \Hbf_{\ell-m,0}$.
Let 
$$u(t,r)=\sqrt{\frac{2}{3}}\frac{\psi(t,r)-\pi m}{r}, \quad (u_0,u_1)(r)=\vec{u}(0,r).$$
Since $(\psi_0-\pi m,\psi_1)\in \Hbf(R)$ for all $R$, we have $(u_0,u_1)\in \HHH(0^+)$. By the assumption on $(\psi_0,\psi_1)$, we have that $(u_0,u_1)$ is not identically $0$, and that it is not equal to $\pm \left(\frac{\sqrt{2}}{\sqrt{3}r}(\pi-Q(\lambda \cdot)),0\right)$ for all $\lambda>0$. By the definition of $W$, we deduce that $(u_0,u_1)$ satisfies the assumptions of Theorem \ref{T:g_rigidity}. Furthermore,
\begin{multline*}
\int_{R+|t|}^{+\infty} \left((\partial_t u)^2+(\partial_r u)^2\right)r^3dr=\frac{2}{3}\int_{R+|t|}^{+\infty} \left( (\partial_t \psi)^2+\left(\partial_r\psi-r^{-1}(\psi-\pi m) \right)^2 \right)r dr\\
\lesssim
\int_{R+|t|} \left(\partial_t\psi)^2+(\partial_r\psi)^2+\frac{1}{r^2}(\psi-\pi m)^2\right)rdr,
\end{multline*}
which shows that the conclusion \eqref{g_channel} of Theorem \ref{T:g_rigidity} implies 
 \begin{equation}
  \label{channelWM}
  \sum_{\pm}\lim_{t\to\pm\infty} \int_{R+|t|}^{\infty} \left((\partial_t\psi)^2+(\partial_r\psi(t,r))^2+\frac{1}{r^2}(\psi(t,r)-\pi m)^2\right)rdr>0.
 \end{equation} 
It remains to prove that \eqref{channelWM} implies the conclusion \eqref{lower_boundWM} of Theorem \ref{T:rigidityWM}. This is elementary, but requires some preliminaries on equation \eqref{WM} and we postpone the proof to Section \ref{S:resolutionWM}, after the proof of Claim \ref{Cl:smallWM}.
 \end{proof}
The remainder of this section and the next section are dedicated to  the proof of Theorem \ref{T:g_rigidity}. Until the end of Section \ref{S:odd}, $\Lambda$ is an odd, smooth function on $\RR$ satisfying \eqref{propLambda} and such that there exists a stationary solution $W\in \dot{H}^1$ of \eqref{gNLW} satisfying \eqref{gW}.

\subsection{Compactly supported perturbation of a stationary solution}
\label{SS:compact}
We first prove:
\begin{prop}
 \label{P:compact}
 Let $(u_0,u_1)\in \HHH(0^+)$, radial. Assume that there exists $R_0>0$ and $\lambda\in \RR$ such that 
 $$ \forall r>R_0,\quad (u_0,u_1)(r)=\left(\lambda W(\lambda r),0\right),$$ 
 and that the preceding equality does not hold for all $r>0$. 
 Then there exists $R>0$ such that the solution $u$ of \eqref{gNLW} with initial data $(u_0,u_1)$ is defined for $r>R+|t|$ and 
 $$\sum_{\pm} \lim_{t\to \pm \infty} \int_{|x|>R+|t|}|\nabla_{t,x}u(t,x)|^2dx>0.$$
\end{prop}
\begin{proof}
Rescaling, we may assume that $\lambda=1$. Let 
$$h=u-W,\quad (h_0,h_1)=(u_0-W,u_1)=\vec{h}_{\restriction t=0}.$$
Then the equation \eqref{gNLW} satisfied by $u$ is equivalent to
\begin{equation}
 \label{gNLW_h}
 \partial_t^2h-\Delta h=\frac{1}{r^3}\left(\Lambda(r(W+h))-\Lambda(rW)\right).
\end{equation} 
Furthermore by \eqref{propLambda}
\begin{equation}
\label{bound_h_right}
\frac{1}{r^3}\left|\Lambda(r(W+h))-\Lambda(rW)\right|\lesssim |h|(h^2+W^2).
\end{equation} 
We define 
$$\rho=\inf\left\{ \sigma>0,\; \int_{|x|>\sigma} |\nabla h_0|^2+h_1^2=0\right\}$$
By the assumptions of the proposition, $0<\rho<\infty$.
We fix a small $\eps>0$ such that 
\begin{equation}
 \label{small}
 0<\|(h_0,h_1)\|_{\HHH(\rho-\eps)}\ll 1.
\end{equation} 
Using the Strichartz estimates \eqref{Strichartz_cone}, and standard small data theory, we obtain that \eqref{gNLW_h} has a solution $h$ with initial data $(h_0,h_1)$ defined for $\{|x|>\rho-\eps+|t|\}$, and that 
$$\sup_{t\in \RR} \|\vec{h}\|_{\HHH(\rho-\eps+|t|)}\lesssim \eps_0.$$
By the radial Sobolev embedding and the property \eqref{gW} of $W$, we obtain that there exists a constant $C>0$ such that 
$$\forall t\in \RR,\;\forall r>\rho-\eps+|t|,\quad h^2+W^2\leq \frac{C}{r^2}.$$
Combined with \eqref{bound_h_right}, and taking a smaller $\eps$ if necessary, we see that Proposition 4.7 in \cite{DuKeMe19Pa} implies that the following holds for all $t>0$ or for all $t<0$:
\begin{equation}
 \int_{|x|\geq \rho-\eps+|t|}|\nabla_{t,x}h|^2dx\geq \frac{1}{8}\int_{|x|\geq \rho-\eps} \left(|\nabla h_0|^2+h_1^2\right)dx>0,
\end{equation} 
which concludes the proof, since $W\in \dot{H}^1(0^+)$.
\end{proof}

Proposition \ref{P:compact} reduces the proof of Theorem \ref{T:g_rigidity} to
\begin{prop}
\label{P:rigidity}
 Let $R_0>0$ and $u$ be a radial solution of \eqref{gNLW} for $|x|>|t|+R_0$, with initial data in $\HHH(R_0)$. Assume
$$ \sum_{\pm }\lim_{t\to\pm\infty}\int_{|x|>R_0+|t|}|\nabla_{t,x}u(t,x)|^2dx=0.$$
Then there exists $R\geq R_0$ such that one of the following holds:
 \begin{equation}
  \label{conclu_rigidity1}
 \forall r>R,\quad (u_0,u_1)(r)=\left(0,0\right),
 \end{equation} 
or 
 \begin{equation}
  \label{conclu_rigidity2}
 \exists \lambda>0, \; \iota\in \{\pm 1\},\; \forall r>R,\quad (u_0,u_1)(r)=\iota \left(\lambda W(\lambda r),0\right).
 \end{equation} 
 \end{prop}
 The proof of Proposition \ref{P:rigidity} is divided into \S \ref{SS:sign}, and Section \ref{S:odd}.
\subsection{Constant sign solutions}
\label{SS:sign}
In this section, we prove
\begin{prop}
 \label{P:RTa10}
Let $R_0>0$ and $u$ be a radial solution of \eqref{gNLW} defined for $\{|x|>|t|+R_0\}$ with 
\begin{equation*}
 \vec{u}_{\restriction t=0}=(u_0,u_1)\in \HHH(R_0).
\end{equation*} 
Assume
\begin{equation}
 \label{RTa10}
\sum_{\pm} \lim_{t\to \pm\infty} \int_{|x|>|t|+R_0}|\nabla_{t,x}u(t,x)|^2dx=0.
\end{equation} 
Then one of the following holds:
\begin{gather}
 \label{RTa11} 
\exists t_0\in \RR,\;\exists R>R_0,\quad \forall r>|t_0|+R,\quad u(t_0,r)=0\\
\label{RTa12} 
\exists \mu \in \RR\setminus\{0\},\; \exists R>R_0,\; \forall r>R,\quad (u_0,u_1)(r)=\left( \mu W(\mu r),0 \right).
\end{gather}
\end{prop}
Proposition \ref{P:RTa10} reduces the proof of Proposition \ref{P:rigidity} to the classification of solutions satisfying \eqref{RTa10} and \eqref{RTa11}, which we will carry out in the next section. Note that \eqref{RTa11} is not sufficient to conclude immediately that $u$ is identically $0$ outside a wave cone, since $\partial_tu(t_0,r)$ does not have to vanish for $r>|t_0|+R$.

We split the proof of Proposition \ref{P:RTa10} in a few lemmas.
\begin{lemma}
 \label{L:RTa20}The exists a small $\eps_0>0$ with the following property. Let $u$ be a radial solution of \eqref{gNLW} on $\{|x|>R+|t|\}$, where $R\geq 0$. Assume
\begin{gather}
\label{RTa20} 
\int_{|x|>R} |\nabla_{t,x}u(0,x)|^2dx\leq \eps_0^2\\
\label{RTa21}
\sum_{\pm} \lim_{t\to \pm\infty}\int_{|x|>R+|t|} |\nabla_{t,x}u(t,x)|^2dx=0.
\end{gather}
Then
\begin{gather}
 \label{RTa22}
\forall \rho\geq R,\quad |u_0(\rho)|\approx \frac{1}{\rho}\|u\|_{L^4(\rho)}\approx \frac{1}{\rho}\|u_0\|_{\dot{H}^1(\rho)}\\
\label{RTa23}
\forall r\geq \rho\geq R,\quad |u_0(r)|\leq 2\left( \frac{\rho}{r} \right)^{\frac{11}{6}} |u_0(\rho)|. 
\end{gather}
\end{lemma}
\begin{proof}
 Let 
$$u_+(t)=\frac{u(t)+u(-t)}{2},\quad u_-(t)=\frac{u(t)-u(-t)}{2}.$$
Then
\begin{equation}
 \label{RTa24}
\left\{ 
\begin{aligned}
\left|\partial_t^2u_+-\Delta u_+\right|&\lesssim u_+(u_+^2+u_-^2)\\ 
\vec{u}_{+\restriction t=0}&=(u_0,0),
\end{aligned}
\right.
\end{equation} 
where we have used that, $\Lambda$ being odd,
\begin{multline*}
\left|\Lambda(r u(t))+\Lambda(r u(-t))\right|=\left|\Lambda(ru(t))-\Lambda(-ru(-t))\right|\\
\lesssim r^3|u(t)+u(-t)|\left(u^2(t)+u^2(-t)\right), 
\end{multline*}
by \eqref{LipschitzLambda}.
By the small data theory for \eqref{gNLW} (see Proposition \ref{P:smalldata}), the condition \eqref{RTa20} implies $\|u\|_{(L^3L^6)(R)}\lesssim \eps_0^2$. Combining with the equation \eqref{RTa24} and Strichartz estimates, we deduce that for all $\rho\geq R$,
\begin{equation}
 \label{RTa30}
\|u_+\|_{(L^3L^6)(\rho)}+\sup_{t\in \RR} \left\| \vec{u}_+(t)\right\|_{\HHH(\rho+|t|)}\lesssim \|u_0\|_{\dot{H}^1(\rho)}.
\end{equation} 
Using the lower bound for the exterior energy (Lemma \ref{L:inhomog}) and equation \eqref{RTa24} again, we obtain 
\begin{equation}
 \label{RTa31}
\left\| \pi_{\rho}^{\bot} u_0\right\|_{\dot{H}^1(\rho)}\leq C\left( \|u_+^3\|_{(L^1L^2)(\rho)}+\|u_-^2u_+\|_{(L^1L^2)(\rho)} \right)\leq C\eps_0^2 \|u_0\|_{\dot{H}^1(\rho)}.
\end{equation} 
Since 
$$ \pi_{\rho}^{\bot} u_0(r)=u_0(r)-\frac{\rho^2}{r^2}u_0(\rho),\quad \left\| \frac{1}{r^2}\right\|_{\dot{H}^1(\rho)}=\frac{1}{\sqrt{2}\rho},$$
we deduce from \eqref{RTa31}, taking $\eps_0$ small enough,
\begin{equation}
 \label{RTa32}
\|u_0\|_{\dot{H}^1(\rho)} \leq C\rho|u_0(\rho)|.
\end{equation} 
Next, we deduce from \eqref{RTa31}, \eqref{RTa32} and Sobolev embedding that for $\rho\geq R$,
\begin{equation}
 \label{RTa33}
\left\|u_0-\frac{\rho^2}{r^2} u_0(\rho)\right\|_{L^4(\rho)}\leq C\eps^2_0 \rho|u_0(\rho)|.
\end{equation} 
This implies, noting that $\left\| \frac{1}{r^2}\right\|_{L^4(\rho)}=\frac{1}{\sqrt{2}\rho}$,
\begin{equation}
 \label{RTa40}
\left|\|u_0\|_{L^4(\rho)}-\frac{\rho |u_0(\rho)|}{\sqrt{2}}\right|\leq C\eps^2_0\rho |u_0(\rho)|.
\end{equation} 
Combining with \eqref{RTa32} and Sobolev embedding, we obtain \eqref{RTa22}. We also obtain, from \eqref{RTa40},
\begin{equation}
 \label{RTa41}
\|u_0\|^4_{L^4(\rho)}\leq \frac{\rho^4}{4^-} |u_0(\rho)|^4,
\end{equation} 
where $4^-$ is a positive constant, smaller than $4$, that can be chosen arbitrarily close to $4$ provided $\eps_0$ is small enough. Letting 
$$ f(\rho)=\int_{\rho}^{+\infty} |u_0(r)|^4r^3dr=\|u_0\|^4_{L^4(\rho)},$$
we can write \eqref{RTa41} as 
\begin{equation}
 \label{RTa42}
f(\rho)\leq -\frac{\rho}{4^-} f'(\rho),\quad \rho\geq R,
\end{equation} 
that is 
\begin{equation}
 \label{RTa43}
\frac{f'(\rho)}{f(\rho)} \leq -\frac{4^-}{\rho},\quad \rho \geq R.
\end{equation}
Integrating, we deduce that for $r\geq \rho\geq R$,
\begin{equation}
 \label{RTa44}
\frac{f(r)}{f(\rho)}\leq \frac{\rho^{4^-}}{r^{4^-}},
\end{equation}  
that is 
$$ \|u_0\|^4_{L^4(r)}\leq \left( \frac{\rho}{r} \right)^{4^-}\|u_0\|^4_{L^4(\rho)}.$$
Combining with \eqref{RTa40}, we deduce \eqref{RTa23}.
\end{proof}
\begin{lemma}
 \label{L:RTa50} 
Let $u$, $R$ and $\eps_0$ be as in Lemma \ref{L:RTa20} and assume that for all $t$, $u(t)$ is not identically $0$ on $[R+|t|,+\infty)$. Then $u$ never vanishes and has a constant sign on $\{|x|>R+|t|\}$. Furthermore, for all $t\in \RR$, there exists $\ell(t)\neq 0$ such that 
\begin{equation}
 \label{RTa51}
\forall r>R+|t|,\quad \left| r^2u(t,r)-\ell(t)\right|\leq Cr^4|u(t,r)|^3.
\end{equation} 
\end{lemma}
\begin{remark}
\label{R:smallru}
 Recall that by the radial Sobolev inequality, for all $t$, $\lim_{r\to\infty} ru(t,r)=0$. Thus \eqref{RTa51}, dividing by $r^2u(t,r)$, implies 
 \begin{equation}
  \label{limit}
  \lim_{r\to\infty} r^2u(t,r)=\ell(t).
 \end{equation} 
\end{remark}

\begin{proof}
 Taking a smaller $\eps_0>0$ if necessary, we see by finite speed of propagation and the small data theory that for every time $t_0$, the function $(t,r)\mapsto u(t+t_0,r)$ satisfies the assumptions of Lemma \ref{L:RTa20} with $R$ replaced by $R+|t_0|$. By the conclusion of this lemma and the assumption that for all $t$, $u(t,r)$ is not identically $0$ for $r>R+|t|$, we obtain that $u$ does not vanish for $r>R+|t|$. Since $u$ is a continuous function, we deduce that it has a constant sign for $r>R+|t|$. We will assume to fix ideas $u(t,r)>0$ for all $r>|t|+R$. By \eqref{RTa23},
\begin{equation}
 \label{RTa61}
\forall r>\rho\geq R,\quad 0<u_0(r)\leq 2\left( \frac{\rho}{r} \right)^{\frac{11}{6}} u_0(\rho).
\end{equation} 
Next, we observe that for all $(t,r)$ with $r>t+R$, 
$$|u_-(t,r)|=\frac{1}{2}\left|u(t)-u(-t)\right| \leq \frac{u(t)+u(-t)}{2}=u_+(t),$$
where we have used that $u$ is positive. The first inequality in \eqref{RTa31}, together with \eqref{RTa30} implies, for $\rho\geq R$,
\begin{equation}
 \label{RTa62}
\left\| \pi_{\rho}^{\bot}u_0\right\|_{\dot{H}^1(\rho)}\leq C\|u_+^3\|_{(L^1L^2)(\rho)}\leq C\|u_0\|^3_{\dot{H}^1(\rho)}.
\end{equation} 
Using \eqref{RTa22} and Sobolev inequalities as in the proof of Lemma \ref{L:RTa20}, we deduce
\begin{equation}
 \label{RTa63}
\left| \|u_0\|_{L^4(\rho)} -\frac{\rho}{\sqrt{2}} u_0(\rho)\right|\leq C\rho^3 u_0(\rho)^3,\quad \rho\geq R.
\end{equation} 
Letting as before $f(\rho)=\|u_0\|^4_{L^4(\rho)}$, we deduce, for $r\geq \rho$,
$$
\left|f(r)+\frac{r}{4} f'(r)\right|\leq C r^6 |u_0(r)|^6\leq \frac{C}{r^5} \rho^{11} u_0(\rho)^6,
$$
that is 
$$ \left|\frac{d}{dr} \left( r^4f(r) \right)\right|\leq \frac{C}{r^2}\rho^{11} u_0(\rho)^6.$$
As a consequence, $r^4f(r)$ has a limit $\ell'\geq 0$ as $r\to\infty$ and 
$$\left|\ell' -\rho^4f(\rho)\right|\leq C\rho^{10} u_0(\rho)^6,\quad \rho\geq R.$$
Recall that by \eqref{RTa22}, $\rho u_0(\rho)\approx f(\rho)^{1/4}$ goes to $0$ as $\rho\to\infty$. By \eqref{RTa63}, 
$$\left| f(\rho)-\frac{\rho^4}{4}u_0(\rho)^{4}\right|\leq C\rho^6 u_0(\rho)^6,$$
and thus
$$\left|\frac{4\ell'}{\rho^8} -u_0(\rho)^4\right|\leq C\rho^2 |u_0(\rho)|^6.$$
This implies $\ell'>0$ since $u_0(\rho)\neq 0$ and $\rho u_0(\rho)$ is small for $\rho$ large.
Using the bound $|A-B|\leq \frac{|A^4-B^4|}{A^3}$, $A,B>0$, we also deduce
$$\left|\frac{\sqrt{2}{\ell'}^{\frac{1}{4}}}{\rho^2}-u_0(\rho)\right|\leq C\rho^2|u_0(\rho)|^3,\quad \rho\geq R.$$
This yields \eqref{RTa51} for $t=0$. Applying this to the solution $(t,r)\mapsto u(t+t_0,r)$ which satisfies the same assumptions as $u$ with $R$ replaced by $R+|t|$, we deduce \eqref{RTa51} for all $t\in \RR$.
\end{proof}
\begin{lemma}
\label{L:RTa80}
 Let $u$, $R$, $\eps_0$ be as in Lemmas \ref{L:RTa20} and 
\ref{L:RTa50}. Then $\ell$ is independent of $t$ and
\begin{equation}
 \label{RTa80}
 \forall t\in \RR,\; \forall r\geq R+|t|,\quad |u(t,r)|+\frac{1}{r}\|u(t)\|_{\HHH(R)}\lesssim \frac{|\ell|}{r^2}.
\end{equation} 
 \end{lemma}
\begin{proof}
 We prove that $\ell$ is constant by contradiction. Let $(\tau,t)\in \RR^2$, and assume $\ell(t)>\ell(\tau)$. Since $\ell$ does not vanish, and $u$ has a constant sign, we can also assume without loss of generality  (using Remark \ref{R:smallru}), $\ell(\tau)>0$.
 
 Fix a small $\delta>0$, to be specified. By \eqref{limit}, we have that for large $r$
 $$0<(1-\delta)r^2u(t,r)\leq \ell(t)\leq (1+\delta)r^2u(t,r),$$
 and similarly
 $$0<(1-\delta)r^2u(\tau,r)\leq \ell(\tau)\leq (1+\delta)r^2u(\tau,r).$$
 As a consequence,
$$r^2(u(t,r)-u(\tau,r))\geq \frac{1}{1+\delta} \ell(t)-\frac{1}{1-\delta} \ell(\tau)=c>0$$ 
for $\delta >0$ small enough. Thus for large $r$,
\begin{equation}
 \label{lwr_bnd_diff}
u(t,r)-u(\tau,r)\geq \frac{c}{r^2}.
 \end{equation} 
Since $u$ is a finite energy solution of \eqref{gNLW}, on $\{r>R+|t|\}$, we have
$$u(t)-u(\tau)=\int_{\tau}^t\partial_tu(s)ds\in L^2(\rho)$$
for large $\rho$, contradicting \eqref{lwr_bnd_diff}. Thus $\ell$ is independent of $t$.

Going back to \eqref{RTa51}, we obtain
$$\forall \rho>R+|t|,\quad |\rho^2u(t,\rho)|-C\rho^4|u(t,\rho)|^3\leq \ell.$$
Since by \eqref{RTa22}
$$\rho|u(t,\rho)|\approx \|\vec{u}(t)\|_{\HHH(\rho)}\lesssim \eps_0,\quad \rho\geq R+|t|,$$
we deduce \eqref{RTa80} using that $\eps_0$ is small.
\end{proof}

\begin{proof}[Proof of Proposition \ref{P:RTa10}]
 We let $u$ be as in the proposition, and choose $R$ large, so that
 \begin{equation}
  \label{RTa90}
  \int_{|x|>R}|\nabla u_0|^2+|u_1|^2dx\leq \eps_0,
 \end{equation} 
 where $\eps_0$ is small as in the preceding lemmas. We assume that \eqref{RTa11} does not hold so that (taking a larger $R$ if necessary), $u$ satisfies the assumptions of Lemmas \ref{L:RTa50} and \ref{L:RTa80}. We assume to fix ideas that $u$ is positive. Let $\ell$ be as in the two Lemmas, so that by \eqref{RTa51}, \eqref{RTa80},
 \begin{equation}
  \label{RTa91}
  \forall t\in \RR,\; \forall r\geq R+|t|,\quad \left|u(t,r)-\frac{\ell}{r^2}\right|\leq C\frac{\ell^3}{r^4},
 \end{equation}
 for some absolute constant $C>0$.
 Rescaling, we may assume that $\ell=8$, so that, by the assumption \eqref{gW} on $W$,  
 \begin{equation}
\label{RTa100}
\forall t\in \RR,\;\forall r\geq R+|t|,\quad \left|u(t,r)-W(r)\right|\leq \frac{K}{r^4},
\end{equation} 
for some absolute constant $K>0$. Let, for $\rho\geq R$,
$$M(\rho)=\sup_{\substack{t_0\in \RR\\ \sigma\geq \rho+|t_0|}} \sigma^2|u(t_0,\sigma)-W(\sigma)|.$$
By \eqref{RTa100}, 
\begin{equation}
 \label{RTa101}
\forall \rho\geq R,\quad M(\rho)\leq \frac{K}{\rho^2}.
 \end{equation} 
We will show that for $R'\geq R$ large enough,
\begin{equation}
\label{RTa101'}
\forall \rho\geq R', \quad M(2\rho)\geq \frac{1}{2}M(\rho).
\end{equation} 
Note that \eqref{RTa101} and \eqref{RTa101'} would imply, for $\rho>R'$ and $k\in \NN$,
$$ M(\rho)\leq 2^kM(2^{k}\rho)\leq K\frac{1}{2^k\rho^2}$$
and thus $M(\rho)=0$. Thus $u(t,r)=W(r)$ for $r\geq R'+|t|$, as desired. It remains to prove \eqref{RTa101'}.

We fix $t_0\in \RR$ and $\rho\geq R'$. We let 
$$ u(t)=W+h(t),\quad h_{\pm}(t)=\frac{h(t)\pm h(2 t_0-t)}{2}.$$
Using the equation $-\Delta W=\frac{1}{r^3} \Lambda(rW)$,  and the Taylor expansion
$$ \Lambda(w+\eps)=\Lambda(w)+\Lambda'(w)\eps+\Lambda_2(w,\eps)\eps^2,$$
where 
$$ \Lambda_2(w,\eps)=\int_0^1 \Lambda''(w+\theta\eps)(1-\theta)d\theta,\quad |\Lambda_2(w,\eps)|\lesssim |w| +|\eps|,$$
by \eqref{propLambda}.  As a consequence, denoting $h_{t_0}(t)=h(2t_0-t)$,
\begin{multline}
 \label{RTa102}
  \partial_t^2h_+-\Delta h_+\\
  =\frac{1}{2r^3}\Big\{\Lambda'(rW)(rh_+)+\Lambda_2(rW,rh)(rh)^2+\Lambda _2(rW,rh_{t_0})(rh_{t_0})^2\Big\},
\end{multline} 
with the initial data $ \vec{h}_{+\restriction t=0}=(u(t_0)-W,0)$.

We have 
\begin{equation}
 \label{SomeBndLmbd}
|\Lambda'(rW)|\lesssim r^2W^2,\quad |\Lambda_2(rW,rh)|\lesssim r|W|+r|h|.
 \end{equation} 
Let $\sigma\geq \rho+|t_0|$. We have 
$$r\geq \sigma +|t-t_0|\Longrightarrow r \geq \max(\rho+|t|,\rho+|2t_0-t|),$$
and thus 
\begin{equation}
 \label{RTa110}
 r\geq \sigma+|t-t_0|\Longrightarrow
|h_+(t,r)|+|h_{t_0}(t,r)|+|h(t,r)|\lesssim \frac{M(\rho)}{r^2}.
\end{equation} 
Combining \eqref{RTa102}, \eqref{SomeBndLmbd}, \eqref{RTa110} and the lower bound for the exterior energy (Lemma \ref{L:inhomog}), we obtain
\begin{multline}
 \label{truc}
\left\| \pi_{\sigma}^{\bot} h_+(t_0)\right\|_{\dot{H}^1(\sigma)} \\
\lesssim \sum_{k=1,2,3} \left\| W\right\|_{(L^3L^6)(\sigma+|t-t_0|)}^{3-k}\left( \int_{\RR} \left( \int_{\sigma+|t-t_0|}^{\infty} \frac{M(\rho)^{6}}{r^9}dr \right)^{\frac{1}{2}}dt \right)^{\frac k3}.
 \end{multline} 
Let $\eps(\sigma)= \left\| W\right\|_{(L^3L^6)(\sigma+|t|)}^{2}=\left\| W\right\|_{(L^3L^6)(\sigma+|t-t_0|)}^{2}$.
Note that $\eps(\sigma)$ is finite, independent of $t_0$, and goes to $0$ as $\sigma$ goes to infinity. The inequality \eqref{truc} implies
\begin{equation}
 \label{RTa111}
\left\|h_+(t_0,r)-\frac{\sigma^2}{r^2} h_+(t_0,\sigma)\right\|_{\hdot_r(\sigma)}\leq \frac{C\eps(\sigma)}{\sigma}M(\rho)+\frac{CM(\rho)^3}{\sigma^3},\quad \sigma \geq \rho+|t_0|.
\end{equation} 
By the radial Sobolev inequality, 
\begin{equation*}
\left|h_+(t_0,2\sigma)-\frac{\sigma^2}{(2\sigma)^2} h_+(t_0,\sigma)\right|\leq \frac{C\eps(\sigma)}{\sigma^2}M(\rho)+\frac{CM(\rho)^3}{\sigma^4}.
\end{equation*} 
As a consequence
\begin{multline*}
\frac{1}{4} |h_+(t_0,\sigma)|\leq \frac{C\eps(\sigma)}{\sigma^2}M(\rho)+\frac{CM(\rho)^3}{\sigma^4} +|h_+(t_0,2\sigma)|\\
\leq \frac{C\eps(\sigma)}{\sigma^2}M(\rho)+\frac{CM(\rho)^3}{\sigma^4} +\frac{M(2\sigma)}{4\sigma^2} , 
\end{multline*}
where we have used \eqref{RTa110} and the fact that $2\sigma \geq 2\rho+2|t_0|\geq 2\rho+|t_0|$, so that $|h_+(t_0,2\sigma)|\leq \frac{1}{(2\sigma)^2}M(2\sigma)$.

Going back to the definition of $h_+$, we deduce
$$ \sigma^2|u(t_0,\sigma)-W(\sigma)|\leq M(2\rho)+C\tilde{\eps}(\rho)M(\rho),$$
where $\tilde{\eps}(\rho)=\eps(\rho)+\frac{M(\rho)^2}{\rho^2}$ tends to $0$ as $\rho$ tends to infinity. 
This holds for all $t_0\in \RR$ and $\sigma \geq \rho+|t_0|$. Hence
$$M(\rho)\leq M(2\rho)+C\tilde{\eps}(\rho)M(\rho).$$
Taking $\rho$ large enough, we obtain \eqref{RTa101'}, concluding the proof.
\end{proof}

\section{Rigidity theorem, part II: odd solutions}
\label{S:odd}
In this section, we conclude the proof of Proposition \ref{P:rigidity}, and thus of the rigidity theorem (Theorem \ref{T:g_rigidity}) by the following proposition:
\begin{prop}
 \label{P:RTb10}
 Let $R> 0$ and $u$ a radial solution of \eqref{gNLW} on $\{r\geq R+|t|\}$, such that
 \begin{gather}
  \label{RTb10} (u,\partial_tu)_{\restriction t=0}=(0,u_1),\quad u_1\in L^2(R)\\
  \label{RTb11} \sum_{\pm}\lim_{t\to\pm\infty} \int_{|x|>R+|t|} |\nabla_{t,x}u(t,x)|^2dx=0.
 \end{gather}
Then
$$\forall t\in \RR,\; \forall r>R+|t|,\quad u(t,r)=0.$$
\end{prop}
We first check that Proposition \ref{P:rigidity} (and thus Theorem \ref{T:g_rigidity}) follows from Propositions \ref{P:RTa10} and \ref{P:RTb10}. Let $u$ be as in Proposition \ref{P:rigidity}. Then by Proposition \ref{P:RTa10}, one of the properties \eqref{RTa11} or \eqref{RTa12} holds. We must prove that on of the conclusions \eqref{conclu_rigidity1} or \eqref{conclu_rigidity2} of Proposition \ref{P:rigidity} hold. Property \eqref{RTa12} implies immediately \eqref{conclu_rigidity2}. If \eqref{RTa11} holds, we can assume, translating in time and taking a larger $R$, that $u_0(0,r)=0$ for all $r>R$. Thus $u$ satisfies (after a time translation), the assumptions of Proposition \ref{P:RTb10}, and the conclusion of this proposition implies  \eqref{conclu_rigidity1}.

We will prove Proposition \ref{P:RTb10} in \S \ref{SS:proof_odd_rig} after giving two preliminary results.
In \S \ref{SS:regul}, we prove a gain of regularity for a solution $u$ of \eqref{gNLW} satisfying the assumptions of this proposition. More precisely, we prove that for such a solution, $(\partial_tu,\partial_t^2u)(t)$ is in $\HHH(R'+|t|)$ for all $t$ (where $R'$ is large) and that $\partial_tu$ is non-radiative, in the sense that 
\begin{equation}
\label{NRdtu}
\sum_{\pm}\lim_{t\to\pm\infty} \int_{|x|>R'+|t|} |\nabla_{t,x}\partial_tu(t,x)|^2dx=0. 
\end{equation} 
In \S \ref{SS:approx}, we give an explicit, radial, smooth approximate solution $a$ of \eqref{gNLW} on $|x|>|t|$, which satisfies the non-radiative property 
\eqref{NRdtu} and whose initial data is of the form$(0,a_1(r))$, where $a_1\notin L^2(R)$ for $R>0$ (indeed $a_1(r)$ equals $\frac{1}{r^2(\log r)^{1/2}}$ up to a multiplicative constant). 

The proof of Proposition \ref{P:RTb10} in \S \ref{SS:proof_odd_rig} consists in showing that the initial data $u_1(r)$ of a solution as in Proposition \ref{P:RTb10} is of the same order as or greater than $a_1(r)$ for large $r$, contradicting the fact that $u_1\in L^2(R)$. This is done by carefully comparing $u_1$ with an appropriate rescaling of $a_1$, using again the improved lower energy bound of Lemma \ref{L:inhomog}.

\subsection{Gain of regularity}
\label{SS:regul}
In this subsection we prove:
\begin{lemma}
 \label{L:RTb30} Let $u$ be as in Proposition \ref{P:RTb10}. Then there exists a large constant $R'\geq R$ such that 
 \begin{gather}
  \label{RTb30} ru_1\in \dot{H}^1(R')\\
\label{RTb32} \forall \rho \geq R',\quad \frac{1}{\rho}\|u_1\|_{\dot{H}^1(\rho)} \approx |u_1(\rho)|
  \end{gather}
Furthermore, $(\partial_tu,\partial_t^2u)$ is the restriction to $\{r>R'+|t|\}$ of a $C^0(\RR,\dot{H}^1\times L^2)$ function and
  \begin{equation}
  \label{RTb31} \sum_{\pm}\lim_{t\to\pm \infty} \int_{|x|>R'+|t|}|\nabla_{t,x}\partial_t u|^2dx=0.
 \end{equation}
\end{lemma}
\begin{proof}[Proof of Lemma \ref{L:RTb30}]
 \noindent \emph{Step 1: gain of decay.} We let, for $\rho\geq R+|T|$, $T\in \RR$,
 \begin{equation}
  \label{RTb40}
  \eta(T,\rho)=\max_{0\leq t\leq T} \|\partial_tu(t)\|_{L^2(\rho)}.
 \end{equation} 
 In this step, we prove that if $R'$ is large enough and $\rho \geq R'+|T|$, then $r\nabla u(T) \in L^2(2\rho)$ and 
 \begin{equation}
  \label{RTb41}
  \left\|r\nabla u(T)\right\|_{L^2(2\rho)}\lesssim \sqrt{\int_{\rho}^{\infty} |u(T,r)|^2r^3dr}\lesssim T\eta(T,\rho).
 \end{equation} 
 Indeed, using that $u(t,r)=\int_0^{t} \partial_tu(\tau,r)d\tau$, we have, for all $t\geq 0$ and $\rho \geq R+|t|$, 
 \begin{equation}
  \label{RTb42}
  \|u(t)\|_{L^2(\rho)} \leq \int_0^{t} \|\partial_tu(\tau)\|_{L^2(\rho)}d\tau\leq |t|\eta (t,\rho).
 \end{equation}
 By Lemma \ref{L:RTa20}, for $\rho>R'+|T|$, $R'$ large, 
 \begin{equation}
 \label{Carlos}
 |u(T,\rho)|\approx \frac{1}{\rho}\|\nabla u(T)\|_{L^2(\rho)}.
 \end{equation} 
 Integrating, we obtain, for $\sigma >R'+|T|$,
 $$\int_{\sigma}^{+\infty}|u(T,\rho)|^2\rho^3d\rho\approx \int_{\sigma}^{+\infty} \int_{\rho}^{+\infty} |\partial_ru(T,r)|^2r^3dr\rho \,d\rho,$$
 and thus, by Tonelli's Theorem,
 $$ \int_{\sigma}^{\infty} |u(T,\rho)|^2\rho^3d\rho\approx \int_{\sigma}^{+\infty} (\partial_r u(T,r)|^2(r^5-\sigma^2r^3)dr.$$
 Combining with \eqref{RTb42}, we obtain, for $\sigma>R'+|t|$, 
 $$ \int_{2\sigma}^{\infty} |\partial_r u(T,r)|^2r^5dr\lesssim \int_{\sigma}^{\infty} |u(T,r)|^2r^3dr\lesssim T^2 \eta(T,\sigma)^2,$$
 which is exactly \eqref{RTb41}.
 
 \medskip
 
 \noindent\emph{Step 2. Gain of regularity.}
 In this step, we prove \eqref{RTb30} and \eqref{RTb32}. By \eqref{RTb41}, for $T\in [0,1]$, $\rho>R'+|T|$, 
 \begin{equation}
  \label{RTb50} 
  \int_{2\rho}^{\infty} r^5\left|\frac{1}{T} \partial_r u(T,r)\right|^2dr\lesssim \eta(T,\rho)^2\lesssim \eta(1,R')^2.
 \end{equation} 
 As a consequence, there exist $v_1\in L^2([2R',\infty),r^5 dr)$ and a sequence $\{T_n\}_n$ going to $0$ as $n$ goes to infinity such that 
 \begin{equation}
  \label{RTb51}
  \frac{1}{T_n}\partial_ru(T_n) \xrightharpoonup[n\to\infty]{} v_1 \text{ weakly in }L^2\left([2R',\infty), r^5dr\right).
 \end{equation} 
 We have 
 $$\frac{1}{T_n}u(T_n,2\rho)=-\frac{1}{T_n}\int_{2\rho}^{\infty}\partial_r u(T_n,r)dr,$$
 and thus, using the weak convergence of $\partial_ru(T_n)$,
 \begin{equation}
  \label{RTb60}
  \lim_{n\to\infty} \frac{1}{T_n}u(T_n,2\rho)=-\int_{2\rho}^{\infty} v_1(r)dr,
 \end{equation} 
 for all $\rho \geq R'$. Next, we observe that by \eqref{RTb50} and a radial Sobolev inequality,
 $$\frac{1}{T_n^2} |u(T_n,2\rho)|^2\lesssim \frac{1}{\rho^4} \eta(1,R')^2$$
 for large $n$. This yields by the dominated convergence theorem that the convergence \eqref{RTb60} also holds locally in $L^2(R',\infty)$. Since 
 $$\lim_{t\to 0^+} \frac{1}{t}u(t)\to u_1 \text{ in }L^2(2R),$$
 we deduce, by uniqueness of the limit that for $\rho>R'$
 \begin{equation}
  \label{RTb61}
  u_1(2\rho)=-\int_{2\rho}^{\infty} v_1(r)dr,
 \end{equation} 
where $\int_{R'}^{\infty} (v_1(r))^2r^5dr<\infty$. Hence \eqref{RTb30} (taking a larger $R'$). 

The estimate \eqref{RTb32} follows from \eqref{Carlos}, dividing by $T$ and letting $T$ go to $0$.

\medskip

\noindent\emph{Step 3. Proof that $\partial_tu$ is non-radiative.} In this step we prove \eqref{RTb31}.
We fix a small constant $\eps_0>0$ to be specified. Taking a larger $R'$ if necessary, we can assume $\|r\partial_r u_1\|_{L^2}+\|u_1\|_{H^1(R')}\leq \eps_0$. We define $\tu_1$ by 
$$\tu_1(r)=u_1(R')\text{ if }r<R',\quad \tu_1(r)=u_1(r)\text{ if }r\geq R'.$$
Using that by the radial Sobolev embedding 
$$|u_1(R')|\lesssim \frac{1}{(R')^2} \|r\partial_r u_1\|_{L^2(R')}\lesssim \frac{\eps_0}{(R')^2},$$ 
we obtain
\begin{equation}
 \label{smalltu}
 \|\tu_1\|_{H^1}\lesssim \eps_0.
\end{equation} 
Let $\tu$ be the solution of \eqref{gNLW} with initial data $(0,\tu_1)$ at $t=0$. Since $(0,\tu_1)\in (\dot{H}^2\cap \dot{H}^1)\times H^1)$, a standard persistence of regularity argument yields
$$\overrightarrow{\partial_t u}\in C^0\left(\RR, (\dot{H}^2\cap \dot{H}^1)\times H^1\right).$$
We observe that $\partial_t \tu$ satisfies the equation
\begin{equation}
 \label{RTb70}
 \partial_t^2(\partial_t\tu)-\Delta \partial_t\tu=\frac{1}{r^2} \Lambda'(r\tu)\partial_t\tu,
\end{equation} 
with initial data 
\begin{equation}
 \label{RTb71}
 \overrightarrow{\partial_t \tu}_{\restriction t=0} =(\tu_1,0)\in \dot{H}^1\times L^2.
\end{equation} 
Noting that 
$$ \frac{1}{r^2}\left|\Lambda'(r\tu)\right|\lesssim \tu^2,\quad  \left\| \tu^2\right\|_{(L^{3/2}L^3)}=\left\| \tu\right\|_{(L^{3}L^6)}^2\lesssim \eps_0,$$
we deduce by Strichartz estimates, and taking $\eps_0$ small enough,
\begin{equation}
 \label{RTb72}
 \|\partial_t\tu\|_{(L^3L^6)}\lesssim \|\tu_1\|_{\dot{H}^1}.
\end{equation} 
As a consequence, $\frac{1}{r^2}\Lambda'(r\tilde{u})
\partial_t \tu\in (L^1L^2)$. By finite speed of propagation, 
\begin{equation}
 \label{RTb73}
 r>R+|t|\Longrightarrow u(t,r)=\tu(t,r).
\end{equation} 
Thus the assumption \eqref{RTb11} implies
$$\sum_{\pm}\lim_{t\to\pm\infty}\int_{|x|>R+|t|} |\nabla_{t,x}\tu(t,x)|^2dx=0.$$
Combining with Proposition \ref{P:radiation} in the  appendix, we obtain \eqref{RTb31} for $\tu$ and thus, by \eqref{RTb73}, for $u$. 
%
%
\end{proof}

\subsection{Approximate non-radiative solution}
\label{SS:approx}

\begin{lemma}
\label{L:approx}
There exists a $C^{\infty}$ function $a(t,r)$ defined for $t\in \RR$, $r>2$ such that $\partial_t a\in C^0(\RR,\HHH(1))$ and
\begin{gather}
\label{cond_a1}
\forall t\in \RR,\; \forall r>2,\quad a(t,r)=-a(-t,r)\\
\label{cond_a2}
\forall r>2, \quad a_1(r):=\partial_ta(0,r)=\frac{2}{\sqrt{3}r^2(\log r)^{1/2}}\\
\label{cond_a3}
\forall t\in \RR,\; \forall r>\max(|t|,2), \quad |a(t,r)|+|t\partial_ta(t,r)|\lesssim \frac{|t|}{r^2(\log r)^{1/2}}\\
\label{cond_a4}
\sum_{\pm}\lim_{t\to \pm\infty} \int_{|x|>|t|} |\nabla_{t,x}\partial_t a(t,x)|^2dx=0
\end{gather}
and $b=\partial_t^2 a-\Delta a-\frac{1}{r^3}\Lambda(ra)$ satisfies
\begin{equation}
\label{cond_b}\forall t\in \RR,\; \forall r>\max(|t|,2), \quad |b(t,r)|+|t\partial_tb(t,r)|\lesssim \frac{|t|}{r^4(\log r)^{5/2}}.
\end{equation}
\end{lemma}
\begin{proof}
Let 
$$ a(t,r)=\frac{2t}{\sqrt 3 r^2(\log r)^{1/2}} -\frac{t^3}{3\sqrt{3} r^4(\log r)^{3/2}}.$$
It is easy to see that $a$ satisfies \eqref{cond_a1}, \eqref{cond_a2}, \eqref{cond_a3} and \eqref{cond_a4}. Using the formula
\begin{equation}
\label{laplace_log}
\Delta \left( \frac{1}{r^{\alpha}(\log r)^{\beta}} \right)=\frac{1}{r^{\alpha+2}} \left( \frac{\alpha(\alpha-2)}{(\log r)^{\beta}}+\frac{(2\alpha-2)\beta}{(\log r)^{\beta+1}}+\frac{\beta(\beta+1)}{(\log r)^{\beta+2}} \right) 
\end{equation} 
we can prove that $b$ satisfies \eqref{cond_b}. Indeed by \eqref{propLambda}, 
\begin{gather*}
\left| a^3-\frac{1}{r^3} \Lambda(ra)\right|\lesssim r^2|a|^5\lesssim \frac{|t|}{r^{4}(\log r)^{5/2}},\quad r>\max(|t|,2),\\
\left| \partial_t\left(a^3-\frac{1}{r^3} \Lambda(ra)\right)\right|\lesssim r^2a^4|\partial_t a|\lesssim \frac{1}{r^{4}(\log r)^{5/2}},\quad r>\max(|t|,2),
\end{gather*}
and 
$\partial_t^2a-\Delta a-a^3$ is a linear combination of:
\begin{gather*}
\frac{t}{r^4 (\log r)^{5/2}},\quad \frac{t^3}{r^6 (\log r)^{5/2}},\quad \frac{t^3}{r^6 (\log r)^{7/2}},\quad \frac{t^5}{r^8 (\log r)^{5/2}}\\
\frac{t^7}{r^{10} (\log r)^{7/2}},\quad \frac{t^9}{r^{12} (\log r)^{9/2}}.
\end{gather*}
The key point is the double (approximate) cancellation given by:
$$\partial_t^2\left( \frac{t^3}{3\sqrt{3} r^4(\log r)^{3/2}} \right)=-\Delta \left( \frac{2t}{\sqrt 3 r^2(\log r)^{1/2}} \right)+l.o.t$$
and
$$\Delta \left( \frac{t^3}{3\sqrt{3} r^4(\log r)^{3/2}} \right)=\left( \frac{2t}{\sqrt 3 r^2(\log r)^{1/2}} \right)^3+l.o.t.,$$
where $l.o.t.$ is dominated by $t/r^4(\log r)^{5/2}$ for $r\geq |t|$.
The nonlinearity appears in the last line: the approximate solution $a$ is not the perturbation of a linear solution, but a fully nonlinear object.
\end{proof}

\begin{remark}
 An explicit computation shows that 
$$\indic_{\{|x|>1+|t|\}} a \in L^3L^6.$$
 However $\partial_ta(0)$ barely fails to be in $L^2$. This fact will be crucial in the proof of Proposition \ref{P:RTb10}.
 \end{remark}
\subsection{Proof of the rigidity result}
\label{SS:proof_odd_rig}
We are now in position to prove Proposition \ref{P:RTb10}. 

\emph{Step 1. Preliminaries.}
We argue by contradiction, assuming that $u_1$ is not identically $0$ for $r>R$. By \eqref{RTb32}  in Lemma \ref{L:RTb30}, $u_1$ does not change sign, and we can assume that $u_1(r)>0$ for large $r$. Furthermore, by \eqref{RTb30} in Lemma \ref{L:RTb30}, and the radial Sobolev inequality,
\begin{equation}
 \label{RTc12}
\lim_{r\to\infty} r^2u_1(r)=0.
\end{equation} 
We fix $\rho\gg 1$ and choose $\lambda=\lambda(\rho)$ such that $u_1(\rho)=\partial_t a_{\lambda}(0,\rho)$, where $a$ is the non-radiative  approximate solution given by Lemma \ref{L:approx} and 
 $$a_{\lambda}(t,r)=\lambda a(\lambda t,\lambda r).$$ In other words,
 \begin{equation}
  \label{RTc13}
  u_1(\rho)=\frac{2}{\sqrt{3} \rho^2\log(\rho\lambda)^{1/2}}.
 \end{equation} 
Note that the equation \eqref{RTc13} always has a unique solution $\lambda(\rho)>1/\rho$, and that by \eqref{RTc12} and \eqref{RTc13}, 
\begin{equation}
 \label{RTc20}
 \lim_{\rho\to\infty} \rho \lambda(\rho)=\infty.
\end{equation} 
We also note that by \eqref{cond_a3}, 
\begin{equation}
\label{bound_alambda}
r>\max(|t|,2/\lambda)\Longrightarrow |a_{\lambda}(t,r)|+|t\partial_ta_{\lambda}(t,r)|\lesssim \frac{|t|}{r^2\log^{1/2}(r\lambda)}.
\end{equation} 
We have 
\begin{equation}
 \label{RTc21}
 \partial_t^2 a_{\lambda}-\Delta a_{\lambda}-\frac{1}{r^3}\Lambda(r a_{\lambda})=b_{[\lambda]},
\end{equation} 
where $b_{[\lambda]}(t,r)=\lambda^3 b(\lambda t,\lambda r)$ satisfies, by \eqref{cond_b},
\begin{equation}
 \label{RTc22}
r>\max(|t|,2/\lambda)\Longrightarrow  \left| b_{[\lambda]}(t,r)\right|+\left|t\partial_t b_{[\lambda]}(t,r)\right|\lesssim \frac{|t|}{r^4\log(\lambda r)^{5/2}}.
\end{equation} 
Let 
$w =\partial_t a_{\lambda}-\partial_t u$,
so that $w$ satisfies the equation 
\begin{equation}
 \label{eq_w}
 \partial_t^2w-\Delta w=\frac{1}{r^2}\Lambda'(r a_{\lambda})\partial_ta_{\lambda} -\frac{1}{r^2} \Lambda'(r u)\partial_tu+\partial_tb_{[\lambda]}
\end{equation} 
with initial data $w(0)=(\partial_t a_{\lambda}(0)-u_1,0)$.
By explicit computation, using \eqref{RTc22},
\begin{equation}
 \label{estim_blambda}
 \|\partial_t b_{[\lambda]}\|_{(L^1L^2)(\rho)} \lesssim \frac{1}{\rho(\log(\lambda \rho))^{5/2}}.
\end{equation} 
\emph{Step 2. Channels of energy.} Fix a small $\eps_0>0$, to be specified later. In this step we prove that if $\rho$ is chosen large enough, 
\begin{equation}
 \label{RTc23}
 \left\| u_1-\partial_ta_{\lambda}(0)\right\|_{\hdot(\rho)}\lesssim \frac{1}{\rho (\log(\lambda \rho))^{5/2}}.
\end{equation} 
By the equation \eqref{eq_w}, Strichartz inequalities and finite speed of propagation,
\begin{multline}
 \label{RTd50}
 \|w\|_{(L^2L^8)(\rho)} +\|w\|_{(L^3L^6)(\rho)} +\sup_{t\in\RR}\left\|\vec{w}(t)\right\|_{\HHH(\rho+|t|)}\\
 \lesssim \|w(0)\|_{\hdot(\rho)}+\left\| \frac{1}{r^2}\Lambda'(r a_{\lambda})\partial_ta_{\lambda} -\frac{1}{r^2} \Lambda'(r u)\partial_tu+\partial_tb_{[\lambda]}\right\|_{(L^1L^2)(\rho)}.
\end{multline}
By the exterior energy lower bound of Lemma \ref{L:inhomog} (which also applies to linear solutions outside a wave cone by finite speed of propagation)
\begin{equation}
 \label{RTd41}
 \|\pi_{\rho}^{\bot}w(0)\|_{\hdot(\rho)} \lesssim \left\| \frac{1}{r^2}\Lambda'(r a_{\lambda})\partial_ta_{\lambda} -\frac{1}{r^2} \Lambda'(r u)\partial_tu+\partial_tb_{[\lambda]}\right\|_{(L^1L^2)(\rho)}.
\end{equation} 
Since by the choice of $\lambda$, $w(0,\rho)=0$, we obtain $w(0)=\pi_{\rho}^{\bot}w(0)$. Combining with \eqref{estim_blambda}, \eqref{RTd50} and \eqref{RTd41}, we deduce
\begin{multline}
\label{interm_w}
 \|w\|_{(L^2L^8)(\rho)} +\|w\|_{(L^3L^6)(\rho)} +\left\|w_0\right\|_{\hdot(\rho)}\\
 \lesssim \left\| \frac{1}{r^2}\Lambda'(r a_{\lambda})\partial_ta_{\lambda} -\frac{1}{r^2} \Lambda'(r u)\partial_tu\right\|_{(L^1L^2)(\rho)}+ \frac{1}{\rho(\log(\lambda \rho))^{5/2}},
\end{multline} 
and we are left with bounding from above the first term in the second line of \eqref{interm_w}. We have
\begin{equation}
\label{decomposition}
\frac{1}{r^2}\Lambda'(r a_{\lambda})\partial_ta_{\lambda} -\frac{1}{r^2} \Lambda'(r u)\partial_tu=\frac{1}{r^2}\Lambda'(ru)w+\frac{1}{r^2}\left( \Lambda'(ra_{\lambda})-\Lambda'(ru) \right)\partial_ta_{\lambda}.
\end{equation}
By the bound $|\Lambda'(U)|\lesssim U^2$ and H\"older's inequality,
\begin{equation}
 \label{RTd51}
 \left\| \frac{1}{r^2} \Lambda'(r u)w\right\|_{(L^1L^2)(\rho)}\lesssim \|w\|_{(L^3L^6)(\rho)} \|u\|^2_{(L^3L^6)(\rho)}\leq \frac{\eps_0}{3} \|w\|_{(L^3L^6)(\rho)}
\end{equation} 
for large $\rho$. Next, we have, by the bound $|\Lambda''(U)|\lesssim |U|$ which implies $|\Lambda'(U)-\Lambda'(V)|\lesssim |U-V|\left( |U|+|V| \right)$,
\begin{equation}
\label{RTdLipschitz}
\frac{1}{r^2}\left| \Lambda'(ra_{\lambda})-\frac{1}{r^2} \Lambda'(ru)\right|\lesssim |u||u-a_{\lambda}|+|a_{\lambda}| |u-a_{\lambda}|.
\end{equation} 
We will bound separately $\left\|u(u-a_{\lambda})\partial_ta_{\lambda}\right\|_{(L^1L^2)(\rho)}$ and 
$\left\|a_{\lambda}(u-a_{\lambda})\partial_ta_{\lambda}\right\|_{(L^1L^2)(\rho)}$;
By H\"older inequality,
$$ \left\|u\partial_ta_{\lambda}(u-a_{\lambda})\right\|_{(L^1L^2)(\rho)} \lesssim \|u\|_{(L^2L^8)(\rho)} \|\partial_ta_{\lambda}(u-a_{\lambda)}\|_{(L^2L^{8/3})(\rho)}.$$
Furthermore, using the bound \eqref{bound_alambda}, and the fact that $w=\partial_t (a_{\lambda}-u)$, we obtain for large $\rho$
$$\left|\partial_t a_{\lambda}(u-a_{\lambda})\right| \lesssim \frac{1}{r^2\log(r\lambda)^{1/2}}\left|\int_0^{t} w(\tau,r)d\tau\right|\lesssim
\frac{\eps_0}{r^2}\left|\int_0^{t} w(\tau,r)d\tau\right|,$$
where we have used \eqref{RTc20}. By Lemma \ref{L:Carlos2} in Appendix \ref{A:anal_harm}, 
$$\left\|\partial_t a_{\lambda}(u-a_{\lambda})\right\|_{(L^2L^{8/3})(\rho)}\lesssim \eps_0\|w\|_{(L^2L^8)(R)},$$
and thus
\begin{equation}
\label{RTd60}
\left\|u\partial_ta_{\lambda} (u-a_{\lambda})\right\|_{(L^1L^2)(\rho)}\lesssim \eps_0\|w\|_{(L^2L^8)(\rho)}.
\end{equation}
Using again the bound \eqref{bound_alambda} on $a_{\lambda}$ and $\partial_ta_{\lambda}$, we obtain
$$\left|a_{\lambda}\partial_ta_{\lambda} (u-a_{\lambda})\right|\lesssim \frac{|t|}{r^4\log(\lambda r)}\left|\int_0^t w(\tau,r)d\tau\right|.$$
By Lemma \ref{L:Carlos1} in Appendix \ref{A:anal_harm}, 
\begin{equation}
 \label{RTd61}
 \left\|a_{\lambda}\partial_ta_{\lambda}(u-a_{\lambda})\right\|_{(L^1L^2)(\rho)}\lesssim \frac{1}{(\log (\lambda \rho))^{1/3}} \|w\|_{(L^3L^6)(\rho)}.
\end{equation} 
Combining \eqref{interm_w}, \eqref{decomposition}, \eqref{RTd51},  \eqref{RTd60} and \eqref{RTd61}, we deduce that for large $\rho$,
 \begin{multline}
 \|w\|_{(L^2L^8)(\rho)} +\|w\|_{(L^3L^6)(\rho)} +\left\|w_0\right\|_{\hdot(\rho)}
 \lesssim \\
 \eps_0\left(\|w\|_{(L^2L^8)(\rho)} +\|w\|_{(L^3L^6)(\rho)}  \right)+\frac{1}{\rho(\log(\lambda \rho))^{5/2}}.
 \end{multline}
 This implies \eqref{RTc23} if $\eps_0$ is small enough.
 
 \medskip

 \noindent \emph{Step 3. Differential inequality.} 
 In this step, we deduce from \eqref{RTc23} the following inequality for large $\rho$:
 \begin{equation}
 \label{RTc50}
 \left| \|u_1\|_{L^4(\rho)}^4 -\frac{\rho^4u_1(\rho)^4}{4}\right|\lesssim \rho^{8} u_1(\rho)^6.
\end{equation}
Recall that $\rho^2u_1(\rho)$ is small when $\rho$ is large (see \eqref{RTc12}), so that $ \rho^8u_1(\rho)^6\ll \rho^4u_1(\rho)^4$.
The inequality \eqref{RTc50} can be seen as an approximate differential equation on the $L^4$ norm of $u$. We will deduce a contradiction from this inequality in the next step.

By \eqref{RTc23} and the critical Sobolev inequality, we obtain, for $\rho$ large
\begin{equation}
 \label{RTc31}
 \left\|u_1-\partial_t a_{\lambda}(0)\right\|_{L^4(\rho)} \lesssim \frac{1}{\rho (\log(\lambda \rho))^{5/2}}. 
\end{equation}
Using the explicit value of $\partial_ta_{\lambda}(0)$, we have
\begin{equation*}
 \left\|\partial_ta_{\lambda}(0)\right\|_{L^4(\rho)}^4 =\left( \frac{2}{\sqrt{3}} \right)^4\int_{\rho}^{\infty} \frac{dr}{r^5 \log^2(r\lambda)},
\end{equation*}
which yields, after an integration by parts,
\begin{equation}
 \label{RTc40}
 \left\| \partial_ta_{\lambda}\right\|_{L^4(\rho)}^4=\frac{4}{9\rho^4\log^2(\rho\lambda)}+\mathcal{O}\left( \frac{1}{\rho^4\log^3(\rho\lambda)} \right).
\end{equation} 
By \eqref{RTc31}, $\left| \|u_1\|_{L^4(\rho)}-\|\partial_ta_{\lambda}(0)\|_{L^4(\rho)}\right|\lesssim \frac{1}{\rho \log(\lambda\rho)^{5/2}}$. Thus 
$$\|u_1\|_{L^4(\rho)} \sim \sqrt{\frac{2}{3}} \frac{1}{\rho \left(\log(\rho\lambda)\right)^{1/2}}$$
and we deduce from \eqref{RTc31}
\begin{equation}
 \label{RTc41}
 \left|\|u_1\|_{L^4(\rho)}^4-\frac{4}{9\rho^4 \log^2(\rho\lambda)}\right|\lesssim \frac{1}{\rho^4\log^3(\rho\lambda)}.
\end{equation} 
Going back to the definition of $\lambda$ (see \eqref{RTc13}), we deduce \eqref{RTc50}).

\medskip

\noindent\emph{Step 4. Lower bound on $u_1$ and contradiction.}

In this step, we deduce from \eqref{RTc50} and the fact that $\rho^2u_1(\rho)$ goes to $0$ that for some constant $C>0$, 
\begin{equation}
 \label{RTc51}
 u_1(\rho)\geq  \frac{1}{C\rho^2(\log\rho)^{1/2}},\quad \rho\to\infty,
\end{equation} 
a contradiction since $\rho\mapsto \frac{1}{\rho^2(\log\rho)^{1/2}}$ is not in $L^2(r^3dr)$. 

Let 
$$f(\rho)=\|u_1\|^4_{L^4(\rho)}.$$
Since by \eqref{RTc50}, $f(\rho)\sim \frac{\rho^4 u_1(\rho)^4}{4}$, we have
\begin{equation}
 \label{RTc52}
 \lim_{\rho\to\infty}\rho^4f(\rho)=0. 
\end{equation} 
Also, we can rewrite \eqref{RTc50} as
\begin{equation}
 \label{RTc60} 
 \left|f(\rho)+\frac{\rho}{4}f'(\rho)\right|\lesssim f(\rho)^{3/2}\rho^2.
\end{equation} 
Next, we let $g(\rho)=\rho^4f(\rho)$ and note that by \eqref{RTc52}, 
\begin{equation}
 \label{RTc61} 
 \lim_{\rho\to\infty} g(\rho)=0.
\end{equation} 
Since $f'(\rho)=-\frac{4}{\rho^5}g(\rho)+\frac{1}{\rho^4}g'(\rho)$, the estimate \eqref{RTc60} reads
\begin{equation}
 \label{RTc63}
 \left|\frac{g'(\rho)}{g(\rho)^{3/2}}\right|\lesssim \frac{1}{\rho}.
\end{equation} 
Integrating between $\rho_0$ and $\rho>\rho_0$, where $\rho_0$ is large, we deduce that for large $\rho$,
$$\frac{1}{g(\rho)^{1/2}} \leq C\log \rho,$$
where $C$ is a large positive constant (depending on $g$). This implies that for large $\rho$
$$g(\rho)\geq \frac{1}{C(\log \rho)^2}.$$
Since $g(\rho)=\rho^4f(\rho)$ and by \eqref{RTc50}, $f(\rho)\approx \rho^4u_1(\rho)^4$, we deduce \eqref{RTc51} as announced.
\qed

\section{Soliton resolution for the wave equation}
\label{S:resolution}
In this Section we prove Theorem \ref{T:main}. We start with some preliminaries on profile decomposition for the wave equations \eqref{FW} and \eqref{NLW}.
\subsection{Profile decomposition}
\label{SS:profile}
Let $\big\{(u_{0,n},u_{1,n})\big\}_n$ be a bounded sequence of radial functions in $\HHH$. We say that it admits a profile decomposition if for all $j\geq 1$, there exist a solution $U^j_L$ to the free wave equation with initial data in $\HHH$ and sequences of parameters $\{\lambda_{j,n}\}_n\in (0,\infty)^{\NN}$, $\{t_{j,n}\}_n\in \RR^{\NN}$ such that
\begin{equation}
 \label{psdo_orth}
 j\neq k \Longrightarrow \lim_{n\to\infty}\frac{\lambda_{j,n}}{\lambda_{k,n}}+\frac{\lambda_{k,n}}{\lambda_{j,n}}+\frac{|t_{j,n}-t_{k,n}|}{\lambda_{j,n}}=+\infty,
\end{equation} 
and, denoting 
\begin{gather}
 \label{rescaled_lin}
 U^j_{L,n}(t,r)=\frac{1}{\lambda_{j,n}}U^j_L\left( \frac{t-t_{j,n}}{\lambda_{j,n}},\frac{r}{\lambda_{j,n}} \right),\quad j\geq 1\\
 w_{n}^J(t)=S_L(t)(u_{0,n},u_{1,n})-\sum_{j=1}^J U^j_{L,n}(t),
\end{gather} 
one has 
\begin{equation}
 \label{wnJ_dispersive}
 \lim_{J\to\infty}\limsup_{n\to\infty}\|w_n^J\|_{L^3L^6}=0.
\end{equation} 
We recall (see \cite{BaGe99}, \cite{Bulut10}) that any bounded sequence in $\HHH$ has a subsequence that admits a profile decomposition. The properties above imply that the following weak convergences hold:
\begin{equation}
 \label{wlim_w}
 j\leq J\Longrightarrow \left( \lambda_{j,n}w_n^J\left(t_{j,n},\lambda_{j,n}\cdot \right),\lambda_{j,n}^{2}\partial _tw_n^J\left(t_{j,n},\lambda_{j,n}\cdot \right)\right) \xrightharpoonup[n\to\infty]{} 0 \text{ in }\HHH.
\end{equation} 
Furthermore we have the Pythagorean expansions valid, for all $J\geq 1$
\begin{align}
 \label{Pyt1}
\|(u_{0,n},u_{1,n})\|^2_{\HHH}&=\sum_{j=1}^J \left\|\vec{U}^j_L(0)\right\|^2_{\HHH}+\left\| \vec{w}_n^J(0)\right\|^2_{\HHH}+o_n(1)\\
\label{Pyt2}
E(u_{0,n},u_{1,n})&=\sum_{j=1}^J E(\vec{U}_{L,n}^j(0))+E(\vec{w}_n^J(0))+o_n(1).
\end{align}

If $\{(u_{0,n},u_{1,n})\}_n$ admits a profile decomposition, we can assume, extracting subsequences and time-translating the profiles if necessary, that the following limits exist:
$$\lim_{n\to\infty}\frac{-t_{j,n}}{\lambda_{j,n}}=\tau_j\in \{-\infty,0,\infty\}.$$
If $\tau_j=0$, we will assume without loss of generality $t_{j,n}=0$ for all $n$.
If $\tau_j=0$, we define $U^j$ as  the unique solution to the nonlinear wave equation \eqref{NLW} such that
\begin{equation}
\label{NL_profiles}
\vec{U}^j(0)=\vec{U}^j_L(0). 
\end{equation} 
If $\tau_j\in \{\pm\infty\}$, we simply let $U^j=U_L^j$. We denote by $U^j_n$ the rescaled profile:
$$ U^j_n(t,r)=\frac{1}{\lambda_{j,n}}U^j\left( \frac{t-t_{j,n}}{\lambda_{j,n}},\frac{r}{\lambda_{j,n}} \right).$$

We will now state a superposition principle (sometime called \emph{nonlinear profile decomposition}) outside wave cones using the profiles $U^j$. We denote by $\JJJ$ the set of indices $j\geq 1$ such that $\tau_j=0$ and the solution $U^j$ of \eqref{NLW} with initial data $(U_0^j,U_1^j)$ cannot be extended to a solution with finite $L^3L^6$ norm to the wave cone $\{|x|>|t|\}$. By the small data theory and \eqref{Pyt1}, $\JJJ$ is finite. If $j\notin \JJJ$, then the solution $U_n^j(t,x)$ is well defined for $\{|x|>|t|\}$ and 
\begin{equation}
\label{unif_bound}
\sup_{n\to\infty} \|U_n^j\|_{(L^3L^6)(\{|x|>|t|\})}<\infty. 
\end{equation}

If $\JJJ$ is empty, we let $R_n=0$ for all $n$. If not,  after extraction, \eqref{psdo_orth} implies that there is a unique $j_0\in \JJJ$ such that
$$\forall j\in \JJJ\setminus \{j_0\}, \quad \lim_{n\to\infty}\frac{\lambda_{j,n}}{\lambda_{j_0,n}}=0.$$
In this case, we fix a $R>0$ such that $U^{j_0}$ is defined (with a finite $L^3L^6$ norm) in $\{|x|>R+|t|\}$, and we let $R_n=R \lambda_{j_0,n}$ for all $n$.
\begin{prop}
 \label{P:NL_profile}
 Let $\{(u_{0,n},u_{1,n})\}_n$, $U^j_n$, $w_n^J$, $R_n$ be a above. Then for large $n$, there is a solution $u_n$ of \eqref{NLW} defined in $\{r>R_n+|t|\}$ with initial data $\{(u_{0,n},u_{1,n})\}_n$ at $t=0$. Furthermore, denoting, for $J\geq 1$, $r>R_n+|t|$,
 $$\epsilon_n^J(t,r)=u_n(t,r)-\sum_{j=1}^J U_n^j(t,r)-w_n^J(t,r),$$
 we have 
 $$\lim_{J\to\infty} \lim_{n\to\infty}\left[\|\epsilon_n^J\|_{(L^3L^6)(R_n)}+\sup_{t\geq 0}\left\|\vec{\epsilon}_n^J(t)\right\|_{\HHH(t+R_n)}\right]=0.$$
\end{prop}
\begin{proof}
The proof uses
\begin{gather*}
\forall j\geq 1,\quad \limsup_{n\to\infty}\|U_n^j\|_{(L^3L^6)(R_n)}<\infty\\
\forall j,k,\ell\geq 1,\quad j\neq k\Longrightarrow \left\|U_n^jU_n^kU_n^{\ell}\right\|_{(L^1L^2)(R_n)}=0,
\end{gather*}
and long-time perturbation theory arguments.

We omit the details, since they are similar to the proof when the solution is not restricted to the exterior of a wave cone (see e.g. the Main Theorem p. 135 in \cite{BaGe99} in the defocusing case, and a sketch of proof in the focusing case right after Proposition 2.8 in \cite{DuKeMe11a}). The proof is also similar to the one of the corresponding result for wave maps which is detailed below. 
\end{proof}
We next state a pseudo-orthogonality property of the profiles that will be needed in the proof of Theorem \ref{T:main}. See Appendix \ref{A:psdo_ortho} for the proof.
\begin{lemma}
\label{L:psdo_ortho}
  Let $\{(u_{0,n},u_{1,n})\}_n$, $U^j_n$, $w_n^J$, $R_n$ be a above. Let $\{s_n\}_n$ be a sequence of times, $\rho_n,\rho'_n>0$ be such that $R_n+|s_n|\leq \rho_n<\rho'_n$ ($\rho_n'=\infty$ is allowed). Then
  \begin{align}
   \label{Ps10}
   j\neq k&\Longrightarrow \lim_{n\to\infty} \int_{\rho_n<|x|<\rho_n'}\nabla_{t,x}U_n^j(s_n,x)\cdot \nabla_{t,x}U_n^k(s_n,x)dx=0\\
  \label{Ps11}
   j\leq J&\Longrightarrow \lim_{n\to\infty} \int_{\rho_n<|x|<\rho_n'}\nabla_{t,x}U_n^j(s_n,x)\cdot \nabla_{t,x}w_n^J(s_n,x)dx=0.
  \end{align}
\end{lemma}

\subsection{Finite time blow-up case}
In this subsection we deduce the finite time blow-up case of Theorem \ref{T:main} from the rigidity theorem, Theorem \ref{T:rigidity}. The proof is inspired by the one of the analogous result in space dimension 3 (see \cite{DuKeMe13}). However, unlike in space dimension 3, we will need the fact (proved in \cite{CoKeLaSc18}) that the soliton resolution holds for a sequence of times.

We consider a solution $u$ of \eqref{NLW}, with initial data \eqref{ID} such that $T_+=T_+(u)$ is finite and 
\begin{equation}
\label{liminf}
 \liminf_{t\overset{<}{\longrightarrow} T_+}\left\|\vec{u}(t)\right\|_{\HHH}<\infty
\end{equation} 
\subsubsection{Convergence of the solution}
We first recall that the solution $u$ converges outside the wave cone:
\begin{lemma}
\label{L:v0v1}
 There exists $(v_0,v_1)\in \HHH$ such that 
$$\lim_{t\to T_+}\int_{|x|>|T_+-t|} |\nabla (u(t,x)-v_0(x))|^2+(\partial_tu(t,x)-v_1(x))^2dx=0.$$
\end{lemma}
Lemma \ref{L:v0v1} is identical to Lemma 4.1 in \cite{DuKeMe13}. We refer to this article for the proof. Note that \cite[Lemma 4.1]{DuKeMe13} is stated in space dimension $3$, however the proof is independent of the dimension.
\subsubsection{Analysis along a sequence of times}
The core of the proof of Theorem \ref{T:main} in the finite time blow-up case is the following proposition, that we will deduce from Theorem \ref{T:rigidity}:
\begin{prop}
 \label{P:subsequence}
Let $\{t_n\}_n$ be a sequence of times such that 
\begin{gather*}
\lim_{n\to\infty} t_n=T_+,\quad \forall n,\; 0\leq t_n<T_+\\
\limsup_{n\to\infty} \|\vec{u}(t_n)\|_{\HHH}<\infty.
\end{gather*}
Then there exists a subsequence of $\{t_n\}_n$ (that we still denote by $\{t_n\}_n$), $J\geq 0$, $(\iota_j)_j\in \{\pm 1\}^J$, sequences $\{\lambda_{j,n}\}_n$, $1\leq j\leq J$ with 
$$\lambda_{1,n}\ll \ldots \ll \lambda_{J,n}\ll T_+-t_n,$$
such that
\begin{gather}
\label{CVH1}
\lim_{n\to \infty} \left\|u(t_n)-v_0 -\sum_{j=1}^J \frac{\iota_j}{\lambda_{j,n}} W\left( \frac{\cdot}{\lambda_{j,n}} \right)\right\|_{\dot{H}^1}=0\\
\label{dev_energy}
E(u_0,u_1)=E(v_0,v_1)+ J E(W,0)+ \frac{1}{2}\left\|\partial_t u(t_n)-v_1\right\|_{L^2}^2+o_n(1),
\end{gather}
where $(v_0,v_1)$ are given by Lemma \ref{L:v0v1}.
\end{prop}
Proposition \ref{P:subsequence} is the analog in space dimension 4 of Proposition 4.2 of \cite{DuKeMe13}. However the conclusion \eqref{CVH1} of Proposition \ref{P:subsequence} is weaker than the one of Proposition 4.2 of \cite{DuKeMe13}, which states in addition that $\partial_tu(t_n)-v_1$ goes to $0$ as $n$ goes to infinity in $L^2$. This weaker conclusion is due to the fact that in space dimension $4$, the lower bound for the asymptotic linear energy \eqref{CKS} is weaker than its analog in space dimension $3$.

\begin{proof}
To simplify notations, we assume without loss of generality in all the proof that
$$T_+(u)=1.$$
\noindent\emph{Step 1.} We first prove that we can assume
\begin{equation}
 \label{v0v1small}
 \|(v_0,v_1)\|_{\HHH}\leq \eps,
\end{equation} 
where $\eps$ is a fixed small constant. Let $v$ be the solution of \eqref{NLW} such that $\vec{v}(1)=(v_0,v_1)$. By finite speed of propagation and the definition of $(v_0,v_1)$, $v(t,r)=u(t,r)$ for $r>1-t$, $t<1$ close to $1$. Thus
\begin{multline*}
\int_{1-t}^{3-3t} \left( (\partial_{t,r}u(t,r))^2+\frac{1}{r^2}(u(t,r))^2 \right)r^3dr\\=\int_{1-t}^{3-3t} \left( (\partial_{t,r}v(t,r))^2+\frac{1}{r^2}(v(t,r))^2 \right)r^3dr\underset{t\to 1}{\longrightarrow}0. 
\end{multline*}
Hence, we can choose $t_0<1$ close to $1$ such that 
\begin{equation}
 \label{Add10}
 t_0<t<1\Longrightarrow \int_{1-t}^{3-3t} \left( (\partial_{t,r}u(t,r))^2+\frac{1}{r^2}(u(t,r))^2 \right)r^3dr\leq \eps^2.
\end{equation} 
We fix $(\tu_0,\tu_1)\in \HHH$ such that 
\begin{gather}
 \label{tu_u_egal}
 r\leq 3-3t_0\Longrightarrow (\tu_0,\tu_1)(r)=(u(t_0,r),\partial_t u(t_0,r))\\
 \label{Add11}
 \int_{1-t_0}^{\infty} \left( (\partial_{r}\tu_0(r))^2+(\tu_1(r))^2+\frac{1}{r^2}(\tu_0(r))^2 \right)r^3dr\lesssim \eps^2.
\end{gather} 
Let $\tu$ be the solution of \eqref{NLW} such that $(\tu(t_0),\partial_t\tu(t_0))=(\tilde{u}_0,\tilde{u}_1)$. By \eqref{tu_u_egal} and finite speed of propagation,
\begin{equation}
 \label{Add12}
 \tilde{u}(t,r)=u(t,r),\quad r\leq 3-2t_0-t,\quad t_0<t<1.
\end{equation} 
By finite speed of propagation, small data theory and \eqref{Add11}, 
\begin{equation}
 \label{Add13}
 \left\|(\tilde{u},\partial_t\tilde{u})\right\|_{\HHH(1-2t_0+t)}\lesssim \eps,\quad t_0<t<1.
\end{equation} 
Combining \eqref{Add10}, \eqref{Add12} and \eqref{Add13}, we see that 
$$\|(\tilde{u},\partial_t\tilde{u})(t)\|_{\HHH(1-t)}\lesssim \eps,\quad t_0<t<1.$$
Replacing $u$ by $\tu$, we obtain a solution of \eqref{NLW} that blows up at $t=1$, has not changed in $\{r<1-t\}$ and satisfies the additional condition \eqref{v0v1small}. Note that the fact that $u=\tilde{u}$ in $\{r<1-t\}$ implies (defining $\tv$ in the same way as $v$ with $u$ replaced by $\tu$)
$$\lim_{t\to 1} \big\|(u(t)-v(t))-(\tilde{u}(t)-\tilde{v}(t)),\partial_tu(t)-\partial_tv(t))-(\partial_t\tilde{u}(t)-\partial_t\tilde{v}(t))\big\|_{\HHH}=0,$$
so that if \eqref{CVH1}, \eqref{dev_energy} hold for $\tu$, then it holds for the original $u$.

\smallskip

\noindent\emph{Step 2. Profile decomposition.} Extracting subsequences, we assume that $\vec{u}(t_n)-(v_0,v_1)$ admits a profile decomposition as in \S \ref{SS:profile}. We use the same notations $U^j_L$, $U^j$, $\{\lambda_{j,n}\}_n$, $\{t_{j,n}\}_n$, $w_n^J$ as in \S \ref{SS:profile}. Extracting subsequences again and rescaling the profiles, we can assume that there is a partition $\NN\setminus\{0\}=\III_s\cup \III_c\cup\III_+\cup \III_-$ of the set of indices such that 
\begin{align*}
 j\in \III_s&\iff \forall n,\; t_{j,n}=0, U_L^j(0)\in \{W,-W,0\}\text{ and } \partial_t U_L^j(0)=0\\
 j\in \III_c&\iff \forall n,\; t_{j,n}=0\text{ and } U^j\text{ is not a stationary solution of \eqref{NLW}}\\
 j\in \III_{\pm}&\iff \lim_{n\to\infty}\frac{-t_{j,n}}{\lambda_{j,n}}=\pm \infty.
\end{align*}
In Step 3, we prove that $\III_c=\emptyset$. In Step 4, we prove that $\III_{\pm}=\emptyset$. In Step 5, we conclude the proof, by showing that the first component of the dispersive remainder $w_n^J$ goes to $0$ in $\hdot$ as $n\to\infty$.

\medskip

\noindent\emph{Step 3. Compact profiles.} In this step we prove by contradiction that $\III_c$ is empty. Assume that $\III_c$ is not empty. As in Subsection \ref{SS:profile}, we denote by $\JJJ$ the set of indices $j\geq 1$ such that $\forall n$, $t_{j,n}=0$ and the solution $U^j$ cannot be extended to $\{r>|t|\}$ with finite $L^3L^6$ norm. We recall that by the small data theory, $\JJJ$ is finite, and note that $\JJJ\subset \III_c$. If $\JJJ$ is not empty, we can extract subsequences such that there exists a unique $j_0\in \JJJ$ such that 
$$\forall j\in \JJJ\setminus \{j_0\},\quad \lim_{n\to\infty} \frac{\lambda_{j,n}}{\lambda_{j_0,n}}=0.$$
By Theorem \ref{T:rigidity}, there exists $R>0$ such that the solution $U^{j_0}$ is defined for $|x|>R+|t|$ and
\begin{equation}
 \label{seq20}
 \sum_{\pm \infty} \lim_{t\to\pm\infty} \int_{|x|>R+|t|} \left|\nabla_{t,x}U^{j_0}(t,x)\right|^2\,dx=\eta_0>0.
\end{equation} 
If $\JJJ$ is empty, we let $j_0\in \III_c$ and $R=0$, so that $U^{j_0}$ is defined on $\{|x|>|t|\}=\{|x|>R+|t|\}$ and, by Theorem \ref{T:rigidity}, \eqref{seq20} holds. 

In both cases, combining \eqref{seq20} with finite speed of propagation and the small data theory, we see that 
the following holds for all $t>0$ or for all $t<0$:
\begin{equation}
 \label{seq32}
 \int_{|x|>R+|t|} |\nabla_{t,x}U^{j_0}(t,x)|^2dx\geq \eta_1>0. 
\end{equation} 
We let $R_n=R\lambda_{j_0,n}$ for all $n$. By Proposition \ref{P:NL_profile}, for large $n$, the solution $u$ is defined for $\{r>R_n+|t-t_n|\}$ and 
\begin{equation}
 \label{seq30'}
 u(t_n+\tau,r)=v(1+\tau,r)+\sum_{j=1}^JU_n^j(\tau,r)+w_n^J(\tau,r)+\eps_n^J(\tau,r),
\end{equation} 
where 
\begin{equation}
 \label{seq31} \lim_{J\to\infty} \limsup_{n\to\infty} \sup_{\tau} \left\|\eps_n^J\right\|_{\HHH(R_n+|\tau|)}=0,
\end{equation}
and $v$ is as above the solution of \eqref{NLW} with initial data $(v_0,v_1)$ at $t=1$, which is globally defined and scattering by the small data theory. Using the continuity of $v$, we can rewrite \eqref{seq30'} as 
\begin{equation}
 \label{seq30}
 u(t_n+\tau,r)=v(t_n+\tau,r)+\sum_{j=1}^JU_n^j(\tau,r)+w_n^J(\tau,r)+\eps_n^J(\tau,r),
\end{equation}  
where the remainder $\eps_n^J$ has slightly changed but still satisfies \eqref{seq31}. 
We distinguish two cases.

\smallskip

\noindent\emph{Case 1. Channels of energy in the future.} We assume that \eqref{seq32}  holds for all $t>0$. Recall that by finite speed of propagation $u(t,r)=v(t,r)$ for $r>1-t$, $t$ close to $1$. 
By \eqref{seq30} at $\tau=1-t_n$, we obtain, denoting $\rho_n=R\lambda_{j_0,n}+1-t_n$,
\begin{multline}
 \label{seq41}
 0=\sum_{1\leq j\leq J} \int_{|x|>\rho_n} \nabla_{t,x}U_n^{j}(1-t_n,x)\cdot\nabla_{t,x} U_n^{j_0}(1-t_n,x)\,dx\\
 +\int_{|x|>\rho_n} \nabla_{t,x} w_n^J(1-t_n,x)\cdot \nabla_{t,x} U_n^{j_0}(1-t_n,x)dx\\+\int_{|x|>\rho_n} \nabla \eps_n^J(1-t_n,x)\cdot\nabla_{t,x} U_n^{j_0}(1-t_n,x)dx.
\end{multline}
By \eqref{seq32}, at $t=(1-t_n)/\lambda_{j_0,n}$, 
\begin{multline}
 \label{seq42}
 \int_{|x|>\rho_n} |\nabla_{t,x} U_n^{j_0}(1-t_n,x)|^2dx\\
 =\int_{|y|>R+\frac{1-t_n}{\lambda_{j_0,n}}}  \left|\nabla_{t,x}U^{j_0}\left( \frac{1-t_n}{\lambda_{j_0,n}},y \right)\right|^2dy\geq \frac{\eta_1}{2}
\end{multline} 
for large $n$.  On the other hand, by Lemma \ref{L:psdo_ortho},
\begin{align}
 \label{seq50} j\neq j_0&\Longrightarrow \lim_{n\to\infty}\int_{|x|>\rho_n} \nabla_{t,x} U_n^{j}(1-t_n,x)\cdot\nabla_{t,x} U_n^{j_0}(1-t_n,x)\,dx=0\\
 \label{seq51}
 J\geq j_0&\Longrightarrow \lim_{n\to\infty}\int_{|x|>\rho_n} \nabla_{t,x} w_n^{J}(1-t_n,x)\cdot \nabla_{t,x}U_n^{j_0}(1-t_n,x)\,dx=0.
\end{align}
Combining \eqref{seq41}, \eqref{seq42}, \eqref{seq50}, \eqref{seq51} with the property \eqref{seq31} of $\eps_n^J$, we obtain a contradiction.

\smallskip

\noindent \emph{Case 2. Channels of energy in the past.} Next, we assume that \eqref{seq32} holds for all $t\leq 0$. We fix $t_0<1$ in the domain of existence of $u$. By \eqref{seq30}, at $\tau=t_0-t_n$, we have, letting $\tau_n=t_0-t_n$, $\rho_n=R\lambda_{j_0,n}+|\tau_n|$. 
\begin{multline}
 \label{seq52}
 \sum_{j=1}^J\int_{|x|>\rho_n} \nabla U_n^{j}(\tau_n,x)\cdot\nabla U_n^{j_0}(\tau_n,x)dx\\
 +\int_{|x|>\rho_n} \nabla_{t,x} w_n^J(\tau_n,x)\cdot \nabla_{t,x} U_n^{j_0}(\tau_n,x)dx+\int_{|x|>\rho_n} \nabla \eps_n^J(\tau_n,x)\cdot\nabla_{t,x} U_n^{j_0}(\tau_n,x)dx\\
 =\int_{|x|>\rho_n}\left(\nabla_{t,x}u(t_0,x)-\nabla_{t,x}v(t_0,x)\right) \cdot\nabla_{t,x}U_n^{j_0}(\tau_n,x)dx\underset{n\to\infty}{\longrightarrow} 0,
\end{multline}
where we have used that $\lim\tau_n=1-t_0$, and that $\vec{u}(t_0,x)-\vec{v}(t_0,x)=0$ for $|x|>1-t_0$.
Since \eqref{seq32} holds for all $t\leq 0$, we have
\begin{equation}
 \label{seq60}
 \forall n,\quad \int_{|x|>\rho_n}|\nabla U_n^{j_0}(\tau_n,x)|^2dx\geq \eta_1.
\end{equation} 
 By Lemma \ref{L:psdo_ortho}, 
\begin{align}
 \label{seq63} j\neq j_0&\Longrightarrow \lim_{n\to\infty}\int_{|x|>\rho_n} \nabla_{t,x} U_n^{j}(\tau_n,x)\cdot\nabla_{t,x} U_n^{j_0}(\tau_n,x)\,dx=0\\
 \label{seq64}
 J\geq j_0&\Longrightarrow \lim_{n\to\infty}\int_{|x|>\rho_n} \nabla_{t,x} w_n^{J}(\tau_n,x)\cdot \nabla_{t,x}U_n^{j_0}(\tau_n,x)\,dx=0.
\end{align}
 Combining \eqref{seq31}, \eqref{seq52}, \eqref{seq60}, \eqref{seq63} and \eqref{seq64} we obtain a contradiction. 
 
 \medskip
 
 \noindent \emph{Step 4. Scattering profiles.} We next prove that $\III_+$ and $\III_-$ are empty. The proof is close to the proof of the fact that $\III_c$ is empty, and we only sketch it, focusing on the arguments that are new with respect to Step 3. By Step 3, $\JJJ\subset \III_c=\emptyset$, and thus \eqref{seq30} and \eqref{seq31} hold with $R_n=0$. 
 
 Assume that $\III_+\neq \emptyset$, and let $j_0\in \III_+$. Recall that $U^{j_0}$ is a nonzero solution of the linear wave equation. By Proposition \ref{P:radiation} and Remark \ref{R:radiation},  we obtain that there exist $A\in \RR$, $\eta_1>0$ and $T$ such that
 $$\forall t\geq T,\quad \int_{|x|>t+A} |\nabla_{t,x}U^{j_0}(t,x)|^2dx\geq \eta_1>0.$$
Changing variables, we deduce
$$\forall \tau\geq \lambda_{j_0,n} T+t_{j_0,n},\quad \int_{|y|>\tau-t_{j_0,n}+\lambda_{j_0,n}A} |\nabla_{\tau,y}U_n^{j_0}(\tau,y)|^2dy\geq \eta_1.$$
We will use the preceding bound with $\tau=1-t_n$, which is possible since $-t_{j_0,n}/\lambda_{j_0,n}\to \infty$ as $n\to\infty$. Indeed, $\lambda_{j_0,n}T+t_{j_0,n}<0$, and thus $1-t_n\geq \lambda_{j_0,n} T+t_{j_0,n}$ for large $n$. Denoting $\rho_n=1-t_n-t_{j_0,n}+\lambda_{j_0,n}A$, we obtain that for large $n$
$$\int_{|y|>\rho_n} |\nabla_{t,y}U_n^{j_0}(1-t_n,y)|^2dy \geq \eta_1.$$
We note that $\rho_n>1-t_n$ for large $n$. The end of the proof is the same as the one of Case 1 of Step 3 above, and we omit it.

The proof of the fact that $\III_-=\emptyset$ is along the same lines (following Case 2 of Step 3) and is also omitted.

\medskip

\noindent \emph{Step 5.} We have proved the existence of $J\geq 0$, $\{\iota_j\}_{1\leq j\leq J}\in \{\pm 1\}^J$, $0<\lambda_{1,n}\ll \ldots \ll \lambda_{J,n}$ such that 
\begin{equation}
 \label{seq80}
 \vec{u}(t_n)=\vec{v}(t_n)+\sum_{j=1}^{J} \left( \frac{\iota_j}{\lambda_{j,n}}W\left( \frac{\cdot}{\lambda_{j,n}} \right),0 \right)+(w_{0,n},w_{1,n}),
\end{equation} 
where, denoting by $w_n$ the solution of the linear wave equation with initial data $(w_{0,n},w_{1,n})$,
$$ \lim_{n\to\infty} \|w_n\|_{L^3(\RR,L^6)}=0.$$
 Using that the support of $u-v$ is included in $\{r<1-t\}$, that, by \eqref{seq80},
 $$ \lambda_{J,n} (u-v)(t_n,\lambda_{J,n}\cdot)\xrightharpoonup[n\to\infty]{}\iota_{J} W,$$
 and that $W$ is not compactly supported, we obtain 
 \begin{equation}
  \label{noSS}
 \lim_{n\to\infty}\frac{\lambda_{J,n}}{1-t_n}=0.
 \end{equation} 
 In this step, we prove that 
 \begin{equation}
  \label{conclu_Step4}
\lim_{n\to\infty}\|w_{0,n}\|_{\hdot}=0,
\end{equation} 
 which will conclude the proof of Proposition \ref{P:subsequence}. By \eqref{CKS} and the time decay of the exterior energy for the linear wave equation, the following holds for all $t\geq 0$ or for all $t\leq 0$:
 \begin{equation}
  \label{channelwn}
  C\int_{|x|>|t|} |\nabla_{t,x}w_n(t,x)|^2dx\geq \|w_{0,n}\|_{\hdot}^2.
 \end{equation} 
 On the other hand, since by \eqref{seq80}, 
 $$\lim_{n\to\infty} \int_{|x|\geq 1-t_n} \left(|\nabla w_{0,n}(x)|^2+w_{1,n}^2(x)\right)dx=0,$$
 we see that for any $\eps>0$, if $n$ is large enough, 
 $$ \forall t,\quad \int_{|x|\geq 1-t_n+|t|} |\nabla_{t,x} w_n(t,x)|^2dx\leq \eps.$$
 Thus for large $n$, the following holds for all $t\geq 0$ or for all $t\leq 0$.
 \begin{equation}
  \label{channelwn'}
  C\int_{|t|<|x|<|t|+(1-t_n)} |\nabla_{t,x}w_n(t,x)|^2dx\geq \|w_{0,n}\|_{\hdot}^2.
 \end{equation} 
 Arguing by contradiction along the same lines as in Step 3, we obtain \eqref{conclu_Step4}. This and \eqref{seq80} give \eqref{CVH1} and \eqref{dev_energy}.
 \end{proof}
\subsubsection{End of the proof}
We are now in position to prove Theorem \ref{T:main} in the finite time blow-up case.
We consider $u$ as above, and let $(v_0,v_1)$ be given by Lemma \ref{L:v0v1}.

\noindent\emph{Step 1. Boundedness.} We first prove that $u$ remains bounded in $\HHH$ as $t\to T_+$. For this, we let $\{t_n\}_n\to T_+$ be a sequence of times such that $\{\vec{u}(t_n)\}_n$ is bounded in $\HHH$. By Proposition \ref{P:subsequence},  there exists $J$ such that 
\begin{gather}
 \label{devH1}
\|u(t_n)\|^2_{\hdot}=\|v_0\|^2_{\hdot}+J\|W\|^2_{\hdot} +o_n(1)\\
\label{dev_energy'}
E(u_0,u_1)=E(v_0,v_1)+ J E(W,0)+ \frac{1}{2}\left\|\partial_t u(t_n)-v_1\right\|_{L^2}^2+o_n(1).
\end{gather}
From \eqref{dev_energy'}, we deduce 
$$\limsup_{n\to\infty} \left\|\partial_t u(t_n)-v_1\right\|_{L^2}^2\leq E(u_0,u_1)-E(v_0,v_1)$$
and
$$ J\leq \frac{E(u_0,u_1)-E(v_0,v_1)}{E(W,0)}.$$
From the equation $-\Delta W=W^3$, we obtain $\int |\nabla W|^2=\int |W|^4$ and thus $\|W\|^2_{\hdot}=4E(W,0)$. Using also \eqref{devH1}, we deduce
$$\limsup_{n\to\infty} \|u(t_n)\|^2_{\hdot} \leq \|v_0\|^2_{\hdot}+ 4(E(u_0,u_1)-E(v_0,v_1)).$$
Combining, we see that there exists a constant $C_0$, depending only on $u$ (but independent of the choice of the sequence $\{t_n\}_n$) such that if $\{\vec{u}(t_n)\}_n$ is bounded in $\HHH$, 
$$ \limsup_{n\to\infty}\|\vec{u}(t_n)\|_{\HHH}\leq C_0.$$
Using that 
$$\liminf_{t\to\infty} \|\vec{u}(t)\|_{\HHH}<\infty$$
we deduce immediately, since $t\mapsto \|\vec{u}(t)\|_{\HHH}$ is continuous,
$$\limsup_{t\to T_+}\|\vec{u}(t)\|_{\HHH}\leq C_0.$$ 

\medskip

\noindent\emph{Step 2. Convergence to $0$ of the time derivative inside the wave cone.} In this step we prove
\begin{equation}
 \label{BUp30}
\lim_{t\to T_+}\|\partial_tu(t)-v_1\|_{L^2}=0.
\end{equation} 
By Step 1 and the soliton resolution for a sequence of times proved in \cite{CoKeLaSc18}, there exists a sequence $\{t_n\}_n$, $J\geq 1$, $(\iota_j)_j\in \{\pm 1\}^J$, sequences $\{\lambda_{j,n}\}_n$, $1\leq j\leq J$ with 
\begin{equation}
\label{cond_lambda}
\lambda_{1,n}\ll \ldots \ll \lambda_{J,n}\ll T_+-t_n, 
\end{equation} 
such that
\begin{equation}
\label{BUp31}
\lim_{n\to \infty} \left\|\vec{u}(t_n)-(v_0,v_1) -\sum_{j=1}^J \left(\frac{\iota_j}{\lambda_{j,n}} W\left( \frac{\cdot}{\lambda_{j,n}} \right),0\right)\right\|_{\HHH}=0.
\end{equation}
Note that the conclusion of Step 1 is necessary here, since the main result of \cite{CoKeLaSc18} is valid assuming that $\vec{u}(t)$ is bounded in $\HHH$, a stronger statement than our assumption \eqref{liminf}.

As an immediate consequence of this, we obtain
\begin{equation}
 \label{BUp32}
E(u_0,u_1)=E(v_0,v_1)+JE(W,0).
\end{equation} 
We prove \eqref{BUp30} by contradiction. If it does not hold, using \eqref{BUp31}, we can find an arbitrarily small $\eps_0>0$ and a sequence $t_n'\to T_+$ such that $\|\partial_tu(t_n')-v_1\|_{L^2}=\eps_0$. By Proposition \ref{P:subsequence}, there exists an integer $J'\geq 0$ such that
$$E(u_0,u_1)=E(v_0,v_1)+J'E(W,0)+\frac{1}{2}\eps_0^2.$$
Combining with \eqref{BUp32}, we deduce that 
$$\frac{1}{2}\eps_0^2=(J-J')E(W,0),$$
a contradiction if $\eps_0^2$ is smaller than $E(W,0)$. 

\medskip

\noindent\emph{Step 3. End of the proof.} Combining Steps $1$ and $2$ with Proposition \ref{P:subsequence}, we see that for any sequence $\{t_n\}_n\to T_+$, extracting subsequences if necessary, there exist $J\geq 1$, $(\iota_j)_j\in \{\pm 1\}^J$, sequences $\{\lambda_{j,n}\}_n$, $1\leq j\leq J$ such that \eqref{cond_lambda} and \eqref{BUp31} hold. We note that 
$$ J=\frac{E(u_0,u_1)-E(v_0,v_1)}{E(W,0)}$$
is independent of the choice of $\{t_n\}_n$.  The case $J=0$ is excluded since $T_+$ is the maximal time of existence of $u$. 

It remains to construct the scaling parameters $\lambda_j(t)$ such that the expansion \eqref{expansion_u_bup} holds. 
This can be done as in \cite{DuKeMe13}, Section 3.5
defining, for $j=1\ldots J$ and $t<1$ close to $1$,
$$ B_j:=(j-1) \|\nabla W\|_{L^2}^2+\int_{|x|\leq 1}|\nabla W(x)|^2\,dx$$
and
\begin{equation*}
\lambda_{j}(t):=\inf\left\{ \lambda>0\text{ s.t. } \int_{|x|\leq \lambda} \left|\nabla(u-v)(t,x)\right|^2\,dx\geq B_j\right\}.
\end{equation*} 
We refer to Section 3.5 of \cite{DuKeMe13} for the proof that the conclusion of Theorem \ref{T:main} in the finite time blow-up case  holds with this choice of $\lambda_j(t)$ 

\subsection{Comments on the proof in the global case}
We recall:
\begin{prop}
\label{P:global}
Let $u$ be a solution of \eqref{NLW} such that $T_+(u)=+\infty$. Then
\begin{itemize}
 \item $\ds \liminf_{t\to \infty} \|\vec{u}(t)\|_{\HHH}<\infty$.
 \item There exists a finite energy solution $v_L$ of the free wave equation such that
 $$ \forall A\in \RR,\quad \lim_{t\to\infty} \int_{|x|>t+A} |\nabla_{t,x}(u-v_L)(t,x)|^2dx=0.$$
\end{itemize}
\end{prop}
See Subsections 3.2 and 3.3 in \cite{DuKeMe13} for the proofs. The proofs are written in space dimension $3$ there, but are indeed independent of the dimension. See also \cite{CoKeLaSc18}, Subsection 4.1, for the proof of the second point in space dimension $4$.

In view of Proposition \ref{P:global}, the proof of Theorem \ref{T:main} in the global case is very close to the corresponding proof in the finite time blow-up case, replacing $(v_0,v_1)$ and $t=1$ by $\vec{v}_L(t)$ and $t=+\infty$, and we omit it. 

\section{Soliton resolution for corotional wave maps}
\label{S:resolutionWM}
\subsection{Preliminaries on wave maps}
\subsubsection{Miscellaneous results on wave maps}
We gather here a few standard results on wave maps. 
\begin{claim}
 \label{Cl:smallWM}
 Let $(\psi_0,\psi_1)\in \Ebf$, such that $E_M(\psi_0,\psi_1)$  is small. Then there exists $\ell$ such that $(\psi_0,\psi_1)\in \Hbf_{\ell,\ell}$ and
 $$ \|(\psi_0-\ell\pi,\psi_1)\|_{\Hbf}^2\approx E_M(\psi_0,\psi_1).$$
 Similarly, if $R>0$, $(\psi_0-\ell\pi,\psi_1)\in \Hbf(R)$ for some $\ell\in \ZZ$, and 
 $$ \int_{R}^{\infty} \left(\psi_1^2+(\partial_r\psi_0)^2 +\frac{1}{r^2}\sin^2\psi_0(r)\right)rdr$$ 
 is small, then
 \begin{multline*}
\int_{R}^{\infty} \left(\psi_1^2+(\partial_r\psi_0)^2 +\frac{1}{r^2}\sin^2\psi_0(r)\right)rdr\\
\approx \int_{R}^{\infty} \left(\psi_1^2+(\partial_r\psi_0)^2 +\frac{1}{r^2}(\psi_0(r)-\ell\pi)^2\right)rdr.
 \end{multline*}
\end{claim}
\begin{proof}
 (See also \cite[Lemma 5.3]{JiaKenig17}). Let $(\psi_0,\psi_1)\in \Ebf$, and $(\ell,m)\in \ZZ^2$ such that $(\psi_0,\psi_1)\in \Hbf_{\ell,m}$. Since $\psi_0(0)=\ell\pi$, we have, for any $r>0$, 
 $$ \left|\int_{\ell\pi}^{\psi_0(r)} |\sin(\sigma)|d\sigma\right|\leq \int_{0}^r |\sin(\psi_0(\sigma))|\,|\partial_r\psi_0(\sigma)|d\sigma\leq E_M(\psi_0,\psi_1).$$
 Since $E_M(\psi_0,\psi_1)$ is small, we deduce that $|\psi_0(r)-\ell\pi|$ is small for all $r$. This implies in particular $\psi_0\in \Hbf_{\ell,\ell}$, $\frac{1}{2}\left|\psi_0(r)-\ell\pi\right|\leq |\sin(\psi_0(r))|\leq |\psi_0(r)-\ell\pi|$ and the first assertion of the claim follows. The proof of the second assertion is similar and we omit it.
\end{proof}
Recalling that the outer energy
 $$E_{\out}(\vec{\psi}(t))=\frac{1}{2}\int_{|t|}^{\infty} \left((\partial_t \psi(t,r))^2+(\partial_r \psi(t,r))^2+\frac{\sin^2 \psi(t,r)}{r^2}\right)rdr$$
 is nonincreasing for $t>0$ (and nondecreasing for $t<0$), we can now complete the proof of Theorem \ref{T:rigidityWM}. 
\begin{proof}[End of the proof of the rigidity theorem]
Let $\psi$ be a solution of \eqref{WM} which is not a stationary solution, with data $(\psi_0,\psi_1)\in \Hbf_{\ell,m}$. We have already proved \eqref{channelWM}. Assume to fix ideas that 
\begin{equation}
\label{WM19}
\lim_{t\to +\infty} \int_{|t|}^{\infty} \left((\partial_{t,r}\psi(t,r))^2+\frac{1}{r^2}(\psi(t,r)-m\pi)^2 \right)rdr>0. 
\end{equation} 
Let 
\begin{equation*}
\eta=\lim_{t\to\infty}\int_{t}^{\infty} \left((\partial_{t,r}\psi(t,r))^2+\frac{1}{r^2}\sin^2\psi(t,r)\right)rdr=\lim_{t\to\infty}E_{\out}(\vec{\psi}(t)). 
\end{equation*} 
Since $E_{\out}$ is nonincreasing for $t\geq 0$, we are reduced to prove that $\eta>0$. However if $\eta=0$, then for large $t$, $E_{\out}(\vec{\psi}(t))$ is small and the second assertion of Claim \ref{Cl:smallWM} applies. This yields a contradiction with \eqref{WM19}, concluding the proof of \eqref{lower_boundWM}.
\end{proof}

We will need the following estimates on the sinus functions:
\begin{claim}
\label{Cl:sin}
Let $J\geq 2$, $(a_j)_{1\leq j\leq J}\in \RR^J$. Then
\begin{gather}
\label{sin1}
 \left|\sin\Big( 2\sum_{j=1}^J a_j \Big)-\sum_{j=1}^J\sin(2a_j)\right|\lesssim \sum_{j\neq k}|\sin(2a_j)|\,\sin^2(a_k)\lesssim \sum_{j\neq k}|\sin(a_j)|\,\sin^2(a_k)\\
 \label{sin2}
 \left|\sin^2\Big( \sum_{j=1}^J a_j \Big)-\sum_{j=1}^J\sin^2(a_j)\right|\lesssim 
\sum_{j\neq k}\left| \sin a_j\sin a_k\right|
\end{gather}
\end{claim}
\begin{proof}
 \noindent\emph{Proof of \eqref{sin1}.} The inequality \eqref{sin1} with $J=2$ follows from
 $$ \sin(2(a_1+a_2))-\sin(2a_1)-\sin(2a_2)=-2\sin(2a_1)\sin^2a_2-2\sin(2a_2)\sin^2a_1.$$
The general case follows from a straightforward induction and the bound $|\sin(a+b)|\leq |\sin a|+[\sin b|$.

\medskip

\noindent\emph{Proof of \eqref{sin2}.} In the case $J=2$, we have
$$\left|\sin^2(a_1+a_2)-\sin^2a_1-\sin^2a_2\right|=2\left|\sin a_1\sin a_2 \cos (a_1-a_2)\right|\leq 2|\sin a_1 \sin a_2 |.$$
The general case follows from an elementary induction.
\end{proof}

In the following proposition, we gather some important facts about finite time blow-up solutions of \eqref{WM} 
\begin{prop}
\label{P:BlowupWM}
Let $\psi$ be a finite energy solution of \eqref{WM} with maximal time of existence $T_+<\infty$. Then there exists $(\varphi_0,\varphi_1)\in \Ebf$ such that  for all $R>0$,
$$\lim_{t\to T_+} \int_{R}^{\infty} \left((\partial_t \psi(t)-\varphi_1)^2+(\partial_r \psi(t)-\varphi_0)^2+\frac{1}{r^2}(\psi(t)-\varphi_0)^2\right)r\,dr=0.$$
The solution of \eqref{WM}, with initial data $(\varphi_0,\varphi_1)$ at $t=T_+$ satisfies
$$ \psi(t,r)=\varphi(t,r),\quad 0<r<T_+-t, \quad t<T_+\text{ close to }T_+.$$
Furthermore,
\begin{equation}
 \label{QuantWM}
\lim_{t\to T_+} \frac 12\int_0^{T_+-t} \left((\partial_t \psi(t,r))^2+(\partial_r \psi(t,r))^2+\frac{\sin^2(r)}{r^2}\right)rdr\geq E_M(Q,0).
\end{equation} 
Finally, there exists $t_n\to T_+$ such that 
\begin{equation}
 \label{derivative0_WMBup}
 \lim_{n\to\infty} \int_0^{T_+-t_n}(\partial_t \psi(t_n,r)^2rdr=0.
\end{equation} 
\end{prop}
Note that the energy inside the wave cone
$$E_{\inn}(\vec{\psi}(t))=\frac{1}{2}\int_0^{T_+-t} \left((\partial_t \psi(t,r))^2+(\partial_r \psi(t,r))^2+\frac{\sin^2(r)}{r^2}\right) rdr$$
is a nonincreasing function of $t\in [0,T_+)$, so that the limit appearing in \eqref{QuantWM} exists.

The existence of $(\varphi_0,\varphi_1)$ follows from small data well-posedness arguments and finite speed of propagation. See \cite[Lemma 5.2]{CoKeLaSc15a}, \cite[Proposition 5.2]{Cote15}. The lower bound \eqref{QuantWM} is an immediate consequence of the work of Struwe \cite{Struwe03b}. Finally, \eqref{derivative0_WMBup} can be seen as a consequence of the main result of \cite{Cote15}, but is indeed a step of the proof of this main result (see \cite[Corollary 2.7]{Cote15}).

Recall the linearized equation for \eqref{WM} around a constant solution:
\begin{equation}
 \label{LWM}
 \left\{
\begin{aligned}
 \partial_t^2\psi_L-\partial_r^2\psi_L-\frac{1}{r}\partial_r\psi_L+\frac{1}{r^2}\psi_L=0\\
 \vec{\psi}_{L\restriction t=0}=(\psi_0,\psi_1)\in \Hbf.
 \end{aligned}
 \right.
\end{equation} 
\begin{prop}
 \label{P:globalWM}
 Let $\psi$ be a solution of \eqref{WM} with maximal time of existence $T_+=+\infty$ and initial data $(\psi_0,\psi_1)\in \Hbf_{\ell,m}$. Then there exists a solution $\psi_L$ of \eqref{LWM} with initial data in $\Hbf$ and an increasing positive function $\alpha(t)=o(t)$ as $t\to\infty$ such that 
 \begin{equation}
 \label{existence_psiL}
 \lim_{t\to\infty} \left\|(\psi(t)-m\pi-\psi_L(t),\partial_t \psi(t)-\partial_t\psi_L(t))\right\|_{\Hbf(\alpha(t))}=0.  
 \end{equation} 
 Furthermore, there exists $t_n\to+\infty$ such that 
\begin{equation}
 \label{derivative0_WMglobal}
 \lim_{n\to\infty} \int_0^{\infty} (\partial_t \psi(t_n,r) -\partial_t \psi_{L}(t_n,r))^2rdr=0.
\end{equation} 
\end{prop}
See \cite[Proposition 5.1]{Cote15} for \eqref{existence_psiL}  and \cite[Corollary 2.3]{Cote15} for \eqref{derivative0_WMglobal}. 
\subsubsection{Rigidity theorem }

\subsubsection{A space-time bound outside the wave cone}
We will denote by $S$ the space of measurable functions
$\psi$ of $(t,r)\in \RR\times (0,\infty)$ such that the following norm is finite.
\begin{equation}
 \label{P10}
 \|\psi\|_{S}=\left( \int_{\RR} \left( \int_0^{\infty} |\psi(t,r)|^6\frac{dr}{r^3} \right)^{\frac 12}dt \right)^{\frac 13},
\end{equation} 
and by $S(\{r>|t|\}$ the space of restrictions of $S$ to $\{r>|t|\}$, with the norm
\begin{equation}
 \label{P10'}
 \|\psi\|_{S(\{r>[t[\})}=\left( \int_{\RR} \left( \int_{|t|}^{\infty} |\psi(t,r)|^6\frac{dr}{r^3} \right)^{\frac 12}dt \right)^{\frac 13},
\end{equation}
By the change of function $u=\frac{1}{r}\psi$ and the Strichartz estimate for the wave equation in dimension $1+4$, all solutions of \eqref{LWM} are in $S$. This is of course not the case for general finite energy solution of the wave map equation. However we have the following space time bound, which will be useful to construct a nonlinear profile decomposition for \eqref{WM} outside wave cones.
\begin{lemma}
 \label{L:L3L6sinus}
 Let $(\psi_0,\psi_1)\in \Hbf_{\ell,m}$, $(\ell,m)\in \ZZ^2$, and let $\psi$ be the solution of \eqref{WM} with initial data $(\psi_0,\psi_1)$. Then $\sin \psi\in S(\{r>|t|\})$, i.e.
 $$ \int_{-\infty}^{\infty} \left(\int_{|t|}^{\infty}\frac{\sin^6\psi(t,r)}{r^3}dr\right)^{1/2}dt<\infty.$$
\end{lemma}
\begin{remark}
The solution $\psi$ is well-defined for $r>|t|$, see Claim \ref{Cl:outsideWM}.
\end{remark}
\begin{proof}
 Let $u=\frac{\psi-\ell\pi}{r}$, $u_0=\frac{(\psi_0-\ell\pi)}{r}$, $u_1=\frac{\psi_1}{r}$. Since $(\varphi_0,\varphi_1)\in \Hbf_{\ell,m}$, $(u_0,u_1)$ can be considered as a radial function in $(\dot{H}^1\times L^2)_{\loc}(\RR^4)$. By the local well-posedness for the equation \eqref{gNLW} and finite speed of propagation, $u$ is defined on $\{-t_0\leq t\leq t_0,\; 0<r\leq 2\}$ for some small $t_0>0$, and 
 $$ \int_{-t_0}^{t_0} \left(\int_0^2 u^6(t,r)r^3dr\right)^{1/2}dt<\infty. $$
 Hence
 \begin{equation}
 \label{sin10}
 \int_{-t_0}^{t_0} \left(\int_{|t|}^2 \frac{(\sin{\psi})^6}{r^3}dr\right)^{1/2}dt\leq \int_{-t_0}^{t_0} \left(\int_{|t|}^2 \frac{(\psi-\ell\pi)^6(t,r)}{r^3}dr\right)^{1/2}dt<\infty. 
 \end{equation} 
Similarly, for all $R>0$, $\tilde{u}=\frac{\psi-m\pi}{r}$ is a solution of \eqref{gNLW} defined for $r>R+|t|$ with initial data in $\HHH(R)$. Thus for all $T>0$, 
 $$ \int_{-T}^{T} \left(\int_{R+|t|}^{\infty} {\tilde{u}}^6(t,r)r^3dr\right)^{1/2}dt<\infty,$$
 which yields
 $$ \int_{-T}^{T} \left(\int_{R+|t|}^{\infty} \frac{(\sin\psi)^6}{r^3}dr\right)^{1/2}dt<\infty.$$
 Combining with a space translation in time and \eqref{sin10} we obtain
 \begin{equation}
  \label{sin21}
  \forall T>0, \quad \int_{-T}^{T} \left(\int_{|t|}^{\infty} \frac{(\sin\psi)^6}{r^3}dr\right)^{1/2}dt<\infty
 \end{equation} 
 
Using that $E_{\out}$ is nonincreasing for $t>0$, <e define
 \begin{equation}
  \label{sin22}
  E_{\out}(\infty)=\lim_{t\to+\infty} E_{\out}\big(\vec{\psi}(t)\big).
 \end{equation} 
 Let $T$ such that 
 \begin{equation}
  \label{sin23} 
  E_{\out}\big(\vec{\psi}(T)\big)\leq E_{\out}(\infty) +\frac{1}{4}E_M(Q,0).
 \end{equation}  
 We extend $\vec{\psi}(T,r)$, which is defined for $r>T$, as follows. We define 
 $$(\phi_0,\phi_1)(r)=(\psi(T,r),\partial_t\psi(T,r)),\quad r\geq T,$$
 and $\varphi_1(r)=0$ if $r<T$. Let $p\in \ZZ$ such that $p\pi\leq \psi(T,T)< (p+1)\pi$. For $r<T$, we let 
 $$  \begin{cases} \phi_0(r)=p\pi  &\text{ if }\psi(T,T)=p\pi\\
\phi_0(r)=Q(\lambda r)+p\pi&\text{ if } p\pi<\psi(T,T)\leq p\pi+\pi/2\\
\phi_0(r)=(p+1)\pi-Q(\lambda r)&\text{ if } p\pi+\pi/2<\psi(T,T)< (p+1)\pi,\end{cases}$$
 where in the second and third case, $\lambda$ is chosen  so that $\phi_0$ is continuous at $r=T$. We note that $(\phi_0,\phi_1)\in \Ebf$. We claim
 \begin{equation}
 \label{sin30}
 E_M(\phi_0,\phi_1)\leq E_{\out}(\infty) +\frac{3}{4}E_M(Q,0).
 \end{equation} 
 In view of \eqref{sin23}, it is sufficient to check that 
\begin{equation}
 \label{sin31}
 \int_0^T (\partial_r\phi_0)^2rdr+\int_0^T \sin^2 \phi_0\frac{dr}{r}\leq E_M(Q,0).
\end{equation} 
In the case where $\phi_0(r)=p\pi$ for $r<T$, this is trivial. In the two other cases, it follows from the fact that $Q(1)=2\arctan 1=\pi/2$ and that
$$\int_0^1 (\partial_rQ)^2rdr+\int_0^1 \sin^2 Q\frac{dr}{r}= E_M(Q,0)$$
which can be proved by an explicit computation.
 
 Let $\phi$ be the solution of \eqref{WM} with initial data $\vec{\phi}(T)=(\phi_0,\phi_1)$. By finite speed of propagation, $\psi(t,r)=\phi(t,r)$, $r>t\geq T$. As a consequence, for $t\geq T$, $E_{\out}(\vec{\phi}(t))\geq E_{\out}(\infty)$. Combining with the conservation of the energy and \eqref{sin30}, we deduce
 $$ \forall t\geq T,\quad (E_M-E_{\out})(\vec{\phi}(t))\leq \frac{3}{4} E_M(Q,0).$$
 In view of the energy concentration for blow-up solutions (see \eqref{QuantWM}), the solution $\phi$ is global in the future. Let $\phi_L$ be the solution of \eqref{LWM} given by Proposition \ref{P:globalWM}, so that 
\begin{equation}
 \lim_{t\to\infty} \left\|(\phi(t)-m\pi-\phi_L(t),\partial_t \phi(t)-\partial_t\phi_L(t))\right\|_{\Hbf(\alpha(t))}=0,
 \end{equation}  
 where $\alpha(t)/t\to 0$ as $t\to\infty$.
 Let $\tT\gg 1$, so that
 $$\left\|\left(\phi\big(\tT\big)-m\pi-\phi_L\big(\tT\big),\partial_t \phi\big(\tT\big)-\partial_t\phi_L\big(\tT\big)\right)\right\|_{\Hbf(\tT)}+\int_{\tT} ^{\infty}\left(\int_{|t|}^{\infty}\frac{\phi^6_L(t,r)}{r^3}dr\right)^{1/2}dt \ll 1.$$ 
Recall that $\frac{1}{r}\phi_L$ is a solution of the linear wave equation in $\RR\times \RR^4$, and that $\frac{\sqrt{2}}{\sqrt{3}r}(\phi-m\pi)$ is solution of the nonlinear equation \eqref{gNLW}.  Then by Proposition \ref{P:smalldata},
$$ \int_{\tT}^{\infty} \left( \int_{|t|}^{\infty} \frac{\left(\phi(t)-m\pi  \right)^6}{r^3} dr\right)^{1/2} dt<\infty, $$
which implies
 $$\int_{\tT}^{\infty} \left(\int_{|t|}^{\infty} \frac{\sin^6\psi(t,r)}{r^3}dr\right)^{1/2} dt=\int_{\tT}^{\infty} \left(\int_{|t|}^{\infty} \frac{\sin^6\phi(t,r)}{r^3}dr\right)^{1/2} dt<\infty.$$
 Using the same argument for negative times, we obtain the desired conclusion.
 \end{proof}

\subsubsection{Profile decomposition}
\label{SSS:profile}
We next state a profile decomposition which is adapted to the wave maps equation \eqref{WM}. 
\begin{prop}[Profile decomposition for wave maps]
 \label{P:profileWM}
 Let $(\ell,m)\in \ZZ^2$. Consider a sequence $\left\{(\psi_{0,n},\psi_{1,n})\right\}_n$ of elements of $\Hbf_{\ell,m}$ such that 
 \begin{equation}
  \label{P11}
  \sup_nE_M(\psi_{0,n},\psi_{1,n})<\infty.
 \end{equation} 
 Then there exists a subsequence of $\left\{(\psi_{0,n},\psi_{1,n})\right\}_n$ (that we denote the same) and, for all $j\in \NN\setminus \{0\}$, for all $n$, $\left(\Psi_{0,n}^j,\Psi_{1,n}^j\right) \in \Ebf$, $\lambda_{j,n}>0$, $t_{j,n}\in \RR$ such that
 \begin{equation}
  \label{P12}
  j\neq k\Longrightarrow \lim_{n\to\infty} \frac{\lambda_{j,n}}{\lambda_{k,n}} +\frac{\lambda_{k,n}}{\lambda_{j,n}} +\frac{|t_{j,n}-t_{k,n}|}{\lambda_{j,n}}=+\infty,
 \end{equation} 
 and, for all $j\geq 1$, one the following holds:
 \begin{description}
  \item[Compact profile] $\forall n$, $t_{j,n}=0$ and
  $$\exists \left( \Psi_0^j,\Psi_1^j \right)\in \Ebf,\quad (\Psi_{0,n}^j(r),\Psi_{1,n}^j(r))=\left( \Psi_0^j\left( \frac{r}{\lambda_{j,n}} \right),\frac{1}{\lambda_{j,n}}\Psi_1^j\left( \frac{r}{\lambda_{j,n}} \right) \right).$$
  \item[Linear wave profile] $\lim_{n\to\infty}\frac{-t_{j,n}}{\lambda_{j,n}}\in \{\pm\infty\}$ and there exists a solution $\Psi^j_L$ of \eqref{LWM} with initial data $(\Psi_0^j,\Psi_1^j)\in \Hbf$ such that 
  $$(\Psi_{0,n}^j(r),\Psi_{1,n}^j(r))=\left( \Psi_L^j\left( \frac{-t_{j,n}}{\lambda_{j,n}},\frac{r}{\lambda_{j,n}} \right),\frac{1}{\lambda_{j,n}}\Psi_L^j\left( \frac{-t_{j,n}}{\lambda_{j,n}},\frac{r}{\lambda_{j,n}} \right) \right).$$
 \end{description}
Furthermore, denoting by
\begin{equation}
\label{P20}
 \left(\omega_{0,n}^J,\omega_{1,n}^J\right)= \left( \psi_{0,n},\psi_{1,n} \right)-\sum_{j=1}^J \left( \Psi_{0,n}^j,\Psi_{1,n}^j \right),
\end{equation} 
then for large $J$, $(\omega_{0,n}^J,\omega_{1,n}^J)\in \Hbf$ for all $n$,
\begin{equation}
 \label{PytWM}
 E_M\left( \psi_{0,n},\psi_{1,n} \right)=\sum_{j=1}^J E_M\left(\Psi_{0,n}^J,\Psi_{1,n}^J\right)+E_M\left( \omega_{0,n}^J,\omega_{1,n}^J \right)+o_n(1),
\end{equation} 
and, letting $\omega_{n}^J$ be the solution of \eqref{LWM} with initial data $(\omega_{0,n}^J,\omega_{1,n}^J)$, one has 
\begin{equation}
 \label{P22}
 \lim_{J\to\infty} \limsup_{n\to\infty}\left\|\omega_n^J\right\|_S+\|\omega_n^J\|_{L^{\infty}_{t,r}}=0.
\end{equation} 
Finally, for large $J$, $\left\{(\omega_{0,n}^J,\omega_{1,n}^J)\right\}_n$ is bounded in $\Hbf$ and 
\begin{equation}
\label{wnJWM}
1\leq j\leq J\Longrightarrow \left(\omega_n^J(t_{j,n},\lambda_{j,n}\cdot),\lambda_{j,n}\partial_t\omega_n^J(t_{j,n},\lambda_{j,n}\cdot)\right)\xrightharpoonup[n\to\infty]{} 0\text{ in }\Hbf. 
\end{equation} 
\end{prop}
We will denote by $\JJJ_C$ the indices corresponding to compact profiles, and by $\JJJ_L$ the indices corresponding to nonzero linear wave profiles. We can assume that there exists at most one nonzero constant profile (which is a compact profile). We will consider the profiles such that $(\Psi_0^j,\Psi_1^j)\equiv 0$, as compact profiles. Of course, if there is an infinite number of nonzero profiles, we can assume (rearranging the indices if necessary) that there are no such profiles.

By \eqref{PytWM}, 
 $$\sum_{j\in \JJJ_C}E_M(\Psi_{0}^j,\Psi_{1}^j)<\infty.$$
 By Claim \ref{Cl:smallWM}, for large $j$, $(\Psi_{0}^j,\Psi_{1}^j)\in \Hbf_{\ell_j,\ell_j}$ for some $\ell_j\in \ZZ$, and, since $(\omega_{0,n}^J,\omega_{1,n}^J)\in \Hbf$ for large $J$, $(\Psi_{0}^j,\Psi_{1}^j)\in \Hbf$ for large $j$.

Proposition \ref{P:profileWM} is \cite[Lemma 5.5]{JiaKenig17}, with some changes in the notations:
 \begin{itemize}
  \item The constants denoted by $\phi_0^l(0)$ that appear in the decomposition (5.16) of \cite{JiaKenig17} can be summed up all together and considered (if this sum is not zero) as a constant profile $(\Psi_0^1,\Psi_1^1)=(p\pi,0)$, $p\in \ZZ\setminus\{0\}$ with $1\in \JJJ_C$.
  \item In \cite[Lemma 5.5]{JiaKenig17} the solutions of \eqref{LWM} are denoted by $ru_L$ where $u_L$ is a radial solution of the $1+4$-dimensional wave equation.
  \item In \cite{JiaKenig17}, the profiles corresponding to indices $j\in \JJJ_C$ are divided into two types of profiles, the ones defined as rescaled $(\psi_0^l,0)\in \Ebf$, and the ones defined as rescaled $U^j_L$, with $t_{j,n}=0$ for all $n$ (see (5.16) in this article). 
\end{itemize}
The only part of Proposition \ref{P:profileWM} which is not included in \cite{JiaKenig17} is the Pythagorean expansion \eqref{PytWM} of the energy for wave maps. The proof is given in Appendix \ref{AA:psdo_orthoWM}. We refer to \cite[Lemma 2.16]{CoKeLaSc15a} for a proof of this Pythagorean expansion in a special case.  

We now state the analog of Proposition \ref{P:NL_profile} (a nonlinear profile decomposition outside a wave cone) for wave maps. Since the solutions of \eqref{WM} are always defined on the entire exterior of a wave cone $\{r>|t|\}$ and satisfy the space-time bound given by Lemma \ref{L:L3L6sinus} there, this  decomposition will be valid on $\{r>|t|\}$. 

Let $(\psi_{0,n},\psi_{1,n})$, and, for $j\in \NN\setminus\{0\}$, $\left(\Psi_0^j,\Psi_1^j\right)$, $\{t_{j,n}\}_n$, $\{\lambda_{j,n}\}_n$, $\omega_n^j$ be as in Proposition \ref{P:profileWM}. If $j\in \JJJ_C$, we denote by $\Psi^j$ the solution of \eqref{WM} with initial data $\left( \Psi_0^j,\Psi_1^j \right)$ at $t=0$, which is well-defined for $r>|t|$ (see Claim \ref{Cl:outsideWM}). If $j\in \JJJ_L$, we will denote $\Psi^j=\Psi^j_L$. Let \begin{equation}
\label{P30}
\Psi_{n}^j(t,r)=\Psi^j\left( \frac{t-t_{j,n}}{\lambda_{j,n}},\frac{r}{\lambda_{j,n}} \right).
\end{equation} 
\begin{prop}
 \label{P:approxWM}
 Let $\psi_n$ be the solution of \eqref{WM} with initial data $(\psi_{0,n},\psi_{1,n})$. Then, denoting,
 \begin{equation}
  \label{NLexpansionWM}
  \Theta_n^J(t,r)=\psi_n(t,r)-\sum_{j=1}^J \Psi_n^j(t,r)-\omega_{n}^J(t,r),\quad r>|t|,
 \end{equation} 
 one has for large $J$ that $\vec{\Theta}_n^J(t)\in \Hbf(|t|)$ for all $t$ and 
 $$\lim_{J\to\infty} \limsup_{n\to\infty} \left( \sup_{t} \left\|\vec{\Theta}_n^J(t)\right\|_{\Hbf(|t|)} +\left\|\Theta_n^J\right\|_{S(\{r>|t|\})}\right)=0.$$
\end{prop}
\begin{proof}
If we assume that the sequence $\left\{(\psi_{0,n},\psi_{1,n})\right\}_n$ is bounded in $\Hbf$, then the conclusion of the proposition can be obtained by considering the sequence of radial finite-energy solutions $\left\{\frac{1}{r}\psi_n\right\}_n$ of the $1+4-$dimensional nonlinear wave equation \eqref{gNLW} (see e.g. \cite[Proposition 2.15]{Cote15}). To work in the more general setting of Proposition \ref{P:approxWM},  we will need the space-time bound for general finite-energy solutions of \eqref{WM} given in Lemma \ref{L:L3L6sinus}.

We denote
\begin{equation}
 \label{P:40}
 \psi_n^J=\sum_{j=1}^J \Psi_n^J+\omega_n^J,\quad \Theta_n^J=\psi_n-\psi_n^J.
\end{equation} 
Then
\begin{multline}
 \label{P41}
 \partial_t^2\Theta_n^J-\partial_r^2\Theta_n^J-\frac{1}{r}\partial_r\Theta_n^J+\frac{1}{r^2}\Theta_n^J
\\ =-\frac{\sin(2\psi_n)-2\psi_n}{2r^2}+\sum_{\substack{1\leq j\leq J\\ j\in \JJJ_C}}\frac{\sin(2\Psi_n^j)-2\Psi_n^j}{2r^2},
 \end{multline}
 with initial data $(\Theta_n^J,\partial_t\Theta_n^j)_{\restriction t=0}=(0,0)$. Expanding the second line of \eqref{P41} and letting $\theta_n^j=\frac{1}{r}\Theta_n^j$, we rewrite it as a radial $1+4$-dimensional wave equation:
 \begin{multline}
 \label{eq_thetanJ}
 \partial_t^2\theta_n^J-\partial_r^2\theta_n^J-\frac{3}{r}\partial_r\theta_n^J\\
 = -\frac{1}{2r^3}\left( \sin\bigg( 2\omega_n^J+2\Theta_n^J +\sum_{j=1}^J2\Psi_n^j\bigg)-\sin\left( 2\omega_n^J \right) -\sin \left( 2\Theta_n^J\right)-\sum_{j=1}^J\sin\left( 2\Psi_n^j \right)\right)\\
 +\frac{1}{r^3} \sum_{\substack{1\leq j\leq J\\ j\in \JJJ_L}}\left( \Psi_n^j-\frac{\sin(2\Psi_n^j)}{2}\right)+\frac{1}{r^3} \left( \Theta_n^J-\frac{\sin(2\Theta_n^J)}{2} \right)+\frac{1}{r^3}\left( \omega_n^J-\frac{\sin(2\omega_n^J)}{2} \right).
\end{multline} 
Let $\Gamma_T:=\{(t,r)\;:\; 0<|t|\leq \min(r,T)\}$. We will prove 
\begin{multline}
 \label{P52}
 \forall T>0,\quad
 \left\|\left(\partial_t^2\theta_n^J-\partial_r^2\theta_n^J-\frac{3}{r}\partial_r\theta_n^J\right)\indic_{\Gamma_T}\right\|_{L^1L^2}\\
 \leq \int_0^T\left\|\theta_n^J(t)\right\|_{L^6(|t|)}F_n(t)dt+C\int_0^T\|\theta_n^J(t)\|^3_{L^6(|t|)}dt+\eps_n^J,
\end{multline} 
where $\eps_n^J$ is independent of $T$ and is such that
\begin{equation}
 \label{P53}
 \lim_{J\to\infty} \limsup_{n\to\infty}\eps_n^J=0,
\end{equation} 
and $F_n$ satisfies
\begin{equation}
 \label{P54}
 F_n\in L^{3/2}(\RR),\quad \sup_{n}\|F_n\|_{L^{3/2}(\RR)}<\infty.
\end{equation} 
In \eqref{P52} the notation $L^6(|t|)$ means as above
$$\|f\|_{L^6(|t|)}=\left(\int_{|t|}^{\infty}\left(f(r)\right)^6r^3dr\right)^{1/6}.$$
We first assume \eqref{P52}, \eqref{P53} and \eqref{P54} and prove the proposition. In all the sequel we assume that $J$ and $n$ are taken large enough, so that $\eps_n^J$ is small. By \eqref{P52}, Strichartz estimates and H\"older's inequality,
$$ \left(\int_0^T \|\theta_n^J(t)\|^3_{L^6(|t|)}dt\right)^{1/3} \leq \eps_n^J+\int_0^T \|\theta_n^J(t)\|_{L^6(|t|)} \left( F_n(t)+C\|\theta_n^J(t)\|^2_{L^6(|t|)} \right)dt.$$
Using \eqref{P54} and a Gr\"onwall-type inequality (see the appendix of  \cite{FaXiCa11}), we obtain that for all $T>0$,
$$\left(\int_0^T \|\theta_n^J(t)\|_{L^6(|t|)}^3dt\right)^{1/3}\leq \eps_n^J \Gamma \left(C+C\Big(\int_0^T \|\theta_n^J\|^3_{L^6(|t|)}dt\Big)^{2/3}  \right),$$
where $\Gamma$ is the standard Gamma function.

Assuming $\eps_n^J\Gamma(C+1)\leq 1/C$, we obtain that
\begin{multline*}
\Big(\int_0^T \|\theta_n^J\|^3_{L^6(|t|)}dt\Big)^{2/3}\leq 1/C\\
\Longrightarrow  \left(\int_0^T \|\theta_n^J(t)\|_{L^6(|t|)}^3dt\right)^{1/3}\leq \eps_n^J \Gamma \left(C+1\right)\leq 1/C,
\end{multline*}
and a standard bootstrap argument yields
$\Big(\int_0^T \|\theta_n^J\|^3_{L^6(|t|)}dt\Big)^{1/3}\lesssim \eps_n^J$ for all $T$. Using \eqref{P52} and again Strichartz estimates we see that it implies
$$\sup_{t\in \RR} \left\|\vec{\theta}_n^J(t)\right\|_{\HHH(|t|)}\lesssim \eps_n^J.$$
Going back to the function $\Theta_n^j(t)=r\theta_n^j$, we deduce the conclusion of the proposition. We are thus left with proving \eqref{P52}. In view of \eqref{eq_thetanJ}, it is sufficient to prove the following inequalities
\begin{gather}
 \label{P70}
 \lim_{n\to \infty} \left\|\indic_{\{r>|t|\}} \frac{1}{r^3} \left(\Psi_n^j-\frac{\sin(2\Psi_n^j)}{2}\right)\right\|_{L^1L^2}=0,\quad \forall j\in \JJJ_L\\
\label{P71}
 \lim_{J\to\infty}\limsup_{n\to \infty} \left\|\indic_{\{r>|t|\}} \frac{1}{r^3} \left(\omega_n^j-\frac{\sin(2\omega_n^j)}{2}\right)\right\|_{L^1L^2}=0\\
 \label{P72}
 \left\|\indic_{\Gamma_T} \frac{1}{r^3} \left(\Theta_n^J-\frac{\sin(2\Theta_n^J)}{2}\right)\right\|_{L^1L^2}\lesssim
\left\|\indic_{\Gamma_T} \theta_n^J\right\|_{L^3L^6}^3,\quad \forall J\gg 1\\
\label{P60}
\left\| \frac{\indic_{\Gamma_T}}{2r^3}\left( \sin\bigg( 2\omega_n^J+2\Theta_n^J +\sum_{j=1}^J2\Psi_n^j\bigg)-\sin\left( 2\omega_n^J \right) -\sin \left( 2\Theta_n^J\right)-\sum_{j=1}^J\sin\left( 2\Psi_n^j \right)\right)\right\|_{L^1L^2} \\ 
\notag
\qquad \qquad \lesssim \int_0^T \left\|\theta_n^J(t)\right\|_{L^6(|t|)}F_n(t)dt+\teps_n^J+\left\|\indic_{\Gamma_T}\theta_n^J\right\|^3_{L^3L^6}, \quad \forall J\gg 1,
 \end{gather}
where $\tilde{\eps}_n^J$ satisfies \eqref{P53}, $F_n$ satisfies \eqref{P54}, and the space-time $L^pL^q$ norms are always with respect to $r^3dr$ and $dt$.

\medskip

\noindent\emph{Proofs of \eqref{P70}, \eqref{P71} and \eqref{P72}.} We use the bound 
\begin{equation}
\label{bnd_sin}
\forall \sigma\in \RR,\quad 
\left|\sigma-\frac{\sin(2\sigma)}{2}\right|\lesssim \sigma^3.
\end{equation}

By \eqref{bnd_sin} and the property \eqref{P22} of $\omega_n^J$, we obtain \eqref{P71}.
The inequality \eqref{P72} follows directly from \eqref{bnd_sin} and H\"older's inequality. 

Let $j\in \JJJ_L$. Then by \eqref{bnd_sin} and H\"older's inequality,
\begin{multline}
\label{CVD}
 \left\|\frac{1}{r^3} \left( \Psi_n^j-\frac{\sin(2\Psi_n^j)}{2} \right)\indic_{\{r>|t|\}}\right\|_{L^1L^2}\\
 \lesssim 
 \int_{\RR}\left( \int_{|t|}^{\infty} \frac{1}{r^3} \left( \Psi^j\left( \frac{t-t_{j,n}}{\lambda_{j,n}},\frac{r}{\lambda_{j,n}} \right) \right)^6dr \right)^{1/2}dt\\
 =\int_{\RR} \left( \int_{\left|\frac{t_{j,n}}{\lambda_{j,n}}+\tau\right|}^{\infty} \frac{1}{\rho^3} \left( \Psi^j(\tau,\rho) \right)^6d\rho\right)^{1/2}d\tau, 
\end{multline}
which tends to $0$ as $n\to\infty$ by dominated convergence, and since $\lim_{n\to\infty} \frac{t_{j,n}}{\lambda_{j,n}}+\tau\in \pm\infty$ for every $\tau\in \RR$ if $j\in \JJJ_L$. Hence \eqref{P70}.

\medskip

\noindent\emph{Proof of \eqref{P60}.} By the Claim \ref{Cl:sin}, we have
\begin{multline*}
\left| \sin\bigg( 2\omega_n^J+2\Theta_n^J +\sum_{j=1}^J2\Psi_n^j\bigg)-\sin\left( 2\omega_n^J \right) -\sin \left( 2\Theta_n^J\right)-\sum_{j=1}^J\sin\left( 2\Psi_n^j \right)\right|\\ \lesssim 
\left|\sin \omega_n^J\right|^3+\left|\sin \Theta_n^J\right|^3
+\left(\left|\sin \omega_n^J\right|+\left|\sin \Theta_n^J\right|\right)\sum_{j=1}^J \sin^2\Psi_n^j+\sum_{\substack{1\leq j,k\leq J\\j\neq k}} \sin^2\Psi_n^j\left|\sin \Psi_n^k\right|,
\end{multline*}
Since $|\sin \omega_n^J|^3\lesssim |\omega_n^J|^3$, we have by \eqref{P22}
\begin{equation}
 \label{P80}\lim_{J\to\infty}\limsup_{n\to\infty} 
 \left\|\frac{1}{r^3} \sin^3\omega_n^J\right\|_{L^1L^2}=0.
\end{equation}
Furthermore,
\begin{equation}
 \label{P81}
 \left\|\indic_{\Gamma_T}\frac{1}{r^3} \sin^3 \Theta_n^J\right\|_{L^1L^2}\lesssim \left\|\indic_{\Gamma_T}\theta_n^J\right\|_{L^3L^6}^3.
\end{equation} 
We will next prove
\begin{equation}
 \label{P100}
 \left\|\frac{1}{r^3}\indic_{\Gamma_T}\sin\Theta_n^J\sum_{j=1}^J (\sin\Psi_n^j)^2\right\|_{L^1L^2}\lesssim \int_0^T \left\|\theta_n^J(t)\right\|_{L^6(|t|)} F_n(t)dt,
 \end{equation}
 (where $F_n$ satisfies \eqref{P54}), and
 \begin{gather}
 \label{P101}
 \lim_{J\to\infty}\limsup_{n\to\infty} \left\|\frac{1}{r^3}\indic_{\{r>|t|\}} (\sin\omega_n^J)\sum_{j=1}^J \sin^2\Psi_n^j\right\|_{L^1L^2}=0\\
 \label{P120}
 j\neq k\Longrightarrow \lim_{n\to\infty}\left\| \frac{1}{r^3} \indic_{\{r>|t|\}}\sin^2\Psi_n^j\sin\Psi_n^k\right\|_{L^1L^2}=0,
\end{gather}
which will complete the proof of \eqref{P60} and thus of the proposition. We first prove that there is a constant $C>0$ independent of $n$ such that
\begin{equation}
\label{P82}
\sum_{j=1}^{\infty} \left\|\indic_{\{r>|t|\}} \frac 1r\sin(\Psi_n^j)\right\|_{L^3L^6}^2\leq C. 
\end{equation}
Let $J_0\geq 1$ such that for all $j\geq J_0$, $(\Psi_0^j,\Psi_1^j)\in \Hbf$ and $E_M(\Psi_0^j,\Psi_1^j)$ is small. By the Claim \ref{Cl:smallWM}, for $j\geq J_0$, $E_M(\Psi_0^j,\Psi_1^j)\approx \|(\Psi_0^j,\Psi_1^j)\|_{\Hbf}^2$. By Strichartz estimates (and the small data theory for \eqref{gNLW} if $j\in \JJJ_C$), $\left\|\indic_{\{\{r>|t|\}} \frac{1}{r}\Psi_n^j\right\|_{L^3L^6}\lesssim \left\|(\Psi_0^j,\Psi_1^j)\right\|_{\Hbf}$ for large $j$. Since $|\sin \Psi_n^j|\leq |\Psi_n^j|$ we deduce, by the Pythagorean expansion \eqref{PytWM} of the energy, 
\begin{equation}
 \label{P90}
\sum_{j=J_0}^{\infty} \left\|\indic_{\{r>|t|\}} \frac{\sin(\Psi_n^j)}{r}\right\|_{L^3L^6}^2\leq C.  
\end{equation} 
Let $j\in \{1,\ldots,J_0-1\}$. If $j\in \JJJ_C$,
\begin{equation}
 \label{P91}
\left\|\indic_{\{r>|t|\}} \frac{\sin(\Psi_n^j)}{r}\right\|_{L^3L^6}=\left\|\indic_{\{r>|t|\}} \frac{\sin(\Psi^j)}{r}\right\|_{L^3L^6},
\end{equation} 
which is finite by Lemma \ref{L:L3L6sinus}, and if $j\in \JJJ_L$,
\begin{equation}
 \label{P92}
 \left\|\indic_{\{r>|t|\}}\frac{\sin \Psi_n^j}{r}\right\|_{L^3L^6}\leq \left\|\indic_{\{r>|t|\}}\frac{\Psi_L^j}{r}\right\|_{L^3L^6},
\end{equation} 
which is finite by Strichartz estimates. Thus \eqref{P82} holds.

\medskip

\noindent\emph{Proof of \eqref{P100}}. We have 
\begin{multline}
\label{P130}
\left\|\frac{1}{r^3}\indic_{\Gamma_T}\sin\Theta_n^J\sum_{j=1}^J \sin^2\Psi_n^j\right\|_{L^1L^2}=\int_0^T 
 \bigg\|\frac{1}{r^3}\sin\Theta_n^J(t)\sum_{j=1}^J \sin^2\Psi_n^j(t)\bigg\|_{L^2(|t|)}dt\\
 \leq \int_0^T \Big\|\frac{1}{r}\Theta_n^J(t)\Big\|_{L^6(|t|)} \bigg\|\frac{1}{r^2}\sum_{j=1}^J \sin^2\Psi_n^j(t)\bigg\|_{L^3(|t|)}dt.
\end{multline}
Furthermore,
\begin{equation}
\label{P131}
\bigg\|\frac{1}{r^2}\sum_{j=1}^J \sin^2\Psi_n^j(t)\bigg\|_{L^3(|t|)}\leq \sum_{j=1}^J \left\| \frac{\sin^2\Psi_n^j(t)}{r^2}\right\|_{L^3(|t|)}=\sum_{j=1}^J \left\| \frac{\sin\Psi_n^j(t)}{r}\right\|^2_{L^6(|t|)}. 
\end{equation} 
Let 
$$F_n(t)=\sum_{j=1}^{\infty} \left\|\frac{\sin \Psi_n^j(t)}{r}\right\|^2_{L^6(|t|)}.$$
Then 
\begin{equation}
\label{P132}
\|F_n\|_{L^{3/2}}\leq \sum_{j=1}^{\infty} \left\|\indic_{\{r>|t|\}}\frac{\sin\Psi_n^J}{r}\right\|^2_{L^3L^6}\leq C
\end{equation}
by \eqref{P90}. Combining \eqref{P130}, \eqref{P131} and \eqref{P132} we obtain \eqref{P100}.

\medskip

\noindent\emph{Proof of \eqref{P101}}. We have
\begin{multline*}
\left\|\frac{\indic_{\{r>|t|\}}}{r^3} (\sin\omega_n^J)\sum_{j=1}^J \sin^2\Psi_n^j\right\|_{L^1L^2}
\\
\leq \left\|\frac{\indic_{\{r>|t|\}}}{r} \sin\omega_n^J\right\|_{L^3L^6}\left\|\frac{\indic_{\{r>|t|\}}}{r^2}\sum_{j=1}^J \sin^2\Psi_n^j\right\|_{L^{3/2}L^3}\\
\leq \left\|\frac{\indic_{\{r>|t|\}}}{r} \sin\omega_n^J\right\|_{L^3L^6}\sum_{j=1}^{\infty} \left\|\indic_{\{r>|t|\}}\frac{\sin\Psi_n^J}{r}\right\|^2_{L^3L^6},
\end{multline*}
and \eqref{P101} follows from
\eqref{P22} and \eqref{P82}.

\medskip

\noindent\emph{Proof of \eqref{P120}}. If $j\in \JJJ_L$, we have by dominated convergence, as in \eqref{CVD},
$$\lim_{n\to\infty}\left\|\indic_{\{r>|t|\}} \frac{1}{r}\sin \Psi_n^j\right\|_{L^3L^6}=0.$$
Thus \eqref{P120} holds if $j\in \JJJ_L$ or $k\in \JJJ_L$. If $j\in \JJJ_C$ and $k\in \JJJ_C$, \eqref{P120} follows from
$$\lim_{n\to\infty}\frac{\lambda_{j,n}}{\lambda_{k,n}}+\frac{\lambda_{k,n}}{\lambda_{j,n}}=\infty,$$
and Lemma \ref{L:L3L6sinus}. 
\end{proof}

\subsection{Proof of the soliton resolution}
\label{SS:solitonWM}
In this subsection, we deduce Theorem \ref{T:mainWM} from the rigidity theorem (Theorem \ref{T:g_rigidity}). The proof is close to the proof of Theorem \ref{T:main} and we will not give all the details. As in the proof of Theorem \ref{T:main}, we will focus on the finite time blow-up case, the proof for global solutions being very similar.

\subsubsection{Analysis along a sequence of times}

We let $\psi$ be a solution of \eqref{WM} such that $T_+=T_+(\psi)<\infty$, and $(\varphi_0,\varphi_1)$ as in Proposition \ref{P:BlowupWM}. The core of the proof of Theorem \ref{T:mainWM} is the following proposition, which is the analog of Proposition \ref{P:subsequence}. Recall from \eqref{defH} the definition of the norm $\|\cdot\|_H$.
\begin{prop}
 \label{P:subsequenceWM}
Let $\{t_n\}_n$ be a sequence of times such that 
\begin{equation*}
\lim_{n\to\infty} t_n=T_+,\quad \forall n,\; 0\leq t_n<T_+.
\end{equation*}
Then there exists a subsequence of $\{t_n\}_n$ (that we still denote by $\{t_n\}_n$), $J\geq 1$, $(\iota_j)_j\in \{\pm 1\}^J$, sequences $\{\lambda_{j,n}\}_n$, $1\leq j\leq J$ with 
$$\lambda_{1,n}\ll \ldots \ll \lambda_{J,n}\ll T_+-t_n,$$
such that
\begin{gather}
\label{CVH1WM}
\lim_{n\to \infty} \left\|\psi(t_n)-\varphi_0 -\sum_{j=1}^J \frac{\iota_j}{\lambda_{j,n}} Q\left( \frac{\cdot}{\lambda_{j,n}} \right)\right\|_{H}=0\\
\label{dev_energyWM}
E_M(\psi_0,\psi_1)=E_M(\psi_0,\psi_1)+ J E_M(Q,0)+ \frac{1}{2}\left\|\partial_t \psi(t_n)-\psi_1\right\|_{L^2}^2+o_n(1).
\end{gather}
\end{prop}
\begin{proof}
 The proof of Proposition \ref{P:subsequenceWM} is close to the proof of Proposition \ref{P:subsequence}, but somehow easier, since solutions of the co-rotational  maps equation \eqref{WM} with initial data at $t=0$ are always defined for $\{r>|t|\}$ and enjoy a global space-time bound there (see Lemma \ref{L:L3L6sinus}).

 To simplify notations, we assume without loss of generality in all the proof that 
 $$T_+(\psi)=1.$$
 
 \noindent\emph{Step 1. Profile decomposition.} We let $\varphi$ be the solution of \eqref{WM} with initial data $(\varphi_0,\varphi_1)$ at $t=1$. By Proposition \ref{P:BlowupWM}, 
 $$\supp (\varphi-\psi)\subset \{r\leq 1-t\}.$$
 Extracting subsequences, we assume that $\Big\{\vec{\psi}(t_n)-(\varphi_0,\varphi_1)\Big\}_n$ has a profile decomposition as in \S \ref{SSS:profile}. We will use the notations of \S \ref{SSS:profile}: $\Psi_n^j$, $\{\lambda_{j,n}\}_n$, $\{t_{j,n}\}_n$ etc... for this decomposition. Recall the partition $\NN\setminus\{0\}=\JJJ_C\cup \JJJ_L$ of the set of indices defined after Proposition \ref{P:profileWM}. Rescaling the profiles, we can assume that there is another partition $\NN\setminus\{0\}=\III_s\cup \III_c\cup\III_+\cup\III_-$ defined by
\begin{align*}
 j\in \III_s&\iff j\in \JJJ_C,\; \Psi_0^j\in \{Q,-Q\}\cup \{m\pi,\; m\in \ZZ\}\text{ and } \Psi_1^j=0,\\
 j\in \III_c&\iff j\in \JJJ_C\text{ and } \Psi^j\text{ is not a stationary solution of \eqref{NLW}},\\
 j\in \III_{\pm}&\iff j\in \JJJ_L\text{ and }\lim_{n\to\infty}\frac{-t_{j,n}}{\lambda_{j,n}}=\pm \infty.
\end{align*} 
In Step 2, we will prove that $\III_c=\emptyset$, in Step 3, that $\III_{\pm}=\emptyset$. In Step 4, we exclude self-similar behaviour for the scaling parameters. In Step 5, we conclude the proof, showing that the first component of the dispersive remainder $\omega_{0,n}^J$ goes to $0$ in $H$ as $n$ goes to infinity.

\smallskip

\noindent\emph{Step 2. Compact profiles.} We prove that $\III_c$ is empty by contradiction. Assume that it is not empty and let $j_0\in \JJJ_c$. Assume $(\Psi_0^{j_0},\Psi_1^{j_0})\in \Hbf_{\ell,m}$. 

By Theorem \ref{T:rigidityWM}, there exists $\eta_1>0$ such that the following holds for all $t\geq 0$ or for all $t\leq 0$:
\begin{equation}
 \label{WM20}
\int_{|t|}^{\infty} \left((\partial_{t,r}\Psi^{j_0}(t,r))^2+\frac{1}{r^2}\sin^2\Psi^{j_0} (t,r)\right)rdr \geq \eta_1.
 \end{equation} 

The sequence $\big\{\vec{\psi}(t_n)\big\}_n$ has the same profile decomposition as the sequence $\big\{\vec{\psi}(t_n)-(\varphi_0,\varphi_1)\big\}_n$, with an additional profile
$$(\Psi_0^0,\Psi_1^0)=(\varphi_0,\varphi_1),\quad \forall n,\; \lambda_{j_0,n}=1,\; t_{j_0,n}=0.$$
By Proposition \ref{P:approxWM} and arguing as in Step 3 of the proof of Proposition \ref{P:subsequence}, we obtain, for all $r>|\tau|$, 
\begin{equation}
 \label{WM30}
 \psi(t_n+\tau)=\varphi(t_n+\tau)+\sum_{j=1}^J \Psi_n^j(\tau,r)+\omega_n^J(\tau,r)+\Theta_n^J(\tau,r), 
\end{equation} 
where 
\begin{equation}
 \label{WM31}
 \lim_{J\to\infty} \limsup_{n\to\infty} \sup_{\tau} \left\| \vec{\Theta}_n^J(\tau)\right\|_{\Hbf(\tau)}=0.
\end{equation}
We distinguish between two cases.

\smallskip

\noindent \emph{Case 1. Channels of energy in the future.} Assume that \eqref{WM20} holds for all $t>0$. Then by \eqref{WM30} at $\tau=1-t_n$,
$$\sum_{j=1}^J \Psi_n^j(1-t_n,r)+\omega_n^J(1-t_n,r)+\Theta_n^J(1-t_n,r)=0, \quad r>1-t_n,$$
which yields
\begin{multline}
 \label{WM32} 
 \int_{1-t_n}^{\infty} \Big( \partial_{t,r}\Big( \sum_{j=1}^J \Psi_n^j(1-t_n,r)+\omega_n^J(1-t_n,r)+\Theta_n^J(1-t_n,r) \Big) \Big)^2rdr\\
 +\int_{1-t_n}^{\infty}\sin^2\Big( \sum_{j=1}^J \Psi_n^j(1-t_n,r)+\omega_n^J(1-t_n,r)+\Theta_n^J(1-t_n,r) \Big)\frac{dr}{r}=0.
\end{multline}
Fixing $J$ large, and using that \eqref{WM20} holds for $t\geq 0$, we claim that the left-hand side of \eqref{WM32} is larger than 
$$ \frac{1}{2}\int_{1-t_n}^{\infty}\left( \partial_{t,r}\Psi_n^{j_0}(1-t_n,r) \right)^2rdr +\frac{1}{2}\int_{1-t_n}^{\infty}\frac{\sin^2\Psi_n^{j_0}(1-t_n,r)}{r} dr\geq \frac{\eta_1}{2}$$
for large $n$, a contradiction. In view of \eqref{WM31}, and Claim \ref{Cl:sin}, it is sufficient to prove 
\begin{gather}
 \label{WM40}
 j\ne j_0\Longrightarrow \lim_{n\to\infty} \int_{1-t_n}^{\infty} \partial_{t,r}\Psi_n^j(1-t_n,r)\cdot\partial_{t,r}\Psi_n^{j_0}(1-t_n,r)rdr=0 \\
 \label{WM40bis}
 j\ne j_0\Longrightarrow \lim_{n\to\infty} \int_{1-t_n}^{\infty} \left|\sin \Psi_n^j(1-t_n,r)\sin \Psi_n^{j_0}(1-t_n,r)\right|\frac{dr}{r}=0 \\
 \label{WM41}
 1\leq j\leq J\Longrightarrow \lim_{n\to\infty} \int_{1-t_n}^{\infty} \partial_{t,r}\Psi_n^j(1-t_n,r)\cdot\partial_{t,r}\omega_n^{J}(1-t_n,r)rdr=0 \\
 \label{WM41bis}
 1\leq j\leq J\Longrightarrow \lim_{n\to\infty} \int_{1-t_n}^{\infty} \left|\sin \Psi_n^j(1-t_n,r)\sin \omega_n^{J}(1-t_n,r)\right|\frac{dr}{r}=0.
\end{gather}
The proof of \eqref{WM40},\ldots,\eqref{WM41} is very close to the proofs of the corresponding properties for the linear wave profiles in Appendix \ref{AA:psdo_ortho} and of the Pythagorean expansion \eqref{PytWM} of the wave maps energy in Appendix \ref{AA:psdo_orthoWM} and we omit it. 

\smallskip

\noindent\emph{Case 2. Channels of energy in the past.}
We assume that \eqref{WM20} holds for all $t\leq 0$. By \eqref{WM30} at $\tau=-t_n$,
\begin{multline}
 \label{WM50} 
 \int_{t_n}^{\infty} \Big| \partial_{t,r}\Big( \sum_{j=1}^J \Psi_n^j(-t_n,r)+\omega_n^J(-t_n,r)+\Theta_n^J(-t_n,r) \Big) \Big|^2rdr\\
 +\int_{t_n}^{\infty}\sin^2\Big( \sum_{j=1}^J \Psi_n^j(-t_n,r)+\omega_n^J(1-t_n,r)+\Theta_n^J(-t_n,r) \Big)\frac{dr}{r}\\
 =\int_{t_n}^{\infty} \left| \partial_{t,r}\psi(0,r)-\partial_{t,r}\varphi(0,r) \right|^2rdr+\int_{t_n}^{\infty} \sin^2(\psi(0,r)-\varphi(0,r))\frac{dr}{r}\underset{n\to\infty}{\longrightarrow}0.
\end{multline}
By the same argument as in Case 1, we obtain that the left-hand side of \eqref{WM50} is larger than
$$ \frac{1}{2}\int_{t_n}^{\infty}\left( \partial_{t,r}\Psi_n^{j_0}(-t_n,r) \right)^2rdr +\frac{1}{2}\int_{t_n}^{\infty}\frac{\sin^2\Psi_n^{j_0}(-t_n,r)}{r} dr\geq \frac{\eta_1}{2},$$
a contradiction.

\smallskip

\noindent\emph{Step 3. Scattering profiles.} In this step we prove by contradiction that $\III_+$ is empty. This step is very close to Case 1 of the preceding step, and also to Step 4 in the proof of Proposition \ref{P:subsequence}, and we only sketch it. The proof that $\III_-$ is empty is very similar and we omit it. 

Assume that $\III_+\neq \emptyset$ and let $j_0\in \III_+$. Since $j_0\in \JJJ_L$, we have 
$$\Psi_n^{j_0}(t,r)=\Psi^{j_0}\left( \frac{t-t_{j_0,n}}{\lambda_{j_0,n}},\frac{r}{\lambda_{j_0,n}} \right)$$
where $\Psi^{j_0}$ is a solution of the linearized equation \eqref{LWM}, with nonzero initial data in $\Hbf$. Then $U^{j_0}=\frac{1}{r}\Psi^{j_0}$ is a radial solution of the linear wave equation in space dimension $4$, and we deduce as in Step 4 of the proof of Proposition \ref{P:subsequence} that there exists $\eta_1>0$, such that
for large $n$, 
$$ \int_{1-t_n}^{\infty} \left|\partial_{t,r}U_n^{j_0}(1-t_n,r)\right|^2r^3dr\geq 2\eta_1,$$
where $U_n^{j_0}(t,r)=\frac{1}{\lambda_{j_0,n}}U^{j_0}\left( \frac{t-t_{j_0,n}}{\lambda_{j_0,n}},\frac{r}{\lambda_{j_0,n}} \right)$. As a consequence, for large $n$
$$ \int_{1-t_n}^{\infty} \left|\partial_{t,r}\Psi_n^{j_0}(1-t_n,r)\right|^2rdr+\int_{1-t_n}^{\infty} \left|\Psi_n^{j_0}(1-t_n,r)\right|^2\frac{dr}{r}  \geq 2\eta_1.$$
Using that $\lim_{n\to\infty}\|\Psi_n^{j_0}(1-t_n)\|_{L^{\infty}}=0$ (see \eqref{Linfty0}), we deduce that for large $n$,
$$ \int_{1-t_n}^{\infty} \left|\partial_{t,r}\Psi_n^{j_0}(1-t_n,r)\right|^2rdr+\int_{1-t_n}^{\infty} \sin^2\Psi_n^{j_0}(1-t_n,r)\frac{dr}{r}  \geq \eta_1.$$
The end of the proof is the same as in Case 1 of Step 2 above, and we omit. 

\smallskip

\noindent\emph{Step 4. Bound on the scaling parameters.} According to the preceding steps, $(\omega_{0,n}^J,\omega_{1,n}^J)$ is, for large $J$, independent of $J$, and we will drop the superscript $J$. Furthermore, all the profiles are in $\III_s$. Using the Pythagorean expansion of the energy \eqref{PytWM} for the sequence $\vec{\psi}(t_n)$, we see that there exists $J_0\geq 0$, $m\in \ZZ$, $(\iota_{j})_{1\leq j\leq J_0}\in \{\pm 1\}^{J_0}$, sequences $\{\lambda_{j,n}\}_n$ with $0<\lambda_{1,n}\ll \ldots \ll \lambda_{J_0,n}$ such that 
\begin{equation}
 \label{WM60}
 \vec{\psi}(t_n)=\vec{\varphi}(t_n)+(m\pi,0)+\sum_{j=1}^{J_0} \left( \iota_jQ\left( \frac{\cdot}{\lambda_{j,n}} \right),0 \right)+\left( \omega_{0,n},\omega_{1,n} \right)+o_n(1) \text{ in }\Hbf.
\end{equation} 
Using that $(\omega_{0,n},\omega_{1,n})\in \Hbf$, and that $\psi(t_n,r)=\varphi(t_n,r)$ for $r>1-t_n$, we see that $m=-\sum_{j=1}^{J_0} \iota_{j}$. In this step, we prove
\begin{equation}
\label{not_SelfS}
\lim_{n\to\infty} \frac{\lambda_{J_0,n}}{1-t_n}=0. 
\end{equation} 
Arguing by contradiction and extracting subsequences, we assume that there exists $R>0$ such that for large $n$,
$$ \lambda_{j_0,n}\geq R(1-t_n).$$
By \eqref{WM60}, and the fact that $\psi(t_n,r)-\varphi(t_n,r)=0$ for $r>1-t_n$, we obtain that  
\begin{equation}
 \label{WM60bis}
\omega_{0,n}(\lambda_{J_0,n}r)+m\pi+\sum_{j=1}^{J_0-1} \iota_j Q\left( \frac{\lambda_{J_0,n}r}{\lambda_{j,n}} \right)+Q(r)=o_n(1)\text{ in }H(R).
 \end{equation} 
Consider $\chi\in C_0^{\infty}\big((R,\infty)\big)$, nonnegative and not identically zero. Since $\omega_{0,n}(\lambda_{J_0,n}\cdot)$ converges to $0$ weakly in $H$, we have
$$\lim_{n\to\infty}\int_0^{\infty} \omega_{0,n}(\lambda_{J_0,n}r)\chi(r)dr=0.$$
Furthermore, for all $j\in \{1,\ldots,J_0-1\}$, we have
$$\lim_{n\to\infty} \iota_j \int_0^{\infty}Q\left( \frac{\lambda_{J_0,n}r}{\lambda_{j,n}} \right)\chi(r) dr=\iota_j\pi \int_0^{\infty}\chi(r)dr.$$
Thus \eqref{WM60bis} implies, letting $n\to\infty$,
$$ \bigg(m+\sum_{j=1}^{J_0-1}\iota_j\bigg)\pi\int_0^{\infty} \chi(r)dr+\iota_{J_0}\int_{0}^{\infty}Q(r)\chi(r)dr=0.$$
Using that $m+\sum_{j=1}^{J_0-1}\iota_j=-\iota_{J_0}$, we deduce
$$\iota_{J_0}\int_{0}^{\infty}(Q-\pi)(r)\chi(r)dr=0, $$
a contradiction since $Q<\pi$.

\smallskip

\noindent\emph{Step 5. Dispersive remainder.}
In this step we prove
\begin{equation}
 \label{WM70}
 \lim_{n\to\infty} \|\omega_{0,n}\|_{H}=0
\end{equation} 
which will conclude the proof of Proposition \ref{P:subsequenceWM}. Using that $(\omega_{0,n},\omega_{1,n})\in \Hbf$, we obtain as usual that $w_n=\frac{1}{r}\omega_n$ is a radial finite energy solution of the free wave equation \eqref{FW} in space dimension $4$. Thus by the result of \cite{CoKeSc14} (see \eqref{CKS}),
\begin{equation*}
 \sum_{\pm} \lim_{t\to\pm\infty} \int_{|t|}^{\infty} \left|\partial_{t,r}w_n(t,r)\right|^2r^3dr\gtrsim \|w_{0,n}\|^2_{\hdot},
\end{equation*} 
where $w_{0,n}=\frac{1}{r}\omega_{0,n}$.
Using the decay of the exterior energy for the linear wave equation, we obtain that the following holds for all $t\geq 0$ or for all $t\leq 0$
\begin{equation*}
 \int_{|t|}^{\infty} \left|\partial_{t,r}w_n(t,r)\right|^2r^3dr\gtrsim \|w_{0,n}\|^2_{\hdot}.
\end{equation*} 
Going back to $\omega_n$, we deduce that the following holds for all $t\geq 0$ or for all $t\leq 0$.
\begin{equation*}
 \int_{|t|}^{\infty}\left| \partial_{t,r}\omega_n(t,r) \right|^2 rdr+\int_{|t|}^{\infty} \frac{1}{r^2} \left(\omega_n(t,r)\right)^2rdr\gtrsim \|\omega_{0,n}\|_{H}^2.
\end{equation*}
Since by \eqref{P22}, $\lim_n\|w_n\|_{L^{\infty}}=0$, we deduce that for all $t\geq 0$ or for all $t\leq 0$
\begin{equation*}
 \int_{|t|}^{\infty}\left| \partial_{t,r}\omega_n(t,r) \right|^2 rdr+\int_{|t|}^{\infty}  \sin^2\omega_n(t,r)\frac{dr}{r}\gtrsim \|\omega_{0,n}\|_{H}^2.
\end{equation*}
The end of the proof of \eqref{WM70} is similar to Step 2, and we omit it.
\end{proof}

\subsubsection{End of the proof of the soliton resolution}
The end of the proof of Theorem \ref{T:mainWM} in the finite time blow-up case is close to the end of the proof of Theorem \ref{T:main}. However, the fact that the energy always gives an a priori bound on a solution $\psi$ of the wave maps equation \eqref{WM} makes the proof slightly simpler. 
 
As before, we let $\psi$ be a solution of \eqref{WM} such that $T_+(\psi)=1$. By Proposition \ref{P:BlowupWM}, there exists a sequence of times $\{t_n\}\to 1$ such that 
\begin{equation}
 \label{WM80} 
 \lim_{n\to\infty}\|\partial_t \psi(t_n)-\varphi_1\|_{L^2}=0.
\end{equation} 
(Here and until the end of this proof, $L^2$ denotes the space $L^2\big((0,\infty)\big)$ with respect to the measure $rdr$.) 
 Combining \eqref{WM80} with Proposition \ref{P:subsequenceWM}, we obtain (extracting subsequences) that there exists $J\geq 0$, $(\iota_j)_j\in \{\pm 1\}^J$, scaling parameters $0<\lambda_{1,n}\ll \ldots \ll \lambda_{J,n}$ such that 
 \begin{equation}
  \label{WM90}
  \lim_{n\to\infty} \bigg\|\vec{\psi}(t_n)-(\varphi_0,\varphi_1)-(m\pi,0)-\sum_{j=1}^{J}\left(\iota_jQ(\lambda_{j,n}\cdot),0\right)\bigg\|_{\Hbf}=0.
 \end{equation} 
 Since $t_n\to 1$ and $1$ is the maximal time of existence of $\psi$, we must have $J\geq 1$. Using the Pythagorean expansion \eqref{PytWM} of the energy, we see that 
 \begin{equation}
  \label{WM91}
  E_M(\psi_0,\psi_1)=E_M(\varphi_0,\varphi_1)+JE_M(Q,0).
 \end{equation} 
 We next prove by contradiction
 \begin{equation}
  \label{WM92}
  \lim_{t\to 1} \|\partial_t\psi(t)-\varphi_1\|_{L^2}=0.
 \end{equation} 
If \eqref{WM92} does not hold, using \eqref{WM80} and the intermediate value theorem, there exists a sequence of times $t_n'\to 1$ such that 
\begin{equation}
 \label{WM93} \lim_{n\to\infty}\left\| \partial_t\psi(t_n')-\varphi_1\right\|_{L^2}=\eps_0.
\end{equation} 
Using Proposition \ref{P:subsequenceWM} along the sequence $\{t_n'\}_n$, we obtain (extracting subsequences) $J'\geq 0$, $(\iota_j')_{j}\in \{\pm 1\}^{J'}$, $0<\lambda_{1,n}'\ll \ldots\ll \lambda_{J',n}'$ such that 
\begin{equation}
 \label{WM94}
 \lim_{n\to\infty} \bigg\|\psi(t_n')-\varphi_0-m\pi-\sum_{j=1}^{J'} \iota_j'Q\bigg( \frac{\cdot}{\lambda_{j,n}'} \bigg)\bigg\|_{H}=0.
\end{equation} 
This yields 
\begin{equation}
 \label{WM95} 
 E_M(\psi_0,\psi_1)=E_M(\varphi_0,\varphi_1)+\frac 12\eps_0^2+J'E_M(Q,0).
\end{equation} 
Combining with \eqref{WM91}, we see that $\frac 12\eps_0^2 =(J-J')E_M(Q,0)$, which yields a contradiction since $\eps_0$ is small and not equal to $0$, concluding the proof of  \eqref{WM91}. 

Combining Proposition \ref{P:subsequenceWM}, the conservation of the energy and \eqref{WM92}, we obtain that $J=\frac{E_M(\psi_0,\psi_1)-E_M(\varphi_{0},\varphi_1)}{E_M(Q,0)}$ is a positive integer, and that
for all sequences of times $t_n\to 1$, there exists a subsequence of $\{t_n\}_n$, $(\iota_j)_j\in \{\pm 1\}^J$, and sequences $\{\lambda_{j,n}\}_n$, $1\leq j\leq J$ with 
$$0<\lambda_{1,n}\ll \ldots \ll \lambda_{J,n}\ll 1-t_n,$$
such that
\begin{equation}
\label{CVHWM}
\lim_{n\to \infty} \bigg\|\vec{\psi}(t_n)-(\varphi_0,\varphi_1) -m\pi-\sum_{j=1}^J \bigg(\iota_j Q\bigg( \frac{\cdot}{\lambda_{j,n}} \bigg),0\bigg)\bigg\|_{\Hbf}=0,
\end{equation}
where $m=\sum \iota_j$.

The fact that we can choose the $\iota_j$ (and thus $m$) independently of the sequence $\{t_n'\}_n$, and that the preceding statement implies the conclusion of Theorem \ref{T:mainWM} (with $\phi_0=\varphi_0-m\pi$, $\phi_1=\varphi_1$) is standard and we omit it (see the end of Section \ref{S:resolution} above, and Section 3.5 of \cite{DuKeMe13}).

\medskip

The proof of Theorem \ref{T:mainWM} in the case where $T_+(\psi)=+\infty$ is very close, using the scattering state $\psi_L$ (defined in Proposition \ref{P:globalWM}) instead of $(\varphi_0,\varphi_1)$. We again omit the details.
 \appendix 

\section{Radiation term for the free wave equation}
\label{A:radiation}
In this appendix we consider the linear inhomogeneous wave equation in any space dimension $N\geq 3$:
\begin{equation}
\label{ILW}
\left\{
\begin{aligned}
\partial_t^2u-\Delta u&=f,\quad (t,x)\in \RR\times \RR^N\\
\vec{u}_{\restriction t=0}&=(u_0,u_1),
\end{aligned}\right.
\end{equation}
with radial data.
\begin{prop}
\label{P:radiation}
 Let
$$(u_0,u_1)\in \HHH=(\hdot\times L^2)_{\rad}(\RR^N), \quad f\in L^1(\RR,L^2_{\rad}(\RR^N)),$$
 and $u$ be the corresponding solution of \eqref{ILW}. Then there exists $G\in L^2(\RR)$ such that
 \begin{align}
  \label{RT10}
\lim_{t\to\infty} \int_0^{+\infty} \left|r^{\frac{N-1}{2}}\partial_r u(t,r)-G(r-t)\right|^2dr&=0\\
\label{RT11}
\lim_{t\to\infty} \int_0^{+\infty} \left|r^{\frac{N-1}{2}}\partial_t u(t,r)+G(r-t)\right|^2dr&=0
\end{align} 
If furthermore $(u_0,u_1)\in \dot{H}^2\times \dot{H}^1$ and $\partial_t f\in L^1L^2$, then 
 \begin{align}
  \label{RT12}
\lim_{t\to\infty} \int_0^{+\infty} \left|r^{\frac{N-1}{2}}\partial_{r}\partial_t u(t,r)+G'(r-t)\right|^2dr&=0\\
\label{RT13}
\lim_{t\to\infty} \int_0^{+\infty} \left|r^{\frac{N-1}{2}}\partial_t^2 u(t,r)-G'(r-t)\right|^2dr&=0
\end{align} 
\end{prop}
\begin{proof}
 The fact that \eqref{RT10} and \eqref{RT11} hold in the case $f=0$ is well-known (see \cite[Theorem 2.1]{DuKeMe19}) and goes back at least to the work of Friedlander \cite{Friedlander62}.

 To prove \eqref{RT10} and \eqref{RT11} in the general case, one can reduce to the case $f=0$, by recalling that if $u$ is a solution of \eqref{ILW} with $f\in L^1(\RR,L^2(\RR^N))$, there exists a finite energy solution $v_L$ of the free wave equation $(\partial_t^2-\Delta)v_L=0$ such that 
 \begin{equation}
 \label{inhom_hom}
 \lim_{t\to\infty} \left\|\vec{u}(t)-\vec{v}_L(t)\right\|_{\HHH}=0.
 \end{equation} 
 This follows immediately from the Duhamel formulation of \eqref{ILW}, which yields that \eqref{inhom_hom} holds with
 \begin{multline*}
  v_L(t)=\cos(t\sqrt{-\Delta})\left(u_0-\int_0^{+\infty}\frac{\sin(s\sqrt{-\Delta})}{\sqrt{-\Delta}}f(s)ds\right)\\
  +\frac{\sin(t\sqrt{-\Delta})}{\sqrt{-\Delta}}\left( u_1+\int_0^{+\infty}\cos\left(s\sqrt{-\Delta}\right)f(s)ds \right).
 \end{multline*}
We next assume that $u_0\in \dot{H}^2$, $u_1\in \dot{H}^1$, $\partial_t f\in L^1(\RR,L^2)$ and prove \eqref{RT12}, \eqref{RT13}. We first note that $\partial_t u$ is solution of 
\begin{equation}
\label{ILW'}
\left\{
\begin{aligned}
\partial_t^2\partial_tu-\Delta \partial_tu&=\partial_t f,\quad (t,x)\in \RR\times \RR^N\\
\vec{u}_{\restriction t=0}&=(u_1,\Delta u_0+f(0))\in \HHH,
\end{aligned}\right.
\end{equation} 
(where $f(0)$ is well defined and in $L^2(\RR^N)$ by an elementary trace lemma). By the considerations above, there exists $H\in L^2(\RR)$ such that 
\begin{align}
\label{RT14}
\lim_{t\to\infty} \int_0^{+\infty} \left|r^{\frac{N-1}{2}}\partial_r \partial_tu(t,r)-H(r-t)\right|^2dr&=0\\
\label{RT15}
\lim_{t\to\infty} \int_0^{+\infty} \left|r^{\frac{N-1}{2}}\partial_t^2 u(t,r)+H(r-t)\right|^2dr&=0,
\end{align} 
and we are reduced to prove that $H=-G'$.

Let $\chi\in C^{\infty}(\RR)$ such that $\chi(\eta)=0$ if $\eta\leq \frac{1}{2}$ and $\chi(\eta)=1$ if $\eta\geq 1$. Since 
$$\lim_{t\to\infty} \int_0^1 |G(r-t)|^2dr=0,$$
the equality \eqref{RT11} implies
\begin{equation*}
 \lim_{t\to \infty} \int_{\RR} \left| \chi(r)r^{\frac{N-1}{2}} \partial_tu(t,r)+G(r-t)\right|^2dr=0,
\end{equation*} 
that is
\begin{equation*}
 \lim_{\tau\to \infty} \int_{\RR} \left| (\eta+\tau)^{\frac{N-1}{2}} \chi(\eta+\tau)\partial_tu(\tau,\eta+\tau)+G(\eta)\right|^2d\eta=0.
\end{equation*} 
In other words, the family of functions
$$\eta\mapsto (\eta+\tau)^{\frac{N-1}{2}}\chi(\eta+\tau)(\partial_t u)(\tau,\eta+\tau)$$
converge to $-G$ in $L^2(\RR)$ as $\tau$ goes to infinity. Taking the derivative in $\eta$ and using that 
$$\lim_{t\to\infty} \int_{0}^{+\infty}|\partial_t u(t,r)|^2r^{N-3}dr=0$$
(which follows from the fact that there exists a finite energy solution $w_L$ of $\partial_t^2w_L-\Delta w_L=0$
such that
$\lim_{t\to\infty} \|\partial_tu(t)-w_L(t)\|_{\hdot}=0,$ and a classical dispersive estimate, see e.g. \cite[Theorem 2.1]{DuKeMe19})
 we obtain  that the family of functions
$$\eta\mapsto (\eta+\tau)^{\frac{N-1}{2}}\chi(\eta+\tau)(\partial_t\partial_r u)(\tau,\eta+\tau)$$
converge to $-G'$ in $\DDD'(\RR)$ as $\tau\to \infty$. However by \eqref{RT14}, we obtain that this family converge to $H$ in $L^2(\RR)$ as $\tau\to\infty$. By uniqueness of the distributional limit, $H=-G'$ concluding the proof.
\end{proof}
\begin{remark}
 \label{R:radiation}
 We recall that for all $G\in L^2(\RR)$, there exists a unique finite energy solution $u$ of \eqref{ILW} such that \eqref{RT10} and \eqref{RT11} hold: see again \cite[Theorem 2.1]{DuKeMe19}.
\end{remark}

\section{Pseudo-ortogonality and channels of energy}
\label{A:psdo_ortho}

\subsection{Linear wave equation}
\label{AA:psdo_ortho}
In this subsection, we prove \eqref{Ps10} and \eqref{Ps11}

Since $\int_{\rho_n<|x|<\rho_n'}\ldots=\int_{\rho_n<|x|}\ldots-\int_{\rho_n'<|x|}$, it is sufficient to treat the case where $\rho_n'=\infty$ for all $n$. Rescaling $u_n$, we also assume $$\forall n,\quad \lambda_{j,n}=1.$$

\smallskip

\noindent \emph{Pseudo-orthogonality between two profiles.}
We first prove \eqref{Ps10}.
We will use the pseudo-orthogonality  \eqref{psdo_orth} of the scaling and time translation parameters.
We fix $j\neq k$. Let 
\begin{multline*}
\eps_n:=\int_{|x|>\rho_n}\nabla_{t,x}U^j_n(s_n,x)\cdot \nabla_{t,x}U^k_n(s_n,x)dx\\
=\int_{|x|>\rho_n}\nabla_{t,x}U^j\left(s_n-t_{j,n},x\right)\cdot \frac{1}{\lambda_{k,n}^2}\nabla_{t,x}U^k\left(\frac{s_n-t_{k,n}}{\lambda_{k,n}},\frac{x}{\lambda_{k,n}}\right)dx.
\end{multline*}
Arguing by contradiction, we can extract subsequences and assume that the following limits exist in $\RR\cup\{\pm\infty\}$.
\begin{equation*}
 \lim_{n\to\infty}s_n-t_{j,n}=T_j,\quad \lim_{n\to\infty} \frac{s_n-t_{k,n}}{\lambda_{k,n}}=T_k.
\end{equation*} 
\noindent\emph{Case 1: $T_j\in \RR$, $T_k\in \RR$.} In this case 
we have 
$$\eps_n=o_n(1)+\int_{|x|>\rho_n} \nabla_{t,x}U^j(T_j,x)\cdot\frac{1}{\lambda_{k,n}^2}\nabla_{t,x} U^k\left( T_k,\frac{x}{\lambda_{k,n}} \right)dx.$$
Since \eqref{psdo_orth} implies $\lim_n\lambda_{k,n}\in \{0,\infty\}$, we obtain right away $\lim_{n}\eps_n=0$.

\smallskip

\noindent\emph{Case 2. $T_j=+\infty$, $T_k\in \RR$.} We note that in this case, there exists $g^j\in L^2(\RR)$ such that 
\begin{multline}
 \label{Ps20} \eps_n=o_n(1)+\int_{\rho_n}^{\infty} g^j(r-s_n+t_{j,n})\frac{1}{\lambda_{k,n}^2} \partial_tU^k\left(T_k,\frac{r}{\lambda_{k,n}}\right)r^{3/2}dr\\
 -\int_{\rho_n}^{\infty} g^j(r-s_n+t_{j,n})\frac{1}{\lambda_{k,n}^2} \partial_rU^k\left(T_k,\frac{r}{\lambda_{k,n}}\right)r^{3/2}dr.
\end{multline} 
Indeed if $\lim_{n\to\infty} t_{j,n}\in \{\pm\infty\}$, $U^j$ is a solution of the free wave equation \eqref{FW}, and \eqref{Ps20} follows from the asymptotic behaviour for these solutions, recalled in Proposition \ref{P:radiation} above. If $t_{j,n}=0$ for all $n$, then by the definition of $R_n$ and the assumptions $\rho_n\geq R_n+|s_n|$, there exists a solution $V^j_{L}$ of \eqref{FW} such that 
$$\lim_{n\to\infty} \int_{|x|>\rho_n} \left|\nabla_{t,x}U^j(s_n-t_{j,n},x)-\nabla_{t,x}V^j_L(s_n-t_{j,n},x)\right|^2dx=0,$$
and the claim follows again from Proposition \ref{P:radiation}. Letting $h^k=r^{3/2} \partial_tU^k(T_k,r)\in L^2(0,\infty)$, we see that 
$$ \int_{\rho_n}^{\infty} g^j(r-s_n+t_{j,n})\frac{1}{\lambda_{k,n}^2} \partial_tU^k\left(T_k,\frac{r}{\lambda_{k,n}}\right)r^{3/2}dr= \int_{\rho_n}^{\infty} g^j(r-s_n+t_{j,n})h^k\left(\frac{r}{\lambda_{k,n}}\right)dr,$$
which tends to $0$ as $n$ tends to infinity because $\lim_ns_n-t_{j,n}=\infty$. Since the term in the second line of \eqref{Ps20} can be treated in  the exact same way, we obtain $\lim_{n\to\infty}\eps_n=0$.

The proof in the cases $T_j=-\infty$, $T_k\in \RR$, and $T_j\in \RR$, $T_k=\pm\infty$ are very close and we omit them.

\smallskip

\noindent\emph{Case 3. $T_j=+\infty$ and $T_k=+\infty$.} Similarly to Case $2$, there exist $g^j \in L^2(\RR)$ and $g^k\in L^2(\RR)$ such that
$$\eps_n=o_n(1)+2\int_{\rho_n}^{\infty} g^j(r-s_n+t_{j,n})\frac{1}{\lambda_{k,n}^2}g^k\left( \frac{r-s_n+t_{k,n}}{\lambda_{k,n}} \right)dr,$$
and the fact that $\lim_{n\to\infty}\eps_n=0$ follows easily from the pseudo-orthogonality property \eqref{psdo_orth}.

The proof in the case $T_j=-\infty$, $T_k=-\infty$ is the same and we omit it.

\smallskip

\noindent\emph{Case 4. $T_j=+\infty$ and $T_k=-\infty$.} 
Using Proposition \ref{P:radiation} as in cases $2$ and $3$, we obtain that the contributions of the integrals of $\partial_tU^j_n\partial_tU^k_n$ and $\partial_rU^j_n\partial_rU^k_n$ in the definition of $\eps_n$ are opposite, and thus
$$\lim_{n\to\infty}\eps_n=0.$$

The proof in the case $T_j=-\infty$, $T_k=+\infty$ is exactly the same. 

\smallskip

\noindent\emph{Pseudo-orthogonality between a profile and the dispersive remainder.}

We next prove \eqref{Ps11}.
We assume again that $T_j=\lim_{n\to\infty} s_n-t_{j,n}$ exists in $\RR\cup\{\pm\infty\}$.

\smallskip

\noindent\emph{Case 1. $T_j\in \RR$.}
Then 
$$\eps_n=o_n(1)+\int_{|x|>\rho_n} \nabla_{t,x}U^j(T_j,x)\nabla_{t,x}w^J_n(s_n,x)dx.$$
If $\lim_{n\to\infty}\rho_n=+\infty$, we deduce immediately that $\lim_{n\to\infty}\eps_n=0$. If not, we can assume $\lim_{n\to\infty}\rho_n=\rho_{\infty}\in [0,\infty)$, and we obtain
$$\eps_n=o_n(1)+\int_{|x|>\rho_{\infty}} \nabla_{t,x}U^j(T_j,x)\nabla_{t,x}w^J_n(s_n,x)dx.$$
By \eqref{Ps10}, we can take $J$ arbitrarily large. By the definitions of the profiles, and \eqref{wlim_w} we obtain, for $J$ large enough,
\begin{equation}
\label{wlim_w'}
\left( w_n^J\left(s_{n}\right),\partial _tw_n^J\left(s_n\right)\right) \xrightharpoonup[n\to\infty]{} 0 \text{ in }\HHH.
\end{equation} 
As a consequence, $\lim_n\eps_n=0$.

\smallskip

\noindent\emph{Case 2. $T_j=+\infty$.} Arguing as in Case 2 in the proof of \eqref{Ps10}, we obtain that there exists $g^j\in L^2(\RR)$ such that 
\begin{equation*}
 \eps_n=o_n(1)+\int_{\rho_n}^{\infty} g^j(r+t_{j,n}-s_n)\Big(\partial_tw_n^J(s_n,r)-\partial_rw_n^J(s_n,r)\Big)r^{3/2}dr.
\end{equation*}

If $\lim_{n}\rho_n+t_{j,n}-s_n=+\infty$, we obtain right away that $\eps_n=o_n(1)$. 

If $\lim_{n}\rho_n+t_{j,n}-s_n=-\infty$, then 
$$ \int_{0}^{\rho_n}\left|g^j(r+t_{j,n}-s_n)\right|^2dr\leq \int_{-\infty}^{\rho_n+t_{j,n}-s_n} |g(\eta)|^2\underset{n\to\infty}{\longrightarrow}0,$$
and thus
\begin{multline*}
 \eps_n=o_n(1)+\int_{0}^{\infty} g^j(r+t_{j,n}-s_n)\Big(\partial_tw_n^J(s_n,r)-\partial_rw_n^J(s_n,r)\Big)r^{3/2}dr\\
 =o_n(1)+\int \nabla_{t,x}V^j_L(s_n-t_{j,n})\cdot \nabla_{t,x}w_n^J(s_n,x)dx,
\end{multline*}
where $V^j_L$ is the radial solution of the linear wave equation such that 
$$\lim_{t\to\infty}  \int_0^{\infty}|r^{3/2} \partial_tV^j_L(t,r)-g^j(r-t)|^2dr+\int_0^{\infty}|r^{3/2} \partial_rV^j_L(t,r)+g^j(r-t)|^2dr=0.$$
By conservation of the energy for the free wave equation \eqref{FW},
\begin{equation*}
\eps_n=o_n(1)+\int \nabla_{t,x}V^j_L(0)\cdot \nabla_{t,x}w_n^J(t_{j,n},x)dx,
\end{equation*}
and thus by \eqref{wlim_w}, $\eps_n=o_n(1)$. 

Finally, if $\lim_{n}\rho_n+t_{j,n}-s_n=c\in \RR$, we see that 
\begin{multline*}
 \eps_n=o_n(1)+\int_{0}^{\infty} g^j_c(r+t_{j,n}-s_n)\Big(\partial_tw_n^J(s_n,r)-\partial_rw_n^J(s_n,r)\Big)r^{3/2}dr\\
 =o_n(1)+\int \nabla_{t,x}V^j_L(s_n-t_{j,n})\cdot \nabla_{t,x}w_n^J(s_n,x)dx,
\end{multline*}
where $g_c(\eta)=\indic_{\eta>c}g(\eta)$, and $V^j_L$ is the radial solution of the linear wave equation such that 
$$\lim_{t\to\infty}  \int_0^{\infty}|r^{3/2} \partial_tV^j_L(t,r)-g^j_c(r-t)|^2dr+\int_0^{\infty}|r^{3/2} \partial_rV^j_L(t,r)+g^j_c(r-t)|^2dr=0.$$
The same proof as before yields $\eps_n=o_n(1)$. 

Since the case $T_j=-\infty$ can be treated exactly in the same way, the proof of \eqref{Ps11} is complete.
\subsection{Pythagorean expansion of the energy for wave maps}
\label{AA:psdo_orthoWM}
In this subsection we prove \eqref{PytWM}. In view of Claim \ref{Cl:sin}, it is sufficient to prove
\begin{gather}
 \label{AA20}
 j\neq k\Longrightarrow \lim_{n\to\infty}\int_0^{\infty} \partial_{r}\Psi_{0,n}^j\partial_r\Psi_{0,n}^krdr+\int_0^{\infty} \Psi_{1,n}^j\Psi_{1,n}^krdr=0\\
 \label{AA21}
 j\neq k\Longrightarrow \lim_{n\to\infty}\int_0^{\infty}\left|\sin \Psi_{0,n}^j\sin \Psi_{0,n}^k\right|\frac{dr}{r}=0\\
 \label{AA22}
 j\leq J\Longrightarrow \lim_{n\to\infty}\int_0^{\infty} \partial_{r}\Psi_{0,n}^j\partial_r\omega_{0,n}^Jrdr+\int_0^{\infty} \Psi_{1,n}^j\omega_{1,n}^Jrdr=0\\
 \label{AA23}
 j\leq J\Longrightarrow \lim_{n\to\infty}\int_0^{\infty}\left|\sin \Psi_{0,n}^j\sin \omega_{0,n}^J\right|\frac{dr}{r}=0.
 \end{gather}
Recall that if $j\in \JJJ_L$, i.e. $\lim_{n\to\infty} -t_{j,n}/\lambda_{j,n}\in \{\pm\infty\}$, then 
$$ \left(\Psi^j_{0,n}(r),\Psi^j_{1,n}(r)\right)=\left(\Psi^j_L\left( \frac{-t_{j,n}}{\lambda_{j,n}},\frac{r}{\lambda_{j,n}} \right),\frac{1}{\lambda_{j,n}}\partial_t \Psi_L^j\left( \frac{-t_{j,n}}{\lambda_{j,n}},\frac{r}{\lambda_{j,n}} \right)\right),$$
where $\Psi^j_L$ is a solution of the linear equation \eqref{LWM} with initial data in $\Hbf$. Thus $U^j=\frac{1}{r}\Psi^j$ is a radial solution of the $1+4$ dimensional wave equation. By standard dispersive estimates for linear wave equations,
\begin{equation}
 \label{AA24}
 \lim_{n\to\infty} \int_0^{\infty} \left|\Psi_{0,n}^j(r)\right|^2\frac{dr}{r}=\lim_{n\to\infty}\int_0^{\infty}\left| U^j\left( \frac{-t_{j,n}}{\lambda_{j,n}},r \right)\right|^2rdr=0.
\end{equation} 
This yields \eqref{AA21} when $j\in \JJJ_L$ or $k\in \JJJ_k$, and \eqref{AA23} when $j\in \JJJ_L$. 

Next, we assume $j\in \JJJ_L$ and $k\in \JJJ_L$ and see that by \eqref{AA24} letting $U^j_n(t,r)=\frac{1}{\lambda_{j,n}}U^j\left(\frac{t-t_{j,n}}{\lambda_{j,n}},\frac{r}{\lambda_{j,n}} \right)$ and defining similarly $U^k_n$, we have
\begin{multline*}
\lim_{n\to\infty}\int_0^{\infty} \partial_{r}\Psi_{0,n}^j\partial_r\Psi_{0,n}^krdr+\int_0^{\infty} \Psi_{1,n}^j\Psi_{1,n}^krdr\\
= \int_{0}^{\infty} \partial_rU_n^j(0,r)\partial_tU_n^j(0,r)r^3dr+  \int_{0}^{\infty} \partial_rU_n^j(0,r)\partial_rU_n^j(0,r)r^3dr+o_n(1),
\end{multline*}
and \eqref{AA20} follows from \eqref{Ps10}. 

It remains to prove \eqref{AA20} when $j\in \JJJ_C$ or $k\in \JJJ_C$, \eqref{AA21} when $j\in \JJJ_C$ and $k\in \JJJ_C$, and \eqref{AA22}, \eqref{AA23} when $j\in \JJJ_C$. 

\smallskip

\noindent\emph{Proof of \eqref{AA20} when $j\in \JJJ_C$, $k\in \JJJ_L$.} Assume to fix ideas $\lim_{n\to\infty} -t_{j,n}/\lambda_{j,n}=+\infty$. By the asymptotic formulas \eqref{RT10} and \eqref{RT11}, there exists $G^k\in L^2(\RR)$ such that 
$$\lim_{t\to+\infty} \int_0^{\infty}\left| r^{3/2}\partial_tU^k(t,r)+G^k(r-t)\right|dr+\int_0^{\infty} \left| r^{3/2}\partial_rU^k(t,r)-G^k(r-t)\right|dr.$$
Using also that $\lim_{t\to\infty}\int_0^{\infty} |U^k(t,r)|^2r dr=0$, we deduce
$$\lim_{t\to+\infty} \int_0^{\infty}\left| r^{1/2}\partial_t\Psi^k(t,r)+G^k(r-t)\right|dr+\int_0^{\infty} \left| r^{1/2}\partial_r\Psi^k(t,r)-G^k(r-t)\right|dr.$$
Thus 
\begin{multline*}
\left|\int_{0}^{\infty}\partial_r\Psi_{0,n}^j(r)\partial_r\Psi_{0,n}^k(r)rdr\right|\\
=\left| \int_{0}^{\infty}
\frac{1}{\lambda_{j,n}}\partial_r\Psi_{0}^j\left(\frac{r}{\lambda_{j,n}}\right)G^k\left(\frac{t_{k,n}+r}{\lambda_{k,n}}\right)\frac{r^{1/2}}{\lambda_{k,n}^{1/2}}dr \right|\\
=\left| \int_{0}^{\infty}\partial_r\Psi_{0}^j\left(s\right)\frac{\lambda_{k,n}^{1/2}}{\lambda_{j,n}^{1/2}} G^k\left(\frac{t_{k,n}+\lambda_{j,n}s}{\lambda_{k,n}}\right)s^{1/2}ds \right|,
\end{multline*}
which goes to $0$ when $n$ goes to infinity (this is immediate when $G^k$ is continuous and compactly supported, and follows by density in the general case). By a very similar computation,
$$ \lim_{n\to\infty}\int_{0}^{\infty}\Psi_{1,n}^j(r)\Psi_{1,n}^k(r)rdr=0,$$
which concludes the proof of \eqref{AA20} in this case.

\smallskip

\noindent\emph{Proof of \eqref{AA20} and \eqref{AA21} when $j,k\in \JJJ_C$.}
We have 
\begin{multline*}
\int_0^{\infty} \partial_r\Psi_{0,n}^j\partial_r\Psi_{0,n}^k
rdr=\int_0^{\infty} \frac{1}{\lambda_{j,n}}\partial_r\Psi_0^j\left( \frac{r}{\lambda_{j,n}} \right)\frac{1}{\lambda_{k,n}}\partial_r\Psi_0^k\left( \frac{r}{\lambda_{k,n}} \right)rdr\\
=\int_0^{\mu_n}\ldots+\int_{\mu_n}^{\infty}\ldots,
\end{multline*}
where $\mu_n=\sqrt{\lambda_{j,n}\lambda_{k,n}}$. Using Cauchy-Schwarz on each of the integrals, we see that it goes to $0$ when $n$ goes to $\infty$, since $\lim_n\frac{\lambda_{j,n}}{\lambda_{k,n}}+\frac{\lambda_{k,n}}{\lambda_{j,n}}=\infty$. By the same proof for the term involving $\Psi_{1,n}^j$ and $\Psi_{1,n}^k$, we obtain \eqref{AA20}.

The proof of \eqref{AA21} when $j,k \in \JJJ_C$ is the same. 

\smallskip

\noindent\emph{Proof of \eqref{AA23} when $j\in \JJJ_C$.}
We have 
$$\int_0^{\infty}\left|\sin \Psi_{0,n}^j(r) \sin \omega_{0,n}^J(r)\right|\frac{dr}{r}=\int_0^{\infty}\left|\sin \Psi_{0}^j\left( r \right) \sin \omega_{0,n}^J(\lambda_{j,n}r)\right|\frac{dr}{r}. $$
Let $\eps>0$. Then
\begin{multline*}
\int_{\eps}^{\eps^{-1}}\left|\sin \Psi_{0}^j(r) \sin \omega_{0,n}^J(\lambda_{j,n}r)\right|\frac{dr}{r}\\
\leq \left(\int_{\eps}^{\eps^{-1}}\left|\sin \Psi_{0}^j(r)\right|^{3/2}\frac{dr}{r^{3/2}}\right)^{2/3}\left( \int_{\eps}^{\eps^{-1}} \left|\omega_{0,n}^J(\lambda_{j,n}r)\right|^3dr \right)^{1/3}. 
\end{multline*}
Let $w_n^J(r)=\frac{1}{r}\omega_n^J(\lambda_{j,n}r)$. Then \eqref{wnJWM} implies that $\lambda_{j,n}w_n^J(\lambda_{j,n}\cdot)$, considered as a radial function on $\RR^4$, converges weakly to $0$ in $\hdot(\RR^4)$. This implies that $\lambda_{j,n}w_n^J(\lambda_{j,n}\cdot)$ converges strongly in $L^3_{\loc}(\RR^4)$, and thus
$$\lim_{n\to\infty} \int_{\eps}^{\eps^{-1}} \left|\omega_{0,n}^J(\lambda_{j,n}r)\right|^3dr =0.$$
As a consequence, for all $\eps$
$$\lim_{n\to\infty}\int_{\eps}^{\eps^{-1}}\left|\sin \Psi_{0}^j(r) \sin \omega_{0,n}^J(\lambda_{j,n}r)\right|\frac{dr}{r}=0,$$
and  \eqref{AA23} follows since $\int_{0}^{\infty} \sin^2\Psi_{0}^j(r)\frac{dr}{r}$ is finite and $\int_0^{\infty}\sin^2\omega_{0,n}^J(\lambda_{j,n}r)\frac{dr}{r}$ is uniformly bounded. 
\section{Boundedness of integral operators on $L^pL^q$ spaces}
\label{A:anal_harm}
In this appendix, we consider, as in the core of the article, functions defined for $t\in \RR$, $x\in \RR^4$ that are radial in the space variable. As before we use the notation
$$ \|u\|_{(L^pL^q)(R)}:=\left\| \indic_{\{|x|>R+|t|\}} u\right\|_{L^p(\RR,L^q(\RR^4))}.$$
\begin{lemma}
 \label{L:Carlos2}
 Let $R\geq 0$.
 For $w\in (L^2L^8)(R)$, define
 $$ (Aw)(t,r)=\frac{1}{r^2}\int_{0}^t w(\tau,r)d\tau.$$
 Then $Aw \in (L^2L^{8/3})(R)$ and 
 $$ \left\|Aw\right\|_{(L^2L^{8/3})(R)} \lesssim \|w\|_{(L^2L^8)(R)},$$
 where the implicit constant is independent of $R\geq 0$.
\end{lemma}
\begin{proof}
 \noindent\emph{Step 1.} We first claim that if $h\in (L^2L^8)(R)$, then $\frac{|t|}{r^2}h\in (L^2L^{8/3})(R)$. Indeed 
 \begin{multline*}
  \left( \int_{R+|t|}^{\infty} |h(t,r)|^{8/3} \left( \frac{|t|}{r^2} \right)^{8/3}r^3dr \right)^{3/8}\\
  \lesssim |t|\left( \int_{R+|t|} h(t,r)^8r^3dr \right)^{1/8} \left( \int_{R+|t|}^{\infty} \frac{r^3}{r^8}dr \right)^{1/4}\\
  \lesssim \frac{|t|}{R+|t|}\left(\int_{R+|t|}^{\infty} h(t,r)^8r^3dr\right)^{1/8}.
 \end{multline*}
Now $\frac{|t|}{R+|t|}\leq 1$ and $t\mapsto \left(\int_{R+|t|}^{\infty} h(t,r)^8r^3dr\right)^{1/8}$ is in $L^2$ in time, concluding this step.

\medskip

\noindent\emph{Step 2.} In this step, we consider a positive function $f$ in $L^2(\RR,L^8(\RR^4))$, and define 
$$M(f)(t,x)=\sup_{I\ni t} \frac{1}{|I|} \int_{I}f(\tau,x)d\tau,$$
where the supremum is taken over all finite intervals $I$ that contain $t$, and we prove
\begin{equation}
\label{L2L8}
\left\|M(f)\right\|_{L^2L^8}\lesssim \|f\|_{L^2L^8}. 
\end{equation} 
Indeed, if $f\in L^2L^2$ then 
\begin{equation}
 \label{L2L2}
\|M(f)\|_{L^2L^2}\lesssim \|f\|_{L^2L^2}
 \end{equation} 
by Fubini and Hardy-Littlewood maximal theorem in the time variable. Furthermore, if $f\in L^2L^{\infty}$, one can prove:
\begin{equation}
 \label{L2Linfty}
\|M(f)\|_{L^2L^{\infty}}\lesssim \|f\|_{L^2L^{\infty}}.
 \end{equation} 
Indeed, 
$$M(f)(t,x)\lesssim \sup_{I\ni t}\frac{1}{|I|} \int_{I}\|f(\tau)\|_{L^{\infty}_x}d\tau$$
and thus
$$\|Mf(t)\|_{L^{\infty}_x}\leq \sup_{I\ni t}\frac{1}{|I|} \int_{I}\|f(\tau)\|_{L^{\infty}_x}d\tau.$$
Using Hardy-Littlewood maximum theorem in the $t$ variable, we obtain \eqref{L2Linfty}. Interpolation gives \eqref{L2L8}. 

\medskip

\noindent\emph{Step 3.} Let $w\in (L^2L^8)(R)$. Let 
$$ f(t,r)=|w(t,r)|\indic_{\{r>t+R\}}.$$
Then, for $r>|t|+R$, we have 
$$\frac{1}{|t|}\left|\int_{0}^{t}w(\tau,r)d\tau\right|\leq \frac{1}{|t|}\int_{-|t|}^{|t|}\indic_{r>\tau+R}|w(\tau,r)|d\tau\leq 2M(f)(t,r),$$
since $r>|t|+R\geq |\tau|+R$ for $0\leq |\tau|\leq |t|$. As a consequence, for $r>|t|+R$, 
$$\left|A(w)(t,r)\right|\lesssim \frac{|t|}{r^2}M(f)(t,r).$$
By Step 1 with $h(t,r)=M(f)(t,r)$, then Step 2, we have 
$$\left\|A(w)\right\|_{(L^2L^{8/3})(R)}\lesssim \|M(f)\|_{(L^2L^8)(R)}\lesssim \|w\|_{(L^2L^8)(R)}.$$
\end{proof}
\begin{lemma}
 \label{L:Carlos1}
 Let $R\geq 1$, $\lambda>1/R$. For $w\in (L^3L^6)(R)$, we define
 $$(Bw)(t,r)=
\frac{t}{r^4\log(r\lambda)^{1/2}}\int_0^{t} w(\tau,r)d\tau.$$
Then $B$ is bounded from $(L^3L^6)(R)$ to $(L^1L^2)(R)$ and 
$$\|Bw\|_{(L^1L^2)(R)}\lesssim \frac{1}{\log(\lambda R)^{1/3}}\|w\|_{(L^3L^6)(R)}$$
 \end{lemma}
\begin{proof}
 We just work with $t>0$. By Minkowski in time,
 $$\left( \int_{R+t}^{\infty} \left( \int_0^t \frac{w(\tau,r) t}{r^4\log (\lambda r)}d\tau \right)^2r^3dr \right)^{1/2}\leq \int_{0}^t \left\|w(\tau,r)\frac{t}{r^4\log (\lambda r)}\right\|_{L^2(R+t)}d\tau.$$
Furthermore, by H\"older
\begin{multline*}
\left\|w(\tau,r)\frac{t}{r^4\log (\lambda r)}\right\|_{L^2(R+t)}=t\left( \int_{R+t}^{\infty} |w(\tau,r)|^2\frac{r^3}{r^8\log^2(\lambda r)}dr \right)^{1/2}\\
\leq t\left( \int_{R+t}^{\infty}|w(\tau,r)|^6r^3dr \right)^{1/6} \left( \int_{R+t}^{\infty}\frac{r^3}{r^{12}\log^3(\lambda r)}dr \right)^{1/3}\\
\lesssim \frac{t}{(R+t)^{8/3}\log(\lambda(R+t))} \int_0^t \left( \int_{R+t}^{\infty} |w(\tau,r)|^6r^3dr \right)^{1/6}d\tau
\end{multline*}
Integrating, we obtain 
\begin{multline*}
 \left\|\indic_{t>0}B(w)\right\|_{(L^1L^2)(R)}\\
 \lesssim \int_0^{\infty} \frac{t}{(R+t)^{8/3}\log(\lambda(R+t))} \int_0^t \left( \int_{R+t}^{\infty} |w(\tau,r)|^6r^3dr \right)^{1/6}d\tau dt
\end{multline*}
By Fubini,
\begin{multline*}
 \left\|\indic_{t>0}B(w)\right\|_{(L^1L^2)(R)}\\
 \lesssim 
\int_0^{\infty} \left( \int_{R+\tau}^{\infty} |w(\tau,r)|^6r^3dr \right)^{1/6} \int_{\tau}^{\infty} \frac{t}{(R+t)^{8/3}\log(\lambda(R+t))} dt\,d\tau\\
\lesssim \int_{0}^{\infty} \left( \int_{R+\tau}^{\infty} |w(\tau,r)|^6r^3dr \right)^{1/6} \int_{\tau}^{\infty}\frac{dt}{(R+t)^{5/3}\log(\lambda(R+t))} d\tau\\
\lesssim\int_{0}^{\infty} \left( \int_{R+\tau}^{\infty} |w(\tau,r)|^6r^3dr \right)^{1/6} \frac{d\tau}{(R+\tau)^{2/3}\log(\lambda(R+\tau))}\\
\lesssim \left( \int_0^{\infty} \left( \int_{R+\tau}^{\infty} |w(\tau,r)|^{6}r^3dr \right)^{1/2}d\tau \right)^{1/3}\left( \int_0^{\infty}\frac{d\tau}{(R+\tau)\log(\lambda(R+\tau))^{3/2}}d\tau \right)^{2/3},
\end{multline*}
where at the last line we have used H\"older's inequality. The desired conclusion follows easily.
 \end{proof}  
 
\bibliographystyle{acm}
\bibliography{/home/duyckaerts/ownCloud2/Recherche/toto} 

\end{document}